\documentclass[10pt]{article}
\usepackage[margin=1.5in]{geometry}

\usepackage{amsmath}

\usepackage{multirow} 
\usepackage{booktabs}
\usepackage{bbold}
\RequirePackage{amsthm,amsmath,amsfonts,amssymb}
\RequirePackage[numbers]{natbib}
\RequirePackage[colorlinks,citecolor=blue,urlcolor=blue]{hyperref}
\RequirePackage{graphicx}

\newtheorem{theorem}{Theorem}[section]
\newtheorem{corollary}[theorem]{Corollary}
\newtheorem{lemma}[theorem]{Lemma}
\newtheorem{definition}[theorem]{Definition}

\theoremstyle{remark}
\newcommand{\khatxn}{K_{X_n,J_n}}
\usepackage[title]{appendix}

\begin{document}

\title{A Sieve Stochastic Gradient Descent Estimator for Online Nonparametric Regression in Sobolev ellipsoids}

\author{Tianyu Zhang \and Noah Simon}

\maketitle

\begin{abstract}
 The goal of regression is to recover an unknown underlying function that best links a set of predictors to an outcome from noisy observations. In nonparametric regression, one assumes that the regression function belongs to a pre-specified infinite-dimensional function space (the hypothesis space). In the online setting, when the observations come in a stream, it is computationally-preferable to iteratively update an estimate rather than refitting an entire model repeatedly. Inspired by nonparametric sieve estimation and stochastic approximation methods, we propose a sieve stochastic gradient descent estimator (Sieve-SGD) when the hypothesis space is a Sobolev ellipsoid. We show that Sieve-SGD has rate-optimal mean squared error (MSE) under a set of simple and direct conditions. The proposed estimator can be constructed with a low computational (time and space) expense: We also formally show that Sieve-SGD requires almost minimal memory usage among all statistically rate-optimal estimators. 
\end{abstract}

\section{Introduction}

It is commonly of interest to understand the association between a number of features (or predictors) and a quantitative outcome. To this end, one often estimates an underlying regression function that best links these two quantities from noisy observations. More formally, suppose we obtain $n$ samples, $\left(X_i,Y_i\right)$, where $X_i\in\mathcal{X}\subset \mathbb{R}^p$ denotes a $p$-vector of features from the $i$-th sample we observe, and $Y_i\in\mathbb{R}$ denotes the $i$-th outcome. Further suppose that each pair $\left(X_i,Y_i\right)$ is independently and identically distributed (i.i.d.) from a fixed but unknown distribution $\rho$ over $\mathcal{X}\times \mathbb{R}\subset \mathbb{R}^p\times \mathbb{R}$. A common target of estimation is the conditional mean $f_{\rho}(X) := E_{\rho}[Y|X]$. Under extremely mild conditions, this conditional mean is the optimal function for predicting $Y$ from $X$ with regard to mean squared error. More formally,
\begin{equation}
     f_{\rho} = \operatorname{argmin}_{f\in L^2_{\rho_X}}  E_{\rho}\left[\left(Y-f(X)\right)^2\right],
 \end{equation}
where $L^2_{\rho_X}$ is the collection of all $\rho_X$-mean square integrable functions and $\rho_X$ is the marginal distribution of $X$. Our goal is to estimate $f_{\rho}$ from our collection of observed data.

In order to make a tractable estimation of $f_{\rho}$ from data, we need to make additional assumptions on its smoothness/structure: The entire $L^2_{\rho_X}$ space is too big to search within \cite{belkin2018reconciling,liang2018just}. We often formally assume that $f_{\rho}$ belongs to a pre-specified function space $\mathcal{F}\subsetneqq L^2_{\rho_X}$. This $\mathcal{F}$ is known as the \emph{hypothesis space} of the regression problem. 

If $\mathcal{F}$ can be indexed by a finite-dimensional parameter set $\Theta\subset\mathbb{R}^d$, $d\in\mathbb{N}^+$, we refer to $\mathcal{F}$ as a \emph{parametric function space} or a \emph{parametric class}. One common parametric class is $\mathcal{F} = \{X^{\top}\beta\,|\,\beta\in\mathbb{R}^d\}$, the set of all linear functions of $X$. Parametric classes can impose overly restrictive assumptions on the form of the regression function that may not be realistic in practice. As such, it has become popular to assume less restrictive structure: It is common to define the hypothesis space based on constraints on derivatives, monotonicity, or other shape-related properties. Such an $\mathcal{F}$ is most naturally written as an infinite-dimensional subset of $L^2_{\rho_X}$. Commonly used examples of $\mathcal{F}$ in the statistics community include H\"older balls, Sobolev spaces \cite{gyorfi2006distribution,nemirovski2000topics,wahba1990spline}, reproducing kernel Hilbert spaces (RKHS) \cite{berlinet2011reproducing,christmann2008support} and Besov spaces \cite{hardle2012wavelets}. These are known as \emph{nonparametric function spaces}, as they cannot naturally be parametrized using a finite length vector. The Sobolev ellipsoid, in particular, is a simple and useful abstraction of many important function spaces \cite{wahba1990spline}. Therefore, we focus on them exclusively as the hypothesis spaces in this paper.

In this paper, we propose an estimator for \emph{online} nonparametric regression. In online estimation, the data are seen sequentially, one sample at a time. After each sample is observed, our estimate of $f_{\rho}$ must be updated, as a prediction may be required at any point in time before all the available samples are processed. In an online problem with $n$ observations, we must sequentially construct $n$ estimates. This is in contrast to the classical batch learning setting where we collect all the data initially and perform estimation only once. In the online setting, it is generally computationally infeasible to repeatedly refit the whole model from scratch for each new observation. Thus, online algorithms are generally carefully developed to permit more tractable \emph{updates} after each new observation \citep{duchi2014multiple,kushner2003stochastic}.

An ideal estimator in online settings should be: i) statistically rate-optimal, i.e. achieve the minimax-rate for estimating $f_{\rho}$ over $\mathcal{F}$; and ii) computationally inexpensive to construct/update. In this paper, we present such an online nonparametric estimator for use when the hypothesis space is a Sobolev ellipsoid, which we term the \emph{Sieve Stochastic Gradient Descent estimator (Sieve-SGD)}. This method can be thought of as an online version of the classical projection estimator \cite{introtononpara}, where the latter is a specific example of sieve estimators \cite{geer2000empirical,shen1997methods}. We use the more general term ``sieve" in naming our method to emphasize its nonparametric nature and avoid confusion with the term ``stochastic projection" \cite{vempala2005random}. We will show that Sieve-SGD can achieve rate-optimal estimation error for $\mathcal{F}$ a Sobolev ellipsoid and asymptotically uses minimal memory (up to a log factor) among all rate-optimal estimators. In addition, our estimator has the same computational cost (up to a constant) as merely examining each allocated memory location every time a new sample $X_i$ is collected. This intimates that in scenarios when our estimator has near optimal space complexity, it may also have near optimal time complexity (though formal investigation of lower bounds for time complexity in this problem is beyond the scope of the current manuscript).

The structure of our paper continues as follows. In Section~\ref{section:batch learning} we briefly cover classical results for batch, nonparametric estimation in Sobolev ellipsoids, focusing on projection estimators (which motivate our method). In Section~\ref{section:online learning} we return to the online setting and explore intuition for how one might combine projection estimation and stochastic gradient descent (SGD) \cite{bottou2010large}. The latter is a well-studied method that has been applied fruitfully to online parametric regression problems. This will help motivate our proposed method, which, as we will see, can be thought of as an SGD estimator with a parameter space of increasing dimension. In Section~\ref{section: related work} we discuss existing nonparametric SGD estimators, and identify some notable drawbacks of current methods. In Section~\ref{section:Sieve-SGD}, we introduce the formal construction of Sieve-SGD and analyze its computational expense. From there, we show that our estimator has a dramatically smaller ``dimension" than existing methods and discuss how this helps to reduce the computational expense. In Section~\ref{section:theory}, we give a theoretical analysis of the statistical properties of Sieve-SGD. In constructing our estimator, we need to decide how quickly to grow the dimension it projects onto. Under minimal assumptions, we characterize the required growth rate and learning rate for our estimator to be statistically and computationally (near) optimal. We will also investigate under what conditions such an optimality result is adaptive/insensitive to our choice of the ``dimension-specific learning rate". Section~\ref{section:simulation} provides simulation studies to illustrate our theoretical results. Finally, in Section~\ref{section:discussion}, we have some further discussion of Sieve-SGD and possible future research directions.
 
 \textit{Notation:} In this paper, we use $C$ to denote a generic constant that does not depend on sample size $n$ (The value of $C$ may be different in different parts of the manuscript). Additionally the notation $a_n = \Theta(b_n)$ means $a_n = O(b_n)$ and $b_n = O(a_n)$. The function $\lfloor x\rfloor$ maps $x$ to the largest integer smaller than $x$. For a vector $x\in\mathbb{R}^p$, $x^{(i)}$ is the $i$-th component of $x$. The notation $x\vee y$ (resp. $x\wedge y$) is shorthand for $\max\{x,y\}$ (resp. $\min\{x,y\}$). The $\|\cdot\|_{\infty}$ norm of a continuous function $f$ is defined as $\|f\|_{\infty} = \sup_{x\in\mathcal{X}}|f(x)|$, where $\mathcal{X}$ is the domain of $f$.

\section{Batch Learning and the Projection Estimator}
\label{section:batch learning}
In this section we consider estimation in the classical batch setting where our estimate is constructed once after all $n$ samples are observed. We will begin by formally introducing a Sobolev ellipsoid: This is the hypothesis space we will use throughout this manuscript. This will be followed by presenting the classical projection estimator \cite{introtononpara}.

Consider a user-specified measure $\nu$ whose support contains $\mathcal{X}$, and the corresponding square-integrable function space $L^2_{\nu}$.  In many interesting cases $\nu$ can be simply taken as Lebesgue measure over $\mathcal{X}$ but it is not necessary in the general form of our theory. To define a Sobolev ellipsoid in $L^2_{\nu}$, suppose we have a complete orthonormal basis $\{\psi_j, j=1,2,...\}\subset L^2_{\nu}$ of $L^2_{\nu}$ \cite{hernandez1996first}. This means
\begin{enumerate}
    \item[i)] For any $f\in L^2_{\nu}$, there exists a unique sequence $(\theta_j)_{j=1}^{\infty}\in \ell^2$ such that 
\begin{equation}
    \lim_{N\rightarrow\infty}\int \left|f(z)-\sum_{j=1}^{N}\theta_j\psi_j(z)\right|^2d\nu(z) = 0\quad\text{(completeness)}
\end{equation}
where $\ell^2$ is the space of square convergent series.
    \item[ii)] $\{\psi_j\}$ is an orthonormal system:
\begin{equation}
  \int \psi_i(z)\psi_j(z)d\nu(z) = \delta_{ij}\quad\text{(orthonormality)}  
\end{equation}
where $\delta_{ij}$ is the Kronecker delta.
\end{enumerate}

We define the \emph{Sobolev ellipsoid} $W(s,Q,\{\psi_j\})$ as:
\begin{equation}
W\left(s, Q, \{\psi_j\}\right) :=\left\{\left.f=\sum_{j=1}^{\infty} \theta_{j} \psi_{j}\ \right|\ \sum_{j=1}^{\infty}\left(\theta_{j} j^{s}\right)^{2} \leq Q^2\right\}
\end{equation}
We refer to $(\theta_j)_{j=1}^{\infty}$ as the (general) \textit{Fourier coefficients} of a function $f$. Throughout this manuscript, we assume the measure $\nu$, basis functions $\psi_j$ and the regularity parameter $s$ are all known. When it is clear which $\psi_j$ we are using, we will denote a Sobolev ellipsoid simply by $W(s,Q)$. We may also use the further simplified notation $W(s)$ because the diameter $Q$ usually plays a secondary role in our theoretical analysis and the proposed method is adaptive to it. Intuitively, by saying a function $f$ belongs to a Sobolev ellipsoid, we are requiring its coefficients $\{\theta_j\}$ to converge to zero faster than $j^{-(s+1/2)}$ (if not, the sum $\sum_{j=1}^{\infty}\left(\theta_{j} j^{s}\right)^{2}$ would diverge to infinity). The larger $s$ is, the faster the decay of $\theta_j$ will be, and thus the stronger our assumption is. 

Sobolev ellipsoids are popular spaces to study for two reasons: 1) They impose a useful structure for theories and computations, especially as a basic example of hypothesis spaces with finite metric entropy; and 2) Many natural spaces of regular functions are Sobolev ellipsoids. For example, if $\mathcal{X} = [0,1]$ with $\nu$ as Lebesgue measure, then for any $s>0$, the periodic Sobolev space 
\begin{equation}
  \mathcal{F} = \left\{f\in L^2_{\nu}\mid \int \left(f^{(s)}(x)\right)^2dx < Q^2, f^{(k)}(0)=f^{(k)}(1),k =0,1,..,s-1\right\}  
\end{equation}
can be written as a Sobolev ellipsoid, using an orthogonal basis of trigonometric functions \cite[Chapter~2]{wahba1990spline}. More generally, for any RKHS $\left(\mathcal{H},\langle \cdot,\cdot\rangle_{\mathcal{H}}\right)$, it is possible to find a set of $\psi_j$ such that $W(s,Q,\{\psi_j\}) = \{f\in\mathcal{H}\ |\ \|f\|_{\mathcal{H}}\leq Q\}$ , i.e. a ball in an RKHS is a Sobolev ellipsoid (see \cite{cucker2002mathematical,sun2005mercer}). 

In everything that follows we will assume that $f_{\rho}$, our target of estimation, lives in a known Sobolev ellipsoid $W\left(s, Q, \{\psi_j\}\right)$; with $\{\psi_j\}$ specified, and orthonormal w.r.t. a specified measure $\nu$ (not necessarily equal to $\rho_X$); and $s$ known (we allow $Q$ to be unknown).

The \emph{Projection Estimator} is a classical estimator naturally associated with a Sobolev ellipsoid. We can treat it as a special case of general sieve estimation \cite[Chapter~10]{geer2000empirical}: The estimates can be characterized by a sequence of finite dimensional linear spaces of increasing dimension (the dimension increases with sample size). For any given $f\in W(s,Q)$, the magnitude of its Fourier coefficients must asymptotically decrease with $j$ fast enough. Thus, it might be sensible to consider an estimator that discards the basis functions far into the tail. This is precisely what the projection estimator does. More formally, for a user-specified truncation level $J_n$, the projection estimator is given by 
\begin{equation}
    \hat{f}_{n,J_n} = \sum_{j=1}^{J_n}\hat\theta_j\psi_j
\end{equation}
where $\hat{\theta}=(\hat{\theta}_1,..,\hat{\theta}_{J_n})^{\top}$ is the solution of the least square problem:
\begin{equation}\label{NPLS}
\min_{\theta\in\mathbb{R}^{J_n}} \sum_{i=1}^n \left(Y_i - \sum_{j=1}^{J_n} \theta_j\psi_j(X_i)\right)^2
\end{equation}

It has been shown (e.g. \cite{introtononpara}, Theorem 1.9) that when we choose $J_n = \Theta( n^{\frac{1}{2s+1}})$, the projection estimator is a rate-optimal estimator over $W(s,Q)$, i.e.
\begin{equation}
    \limsup_{n\rightarrow\infty} \sup_{f_{\rho}\in W(s,Q)} E\left[\left\|\hat f_{n,J_n} - f_{\rho}\right\|_{2}^2\right] = O(n^{-\frac{2s}{2s+1}})
\end{equation}
This result is usually shown in the literature for $X_i$ equally spaced, or drawn from a uniform distribution. But in our theoretical analysis (Section~\ref{section:theory}), we allow $\rho_X$ to be a much more general distribution.

Sieve-SGD is inspired by this (batch) projection estimator. The key here is that the number of basis functions we need to use can be dramatically smaller than sample size, and their analytical forms do not depend on the data (usually reproducing kernel methods use basis functions ``centered" at the feature vectors $X_i$). This possibility has been rarely explored \cite{zhang2021online} by existing nonparametric online estimation research.

\section{Online Learning and Stochastic Approximation}
\label{section:online learning}

We now move to the online learning setting where observations are collected sequentially from a data stream, and an estimate of our function is required after each sample. Such an infinite data stream may really exist, for example, with simulated samples as in reinforcement learning. Or the stream may serve as an abstraction used with large-scale data sets where it is not favorable to handle all the samples at once. It is generally computationally prohibitive to use a method developed for the ``batch'' setting and completely refit it after each observation. Instead methods that iteratively update are preferred. For example, fitting a single projection estimator (solving \eqref{NPLS}) with $n$ observations using $J_n=n^{\frac{1}{2s+1}}$ requires computation of $\Theta(n^{1+\frac{2}{2s+1}})$. Refitting a projection estimator (from scratch) after each observation $i=1,\ldots, n$ with $J_i = \lfloor i^{\frac{1}{2s+1}}\rfloor$ would require an accumulated computation of $\sum_{i=1}^n i^{1+\frac{2}{2s+1}} = \Theta(n^{2+\frac{2}{2s+1}})$. This scales worse than quadratically in $n$. Our goal in the online nonparametric setting is to find a statistically rate-optimal estimator whose computation scales only slightly worse than linearly in $n$.

Online learning has been thoroughly studied for parametric $\mathcal{F}$. Many proposed methods are based on the concept of \emph{stochastic approximation} \cite{kushner2003stochastic}. One of the most popular methods in stochastic approximation is \emph{Stochastic Gradient Descent} (SGD) \cite{bottou2010large}. In the parametric setting, SGD gives a statistically rate-optimal estimator $\hat f_n$ whose population mean-square-error $E\|\hat f_n - f_{\rho}\|_{L^2_{\rho_X}}^2$ is of order $O(\frac{1}{n})$ \cite{babichev2018constant,bach2013non,frostig2015competing}. Both vanilla SGD and its variants have been applied to general convex loss functions and are shown to be statistically rate-optimal under mild conditions \cite{duchi2014multiple}.

\subsection{Parametric SGD}
\label{subsection:parametric sgd}To motivate stochastic optimization in the nonparametric setting, we first give more details on SGD for parametric classes. Here we consider a specific class of functions $\mathcal{F} = \{f = \sum_{j=1}^d \beta^{(j)} \psi_j,\beta\in\mathbb{R}^d\}$ for a set of pre-specified basis functions $\psi_j:\mathbb{R}^p~\rightarrow~\mathbb{R},j = 1,\ldots,d$. We use this example to illustrate the principle of (parametric) SGD. Solving $\operatorname{argmin}_{f\in\mathcal{F}}E\left[(Y-f(X))^2\right]$ reduces to solving
\begin{equation}
\label{OP}
\min_{\beta\in\mathbb{R}^d}\ell(\beta):= \min_{\beta\in\mathbb{R}^d}E\left[\left(Y-\sum_{j=1}^d \beta^{(j)}\psi_j(X)\right)^2\right]
\end{equation}
We assume the minimizer of $\ell(\beta)$ exists and denote it as $\beta^{*}$.

If we knew the true joint distribution $\rho$ of $(X,Y)$ (which never happens in practice), then equation~\eqref{OP} is just a numerical optimization problem which does not involve data. We could use gradient descent to solve it. The gradient of $\ell$ at any point $\beta$ is
\begin{equation}
\begin{aligned}
        \nabla \ell(\beta) = -2E\left[\left(Y-\sum_{j=1}^d \beta^{(j)}\psi_j(X)\right) \left(\psi_1(X),...,\psi_d(X)\right)^{\top}\right]
\end{aligned}
\end{equation}
Thus, the gradient descent updating rule one could use is:
\begin{equation}
\label{paraSGD}
\begin{aligned}
    \hat{\beta}_{0} &= 0\\
    \hat{\beta}_n &= \hat{\beta}_{n-1} - \gamma_n \nabla \ell(\hat\beta_{n-1})
\end{aligned}
\end{equation}
where $\{\gamma_n\}$ is a pre-specified sequence of step-sizes (or learning rate) and $\hat\beta_n\in\mathbb{R}^d$ is the sequence of approximations of $\beta^{*}$. 

In practice, we do not know the joint distribution $\rho$: we must use data to estimate $\beta^*$. In the framework of SGD, this is done by using the data to get unbiased estimates of the gradients and substituting the estimates into our updating rule \eqref{paraSGD}. In particular we note that $\widehat{\nabla \ell(\beta)}:=-2 \left(Y_i-\sum_{j=1}^d \beta^{(j)}\psi_j(X_i)\right)(\psi_1(X_i),...,\psi_d(X_i))^{\top}$ is an unbiased estimator of the gradient $\nabla \ell(\beta)$ based on one sample. This results in the SGD updating rule.

\begin{equation}
\label{eq:paraSGDupdate}
\begin{aligned}
    \hat{\beta}_{0} &= 0\\
    \hat{\beta}_n &= \hat{\beta}_{n-1} - \gamma_n \widehat{\nabla \ell(\hat\beta_{n-1})}\\
    & =  \hat{\beta}_{n-1} + 2\gamma_n\left(Y_{n}-\sum_{j=1}^d \hat\beta_{n-1}^{(j)}\psi_{j}(X_{n})\right)(\psi_1(X_n),...,\psi_d(X_n))^{\top}
\end{aligned}
\end{equation}
So our estimator $\hat f_n$ of $f_{\rho}$ has the following functional update rule, derived from \eqref{eq:paraSGDupdate}:
\begin{equation}
\label{eq:paraSGD}
    \hat f_{n} = \hat f_{n-1} +2\gamma_n\left(Y_n - \hat f_{n-1}(X_n)\right)\sum_{j=1}^d \psi_{j}(X_n)\psi_j.
\end{equation}
Here we have shifted to considering our estimator $\hat{f}_{n}$ as a function, rather than thinking about $\hat{\beta}_n$ a vector of coefficients. This will be important in the nonparametric setting.

\subsection{From parametric SGD to nonparametric SGD}\label{sec:par_to_nonpar}
In this subsection we discuss the intuition in moving from SGD in a finite dimensional parametric space to an infinite dimensional space. 

We assume $f_{\rho}\in W(s,Q,\{\psi_j\})\subset L^2_{\nu}$. Since $\psi_j$ is a complete basis of $L^2_{\nu}$, we can always find an expansion of $f_{\rho}$ w.r.t. $\{\psi_j\}$:
\begin{equation}
    f = \sum_{j=1}^{\infty}\theta_j\psi_j.
\end{equation}
In Subsection \ref{subsection:parametric sgd}, we already discussed the SGD updating rule for a $d$-dimensional model $f(X) = \sum_{j=1}^d \beta^{(j)}\psi_j(X)$. In the nonparametric scenario, the number of basis function is increased from $d$ to infinity: This causes problems if care is not taken.

One might naturally consider applying a direct analog to the finite-dimensional SGD rule~\eqref{eq:paraSGD} here (we omit the constant 2):
\begin{equation}
\label{eq:beforescaledown}
    \hat f_{n} = \hat f_{n-1} +\gamma_n\left(Y_n - \hat f_{n-1}(X_n)\right)\sum_{j=1}^{\infty} \psi_{j}(X_n)\psi_j.
\end{equation}
Unfortunately we run into a severe problem: The series $\sum_{j=1}^{\infty} \psi_j(X_n)\psi_j$ does not converge even if all $\psi_j$ are bounded (it is direct to check when $X_n = 0$ and $\psi_j$ are trigonometric functions). However, as we assume $f_{\rho}\in W(s)$, we know that those higher order components, $\psi_j$, $j\gg 1$ should have very small coefficients. Thus, one natural solution is to use a different step size per component, that decreases as $j$ increases. By doing ``less fitting'' for larger $j$, we can stabilize our update (smaller variance), and yet might still appropriately fit the overall regression function. In particular one might modify \eqref{eq:beforescaledown} to

\begin{equation}
\label{mainform}
    \hat f_{n} = \hat f_{n-1} +\gamma_n\left(Y_n - \hat f_{n-1}(X_n)\right)\sum_{j=1}^{\infty} t_j\psi_{j}(X_n)\psi_j,
\end{equation}
where the component-specific (or dimension-specific) learning rate $t_j>0$ are monotonically decreasing with $j$. For $t_j$ decreasing fast enough and uniformly bounded $\psi_j$, the function series $\sum_{j=1}^{\infty} t_j\psi_j(X_n)\psi_j$ is absolutely convergent. Now \eqref{mainform} becomes a sensible nonparametric SGD updating rule when the hypothesis space is a Sobolev ellipsoid. In addition, sometimes $\sum_{j=1}^{\infty}t_j\psi_j(X_n)\psi_j$ actually has a simply characterized closed form (in particular, for many RKHS). In such cases, \eqref{mainform} results in a relatively straightforward algorithm. More specifically, one can show that when $t_j = j^{-2s}$ and $\gamma_n = \Theta(n^{-\frac{1}{2s+1}})$, the average 
\begin{equation}
\label{eq:polyak}
  \bar f_n:=\frac{1}{n}\sum_{i=1}^n\hat f_i
\end{equation}
is a rate-optimal estimator of $f_{\rho}\in W(s)$. This was recently proposed (though motivated quite differently) in the context of RKHS hypothesis spaces \cite{dieuleveut2016nonparametric}. The authors there engage directly with the \emph{kernel function} for the RKHS (though their updating rule is equivalent to eq~\eqref{mainform}). This will be discussed in more detail in Section~\ref{section: related work}. Our work engages and extends these ideas (in combination with sieve estimation) to form a statistically rate-optimal online estimator with greatly reduced computational and memory complexity. 

\section{Related work}
\label{section: related work}
Nonparametric online learning is a relatively new area.
A few remarkable functional stochastic approximation algorithms have been proposed in the last two decades \cite{calandriello2017efficient,dieuleveut2016nonparametric,marteau2019globally,tarres2014online,ying2008online}. The key ideas in that body of work are intimately related to those mentioned in Section~\ref{sec:par_to_nonpar}, however, they engage those ideas from a different direction: They assume that the hypothesis function space $\mathcal{F}$ is an RKHS, and then leverage the kernel in that space. In particular, the RKHS structure makes it possible to take the gradient of the evaluation functional $L_x(f):= f(x)$, with respect to the RKHS inner product $\langle \cdot , \cdot \rangle_{K}$, i.e.
\begin{equation}
    L_x(f+\epsilon g) = f(x) + \epsilon g(x) = L_x(f) + \epsilon\langle g,K_x \rangle_K.
\end{equation}
Thus, $K_x(\cdot):=K(x,\cdot)\in \mathcal{F}$ is the gradient of functional $L_x$ at $f$. However, one cannot do this in the general  $L^2_{\rho_X}$ space where the evaluation functional is no longer a bounded operator. 

Thus when $\mathcal{F}$ is an RKHS associated with kernel $K$, there is a simple nonparametric SGD updating rule for minimizing $E[(Y-f(X))^2]$ over $\mathcal{F}$:
\begin{equation}
\label{eq:RKHS_sgd}
\begin{aligned}
    \hat f_0 & = 0\\
   \hat f_{n} & = \hat f_{n-1} +\gamma_n\left(Y_n - \hat f_{n-1}(X_n)\right)K(X_n,\cdot) 
\end{aligned}
\end{equation}
Here, because the gradient is taken with respect to the RKHS inner product, we do not have the issue encountered in \eqref{eq:beforescaledown} where our series representation of the ``gradient'' actually did not converge. In fact, by working with the RKHS inner-product, we implicitly carry out the proposal of Section~\ref{sec:par_to_nonpar} and decrease the component-specific learning rate of higher order terms. More specifically, we usually have the Mercer expansion of the kernel function:
\begin{equation}
\label{eq:Kexpansion}
    K(x,z) = \sum_{j=1}^{\infty} t_j\psi_j(x)\psi_j(z),
\end{equation}
with respect to an orthonormal basis $\{\psi_j\}$ of $L^2_{\nu}$. For many common RKHS, we have $t_j = \Theta( j^{-u})$ for some $u>1$ \citep[Appendix~A]{fasshauer2015kernel}. Thus, \eqref{eq:RKHS_sgd} corresponds precisely to the previously discussed update \eqref{mainform}. Most popular RKHS have a kernel $K(x,z)$ with a closed form representation, and thus, rather than having to store an infinite number of coefficients, after $n$ steps the estimate from \eqref{eq:RKHS_sgd} would take the form of a weighted linear combination of $n$ kernel functions \cite{dieuleveut2016nonparametric}:
\begin{equation}
\label{kernelsgdnbasis}
    \hat f_n = \sum_{i=1}^{n}b_i K(X_i,\cdot).
\end{equation}
Although such estimators (with one more Polyak averaging step \eqref{eq:polyak}) have been shown to give rate-optimal MSE \cite{dieuleveut2016nonparametric}, updating them with a new observation $(X_{n+1},Y_{n+1})$ usually involves evaluating $n$ kernel functions at $X_{n+1}$, with computational expense of order $\Theta(n)$. This is in contrast with the constant update cost of $\Theta(d)$ in parametric SGD, where $d$ is the dimension of the parameter. Thus, the time expense of nonparametric kernel SGD will accumulate at order $\Theta(n^2)$. Also, one is required to store the $n$ feature-values $\{X_i\}_{i=1}^n$ to evaluate the estimator which results in $\Theta(n)$ space expense. This relatively large time and space complexity indicates that those kernel-based SGD estimators are not ideal as methods that are nominally designed to deal with large data sets.

There has been several works in the literature aiming at improving the computational aspect of kernel SGD methods \cite{si2018nonlinear,lu2016large,koppel2019parsimonious}. These methods select a subset of the $n$ kernel functions centered at the feature vectors and use them as basis functions to construct estimators (which is also related to Nyström projection). These works emphasize the application aspect of the proposed methods. Either the statistical performance or the computational expense is guaranteed to be optimal. Also, the theoretical analysis in these works typically require the noise variable to have extremely light tails.

There has also been recent work \cite{calandriello2017efficient,marteau2019globally} aimed at improving kernel SGD algorithms by leveraging approximate second order information (SGD only uses the first order information). The estimator in \cite{marteau2019globally} is shown to give rate-optimal MSE and have better (theoretical) computational efficiency than the vanilla kernel SGD mentioned above. However, these algorithms are usually dramatically more complicated and have a myriad of hyper parameters that need to be tuned.

There is another branch of research also called "online nonparametric regression" that engages with a different setting \cite{gaillard2015chaining,rakhlin2015online}. They do not aim to minimize the (population) generalization error directly. Their definition of regret is based on comparing a running average of prediction error and the empirical risk minimizer's training error. While this is an interesting area of research, and might be used to engage with population generalization error, it is less directly applicable as training error is only useful as a means to getting good generalization error.

\section{Online Learning and the Projection Estimator: Sieve-SGD}
\label{section:Sieve-SGD}
In this section, we combine ideas from the projection estimator (in the batch learning setting), and stochastic gradient descent to develop an estimator that is suitable for online nonparametric regression. The estimator we will propose achieves the minimax rate for MSE over a Sobolev ellipsoid, and is much more computationally efficient than standard kernel SGD methods.

As a reminder, the kernel SGD estimator based on \eqref{eq:RKHS_sgd} has minimax rate optimal MSE. When $\sum_{j=1}^{\infty} t_j\psi_i(s)\psi_j(t)$ has an available closed form, it requires $\Theta(n)$ memory and $\Theta(n^2)$ computation for sequentially processing $n$ observations. We aim to improve over this and furthermore to propose an effective estimator in cases where $\sum_{j=1}^{\infty} t_j\psi_i(s)\psi_j(t)$ has no closed form.

Motivated by the projection estimator, we opt to use truncated series in the updating rule, modifying \eqref{mainform} (or equivalently \eqref{eq:RKHS_sgd}) to get
\begin{equation}
\label{eq:truncatedmain}
    \hat f_{n} = \hat f_{n-1} +\gamma_n\left(Y_n - \hat f_{n-1}(X_n)\right)\sum_{j=1}^{J_n} t_j\psi_{j}(X_n)\psi_j
\end{equation}
Here $J_n$ is an increasing sequence of integers that grows  as we collect more observations. When $J_n$ is larger, the updating rule \eqref{eq:truncatedmain} is closer to our original form \eqref{mainform}; however, a smaller $J_n$ is desirable because it results in a lower computational expense. Part of our task is identifying a ``minimal'' $J_n$ that still maintains favorable statistical properties.

Unlike for classical projection estimators, when $t_j$ is properly selected, this truncation level, $J_n$ (so long as it is suitably large) does not impact the bias-variance tradeoff (to first order). By using a truncated, rather than infinite, series, we slightly reduce variance and add some minor higher order bias. For $J_n$ sufficiently large, the first order terms for bias and variance are determined by the sequence $\{t_j\}$ (not the truncation level). This is akin to using a truncated basis for penalized regression in the batch learning setting. For example, in \cite{hall2005theory} and \cite[Section~5.2]{wood2017generalized}, the authors propose to estimate $f_{\rho}$ by solving a penalized regression spline problem, where they use a reduced spline basis for improved computation (rather than including a knot at every point). The bias/variance trade-off there is controlled via the penalty: They are careful to include enough basis elements so that the use of a reduced basis only contributes a second order term to the bias. We will also show that, when $J_n$ is properly selected, so long as the component-specific learning rate $t_j$ is not too large (controlling the variance in the dynamic of SGD) or too small (controlling the bias term), Sieve-SGD can always achieve near optimal performance with very low computational expense. In this setting it is the truncation level, rather than $t_j$, balancing the trade-off between bias and variance, which is analogous to the batch projection estimator.

We will next give details of our proposal. For this proposal we are assuming that $f_{\rho}\in W(s,Q,\{\psi_j\})\subset L_{\nu}^2$, and that $s$ is known. Based on this, we choose our component-specific step-sizes as $t_j = j^{-2\omega}$ (for some $1/2<\omega \leq s$). We also define 
\begin{equation}
    K_{x,J_n}(\cdot) = K^{\omega}_{x,J_n}(\cdot) = \sum_{j=1}^{J_n}j^{-2\omega}\psi_j(x)\psi_j(\cdot).
\end{equation}
In addition to simplifying exposition, this notation relates our method to \eqref{eq:Kexpansion}. The function $K_{x,J_n}(\cdot)$ can be seen as a truncated approximation of the kernel function $K(x,\cdot) =K^{\omega}(x,\cdot) = \sum_{j=1}^{\infty}j^{-2\omega}\psi_j(x)\psi_j(\cdot)$ that drops all the $\psi_j$ with index $j>J_n$. 

\subsection{Sieve Stochastic Gradient Descent}
\label{section:algorithm}
We now explicitly give our Sieve Stochastic Gradient Descent algorithm (Sieve-SGD) for estimation of $f_{\rho}$ in a Sobolev ellipsoid $W(s,Q,\{\psi_j\})$.

Let $J_n = \lfloor n^{\alpha}\rfloor$ for some specified $\alpha>0$ and $\omega \in (\frac{1}{2}, s]$. The parameter $\alpha$ is usually taken between $\frac{1}{2s+1}$ and 1. We use $\gamma_i$ to denote the step size (learning rate) of the $i$-th update and typically choose $\gamma_i = \Theta(i^{-\frac{1}{2s+1}})$. 

\noindent\rule{14cm}{0.4pt}

\textbf{Proposed Algorithm: Sieve Stochastic Gradient Descent (Sieve-SGD)}

\noindent\rule{14cm}{0.4pt}

Set $\alpha, \omega >0$, step size $\{\gamma_i\}$ and basis functions $\{\psi_j\}$. Initialize $\bar f_0 = \hat f_0 = 0$.

For $i = 1,2,...$ :
\begin{enumerate}
\item Calculate $J_i = \lfloor i^{\alpha} \rfloor$, collect data pair $(X_i,Y_i)$.
\item Update $\hat f_i$:
\begin{equation}
\label{submainform}
\begin{aligned}
    \hat f_{i}
    & = \hat f_{i-1} + \gamma_i\left(Y_i - \hat f_{i-1}(X_i)\right)\sum_{j=1}^{J_i}j^{-2\omega}\psi_j(X_i)\psi_j\\
    &=\hat f_{i-1} + \gamma_i\left(Y_i - \hat f_{i-1}(X_i)\right)K_{X_i,J_i}\\
    \end{aligned}
\end{equation}
\item \label{alg:pol}Polyak averaging: Update $\bar f_i$ by
\begin{equation}
\begin{aligned}
\label{Sieve-SGD:estimator}
    \bar f_i &= \frac{1}{i+1}\sum_{k=0}^i \hat f_k\\ &\left(=\frac{i}{i+1}\bar{f}_{i-1} +\frac{1}{i+1}\hat f_i\right)
    \end{aligned}
\end{equation}
\end{enumerate}

\noindent\rule{14cm}{0.4pt}

We refer to the function $\bar f_i$ as the \textit{Sieve-SGD estimate} of $f_{\rho}$. We will later show that $\bar f_i$ has rate-optimal MSE for estimating any $f_{\rho}\in W(s)$. Here we use the language of ``updating a function", but in practice one would update the coefficient vector corresponding to the functions $\{\psi_j\}_{j=1}^{J_n}$. In Appendix~\ref{app:algorithm} we attach a presentation of the algorithm that works directly with the coefficients. This estimator is quite simple, though it does require apriori selection/knowledge of $\{\psi_j\}$ and $s$ (which can be done using a left-out validation set in practice). Unfortunately showing its favorable statistical properties (in Section~\ref{section:theory}) is somewhat more complex!
\subsection{Computational expense}
\label{section:computation}
After examining the updating rule above, one can see that $\hat f_i$ has the form:
\begin{equation}
    \hat f_i(x) = \sum_{j=1}^{J_i} b_j\psi_j(x)
\end{equation}
This requires storing the coefficients $\{b_j\}_{j=1}^{J_i}$ in memory. The main computational burden of each update step is calculating $\hat f_{i-1}(X_i)$ and $K_{X_i,J_i}$. Both require evaluating $J_i$ basis functions at $X_{i}$. Thus, the computational expense of the ``Update $\hat{f}_i$" step above is of order $J_i = \Theta(i^{\alpha})$, when we take evaluating one basis function at one point as $O(1)$. And the total expense of processing $n$ samples is of order $\Theta\left(n^{1+\alpha}\right)$. The space expense is of the same order $\Theta(i^{\alpha})$: We need only store coefficients of $J_i$ basis functions. In Section~\ref{subsection:optimalspace} we will show that, under mild conditions, this memory complexity is near optimal among all estimators with rate-optimal MSE.

This compares favorably with standard kernel SGD \eqref{kernelsgdnbasis} which uses $i$ basis functions at step $i$: Our estimator uses fewer when $\alpha<1$; as we will show later, $\alpha$ can be taken as small as $\frac{1}{2s+1}$ which is a substantial improvement. In practice, the parameter $\alpha$ can either be selected based on our assumptions about $s$ (belief on the smoothness of $f_{\rho}$) or heuristically tuned for empirical performance. 
\subsection{General Convex loss}
Although the main focus of this paper is regression with squared-error loss, our algorithm has a straightforward extension to general convex loss. Suppose we want to minimize the population loss
\begin{equation}
    E\left[\ell(Y,f(X))\right]
\end{equation}
over all functions $f\in W(s,Q,\{\psi_j\})$ and the loss function $\ell(Y,\cdot)$ is convex for each $Y$. In this case, we need only modify step $2$ of the Sieve-SGD estimator in Section~\ref{section:algorithm}. Given loss $\ell(\cdot,\cdot)$, the updating rule for $\hat f_{i}$ takes the general form:
\begin{enumerate}
    \item[2')] Update $\hat f_{i}$:
    \begin{equation}
    \label{eq:convexSieve-SGD}
        \hat f_{i} = \hat f_{i-1} +  \left.\gamma_i\frac{\partial}{\partial v} \ell\left(u,v\right)\right|_{\left(Y_i,\hat f_{i-1}(X_i)\right)}K_{X_i,J_i}
    \end{equation}
\end{enumerate}
For example, with $Y=\{1,-1\}$ considering  nonparametric logistic regression, the loss function one would use is $\ell(Y,f(X)) = \log(1+\exp(-Yf(X)))$. In this case, we have
\begin{equation}
    \left.\frac{\partial}{\partial v} \ell\left(u,v\right)\right|_{\left(Y_i,\hat f_{i-1}(X_i)\right)}
    =\left(1+\exp(Y_i\hat f_{i-1}(X_i))\right)^{-1}Y_i\in \mathbb{R}
\end{equation}
Theoretical guarantees for Sieve-SGD using general convex loss are beyond the scope of this paper. However, in Section~\ref{section:simulation} we provide simulated experiments that show the empirical performance of Sieve-SGD for nonparametric logistic regression. These empirical results intimate that perhaps similar theoretical guarantees to those shown for squared-error-loss hold in a more general setting. 

\subsection{Choice of Basis Functions \& Multivariate Problems}
\label{subsection:multivaraite}
In practice, there are many available choices of univariate $\psi_j$ that in general lead to interesting (Sobolev-type) hypothesis spaces. For example,
\begin{equation}
\label{eq:cosbasis}
    \psi_1(x) = 1,\quad \psi_j = \sqrt{2}\cos((j-1)\pi x), \text{ for }j\geq 2.
\end{equation}
This set of basis functions are the ``eigenfunctions" of Sobolev spaces over $[0,1]$ (Appendix~A.2 in \cite{novak2008tractability}), which means they are orthogonal w.r.t to the Lebesgue inner product and the Sobolev inner product simultaneously. The corresponding Sobolev ellipsoid does not impose periodicity assumptions of $f_{\rho}$ and is very convenient to use in practice. Among many other choices, we can also use algebraic polynomials, or a combination of algebraic polynomial and (periodic) Fourier basis \cite{Eubank1990CurveFB}.

In most applications, the covariate $X_i$'s take value in $\mathbb{R}^p$ where $p>1$. In some situations, there are some ``canonical" choices of basis function $\psi(x):\mathbb{R}^p\rightarrow \mathbb{R}$ that people might use for identifying their (multivariate) Sobolev ellipsoid. For example, when considering estimating a function on a sphere $\mathbb{S}^2$, $\psi_j$ could be taken as the orthonormal spherical harmonics (\cite{kennedy2013classification}, \cite{michel2012lectures}).

In many situations, the basis function $\psi_{j}$ can conveniently be taken as a tensor product of a one-dimensional complete basis, and Sieve-SGD can be directly applied in this multivariate setting. If we are using a univariate Sobolev ellipsoid to represent a ball in an RKHS, then the ellipsoid defined by the tensor product basis will correspond to a ball in the RKHS spanned by the tensor product kernel (though care needs to be taken with the ordering of the basis vectors). Some technical details and numerical examples on this can be found in Appendix~\ref{app:multivariate} and the reference therein. In all of these cases, our theoretical results will hold (so long as the function $f_{\rho}$ belongs to the specified space).

A common alternative approach in multivariate problems is to impose some additional structure on the hypothesis space to make estimation more tractable. This is particularly true when the feature dimension $p$ is large. One popular model is the nonparametric additive model \citep{stone1985additive,hastie2009elements,yuan2016minimax}, which is thought to effectively balance model flexibility and interpretability. For $x\in\mathbb{R}^p$, we might consider assuming/imposing an additive structure on the regression function:
\begin{equation}
\label{eq:additive}
    f_{\rho}(x) = \sum_{k=1}^p f_{\rho,k}\left(x^{(k)}\right)
\end{equation}
where each of the component functions $f_{\rho,k}$ belong to a Sobolev ellipsoid $W_k(s_k,Q_k,\{\psi_{jk}\})$. For ease of exposition, in \eqref{eq:additive}, we assume $E[Y] = 0$ to avoid the need for a common intercept term. For a more complete version with common intercept, see Appendix~\ref{app:multivariate}. For a fixed dimension $p$, when all $W_k = W^*$ (for some Sobolev ellipsoid $W^*$), the minimax rate for estimating such an additive model is identical (up to a multiplicative constant $p$) to the minimax rate in the analogous one-dimension nonparametric regression problem with the same hypothesis space $W^*$ \citep{raskutti2009lower,stone1985additive}. For the additive model \eqref{eq:additive}, the updating rule \eqref{submainform} of Sieve-SGD could be replaced by:
\begin{equation}
    \hat f_{i}
    = \hat f_{i-1} + \gamma_i\left(Y_i - \sum_{k=1}^p\hat f_{i-1,k}\left(X_i^{(k)}\right)\right)\sum_{k=1}^p\sum_{j=1}^{J_{ik}}j^{-2\omega_k}\psi_{jk}\left(X_i^{(k)}\right)\psi_{jk}
\end{equation}
here $J_{ik}$ is the truncation level of $k$-th dimension when the sample size $=i$ and $\hat f_{i-1,k}$ is the estimate of $f_{\rho,k}$. Most of the theory that we develop in Section~\ref{section:theory} could apply here. 

\section{Generalization Guarantees of Sieve-SGD}
\label{section:theory}
In this section, we show Sieve-SGD achieves the minimax rate for nonparametric estimation in Sobolev ellipsoids under mild assumptions. 
We also show that Sieve-SGD has near minimal memory complexity among all estimators that are minimax rate-optimal for estimation in a Sobolev ellipsoid. The conditions on the hyperparameters can be used as theoretical guidance when applying Sieve-SGD to real data problems.
\subsection{Model Assumptions} 
\label{section:modelassumptions}
We begin by listing the conditions we will require in our proof. They reflect different aspects of the problem: independent observations (A1), distribution of $X$ (A2), the hypothesis space assumed for $f_{\rho}$ (A3) and tail behaviour of the noise (A4). These conditions ensure the MSE rate-optimality of Sieve-SGD.
\begin{enumerate}
    \item[A1] (i.i.d. data) The data points $(X_n,Y_n)_{n\in\mathbb{N}}\in \mathcal{X} \times \mathbb{R}$ are independently, identically sampled from a distribution $\rho(X,Y)$.
    \item[A2] (feature distribution) Let $\nu$ be a user-specified measure that is strictly positive on $\mathcal{X}$. Assume the distribution of feature $X$, $\rho_X$, is absolutely continuous w.r.t. $\nu$. Let $p_X = d\rho_X/d\nu$ denote its Radon–Nikodym derivative. We assume for some $u,\ell$ such that $0<\ell<u<\infty$:
    \begin{equation*}
        \ell \leq p_X(x) \leq u \quad\text{for all }x\in\mathcal{X}
    \end{equation*}
    
    \item[A3] (Sobolev ellipsoid) Let $\{\psi_j\}_{j=1}^{\infty}$ be a set of uniformly bounded ($\|\psi_j\|_{\infty} \leq M$), continuous, orthonormal basis of $L^2_{\nu}$. We assume the regression function $f_{\rho}$ falls in a Sobolev ellipsoid, with basis functions given by $\{\psi_j\}$, i.e. for some $s>1, Q<\infty$, \begin{equation}
        f_{\rho}\in W(s,Q,\{\psi_j\})
    \end{equation}
    
    \item[A4] (noise) One of the following two assumptions is satisfied by the noise variable $\epsilon = Y - f_{\rho}(X)$:
    \begin{itemize}
        \item $\epsilon$ is bounded by some $C_{\epsilon}$ almost surely.
        \item $\epsilon$ is independent of the features, $X$, and has a finite second moment $E[\epsilon^2] = C_{\epsilon}^2$.
    \end{itemize}
    
\end{enumerate}

\textbf{Note:} In assumption A3, we do not require  $\psi_j$ to be orthonormal w.r.t. $\rho_X$ (and it is in general not true), but only require them to be orthonormal w.r.t. the known measure $\nu$. In many cases $\nu$ might be taken to be Lesbesgue (or uniform) measure over a domain containing $\mathcal{X}$, as this is the canonical measure under which function spaces such as Sobolev spaces and Besov spaces are defined. As long as the density function $p_X$ satisfies A2, using the non-orthonormal (w.r.t. $\rho_X$) basis functions, $\psi_j$, does not prevent Sieve-SGD from having rate-optimal MSE. Also, the lower bound requirement of $p_X$ in A2 may be due to artifacts in our proof. In reality, especially when the dimension of feature $X$ is higher, such an requirement is hard to be satisfied. According to our simulation results, Sieve-SGD still achieves the minimax rate even when $\rho_X$ has a strictly smaller support than $\nu$. As compared with other work in nonparametric online learning \cite{dieuleveut2016nonparametric,tarres2014online,ying2008online}, our assumptions are more direct. We discuss this in detail later in this section. 

\subsection{Rate optimality when $t_j = j^{-2s}$}
\label{section:fixedtj}
In this section, we present the rate-optimality results of Sieve-SGD when we choosing the component-specific learning rate to be $t_j = j^{-2s}$ (or $\omega = s$ in \eqref{submainform}). In this setting, our theoretical analysis treats Sieve-SGD as a truncated-version (in the basis expansion domain) of a ``correct" kernel SGD procedure (we will discuss the ``incorrect" version very soon in Section~\ref{section:fixedJn}). Here is the main result in this setting:
\begin{theorem}
\label{maintheorem}
Assume A1-A4. Set the component-specific learning rate as $t_j = j^{-2s}$. Also set the overall learning rate to be $\gamma_n = \gamma_0 n^{-\frac{1}{2s+1}}$ with $\gamma_0 \leq (2M^2\zeta(2s))^{-1}$, where $\zeta(k)~=~\sum_{i=1}^{\infty} i^{-k}$.  Choose the number of basis functions to be $J_n \geq  n^{\alpha}\log^2 n \vee 1$ for an arbitrary $\alpha \geq \frac{1}{2s+1}$.

Then the MSE of Sieve-SGD~\eqref{Sieve-SGD:estimator} converges at the following rate
\begin{equation}
  E\|\bar f_{n} - f_{\rho}\|^2_{L^2_{\rho_X}} = O\left(n^{-\frac{2s}{2s+1}}\right).  
\end{equation} 
This implies that Sieve-SGD is a minimax rate-optimal estimator of $f_{\rho}$ over $W(s,Q,\{\psi_j\})$.
\end{theorem}
We now discuss our assumptions and results in more detail, and relate them to what is currently in the literature.\\

\textbf{Note 1}: In the analysis of many reproducing kernel methods for nonparametric estimation \cite{dieuleveut2016nonparametric,tarres2014online, yuan2010reproducing}, the spectrum of the \emph{covariance operator} plays an important role in controlling the statistical behavior of estimators. It is conventional in the community to make assumptions associated with this spectrum, which we find less natural than our related assumptions A2 and A3. The covariance operator is an analog of the covariance matrix in infinite dimensional spaces. For our problem setting, one of the natural covariance operator $T_X$ is defined as: 
\begin{equation} 
\label{eq:TX}
\begin{aligned}
T_X : L_{\rho_{X}}^{2} & \rightarrow L_{\rho_{X}}^{2} \\ 
g & \mapsto \int_{\mathcal{X}} g(\tau) \left(\sum_{j=1}^{\infty} j^{-2s} \psi_j(\tau)\psi_j\right) d\rho_X(\tau). 
\end{aligned}
\end{equation}
A direct analysis of the spectrum of $T_X$ is hard. However, there is a simpler operator that we have in hand which we can relate $T_X$ to:
\begin{equation} 
\label{eq:Tnu}
\begin{aligned}
T_{\nu} : L^{2}_{\nu} & \rightarrow L^{2}_{\nu} \\
g & \mapsto \int_{\mathcal{X}} g(\tau) \left(\sum_{j=1}^{\infty} j^{-2s} \psi_j(\tau)\psi_j \right)d\nu(\tau).
\end{aligned}
\end{equation}
We know the eigensystem of $T_{\nu}$: It is exactly $(j^{-2s},\psi_j)$ (eigenvalue, eigenfunction). It is direct to check because $\{\psi_j\}$'s are orthonormal w.r.t. $\nu$, so $\int \psi_j(\tau) \sum_{j=1}^{\infty} j^{-2s}\psi_j(\tau)\psi_j d\nu(\tau) = j^{-2s}\psi_j$. As an additional contribution, our work shows that with the simple assumption A2 \& A3, we can get knowledge about $T_X$'s eigenvalues from those of $T_{\nu}$.
\begin{lemma} 
\label{maintextlemma}
Given assumptions A2, A3, the $j$-th eigenvalue, $\lambda_j$, (sorted in a decreasing order) of $T_X$ satisfies $\lambda_j = \Theta(j^{-2s})$. 
\end{lemma}
Moreover, the upper bound of the density in A2 ensures the upper bound in Lemma~\ref{maintextlemma} ($\lambda_j = O(j^{-2s})$), and the lower bound of the density ensures the other half of the result. The proof of the above Lemma uses the underlying connection between the eigenvalues of an operator and its metric entropy. For rigorous definitions and proof of Lemma~\ref{maintextlemma}, see Appendix~\ref{app:RKHS}.

Although the exact result of Lemma~\ref{maintextlemma} is not used in the proof of Theorem~\ref{maintheorem} (or Theorem~\ref{maintheorem_fixedJn}). We still present it here since it may be of interest itself and the stated results is less technical and easier to comprehend. The proof of more technical version (Lemma~\ref{lemma:samejdifferentkernels}) follows very closely to that of Lemma~\ref{maintextlemma}. In such a more general version, we investigate the spectrum of covariance operators of form:$T_{X,J_n}^{\omega}(f) = \int f(\tau)\left(\sum_{j=1}^{J_n}j^{-2\omega}\psi_j(x)\psi_j\right)d\rho_X(\tau).$

To prove Theorem \ref{maintheorem}, we need to engage with a series of RKHSs with kernels given by
\begin{equation}
    \begin{aligned}
        K_{J_n}:\mathcal{X}\times \mathcal{X} &\rightarrow \mathbb{R} \\
        (s,t)&\mapsto  \sum_{j=1}^{J_n}j^{-2s}\psi_j(s)\psi_j(t) := K_{J_n}(s,t).
    \end{aligned}
\end{equation}
While we discuss our work in the context of Sobolev ellipsoids, there is an equivalent formulation directly in RKHS. See Appendix~\ref{app:RKHS} for more discussion. Although an explicit form for $K_{J_n}$ is not in general necessary or accessible for Sieve-SGD, the existence (i.e. the absolute convergence of the infinite sum) of $K_{J_n}$ is a direct consequence of A3. This is enough for theoretical analysis. For kernel SGD methods, a fixed kernel (with $J_n = \infty$) is used and there is only one relevant RKHS. This means, on average, kernel SGD is applying the same procedure each iteration; but for Sieve-SGD, we need to engage with a series of increasing RKHSs (on average, Sieve-SGD may not be doing the same thing between iterations). As a side contribution, we present how to handle such a more technically involved case.

\textbf{Note 2}: In contrast to our assumption A3, the hypothesis spaces in \cite{dieuleveut2016nonparametric,tarres2014online,ying2008online,marteau2019globally} are described in terms of ``$T_X$"  and its eigen-decomposition. This unfortunately obfuscates difficulties related to verifying those conditions: In particular because $\rho_X$ is involved in the definition of $T_X$ \eqref{eq:TX}, we need knowledge of (generally unknown) $\rho_X$ to characterize $T_X$, and understand its eigenvalues and eigenfunctions.

More specifically, in the literature we mentioned above, it is often assumed that for some $r\in [1/2,1]$ (Definition \ref{def:SX}):
\begin{equation}
    \|T_X^{-r}(f_{\rho})\|_{L^2_{\rho_X}}^2 < \infty
\end{equation}
This can be related to a Sobolev ellipsoid-type condition
\begin{equation}
\label{eq:theirellipsoid}
    \|T_X^{-r}(f_{\rho})\|_{L^2_{\rho_X}}^2
    =\sum_{j=1}^{\infty} \lambda_{j}^{-2r} \theta_{j}^2 < \infty\quad\text{ where }f_{\rho} = \sum_{j=1}^{\infty}\theta_j\phi_j
\end{equation}
where $(\lambda_j,\phi_j)_{j=1}^{\infty}$ are the eigenvalue and eigenfunctions of operator $T_X$, and $\phi_j$'s are orthonormal w.r.t. $L^2_{\rho_X}$. Unfortunately, we cannot directly engage with $(\lambda_j,\phi_j)_{j=1}^{\infty}$, since calculating them requires knowledge of $\rho_X$. Thus, assumptions formulated in the language of $T_X^{-r}$ are difficult to directly understand. In contrast, our assumptions translate to analyzing the spectrum of $T_{\nu}$, which has no dependence on $\rho_X$, and its spectrum can been directly calculated (as noted above).

\textbf{Note 3:} For parametric SGD methods, usually a bound on the second moment of the gradient vector is required to guarantee rate-optimal performance (both theoretically and in practice). Formally, for optimization problem~\eqref{OP}, it is usually required that $E[\|\nabla \ell(\beta)\|^2] \leq~R^2<~\infty$ for all $\beta\in\mathbb{R}^d$ \cite{borkar2009stochastic,duchi2014multiple}.

For nonparametric stochastic approximation, there is a similar regularity requirement for the ``gradient". The assumptions A2-A3 are enough to ensure this for Sieve-SGD. In our proof, we show that there exists a number $R < \infty$ such that for all $x\in \mathcal{X}$ and any $J_n$, we have $\|K_{x,J_n}\|_K^2\leq R^2$. This result is listed in Lemma~\ref{boundEK} where $R^2 = M^2\zeta(2s)$ and $\zeta(k) = \sum_{i=1}^{\infty} i^{-k}$. In Theorem \ref{maintheorem}, we required $\gamma_0$ to be smaller than $(2M^2\zeta(2s))^{-1}$ to ensure our theoretical guarantees.

\textbf{Note 4:} For completeness, here we state the minimax-rate of our nonparametric regression problem over a Sobolev ellipsoid:
\begin{equation}
    \liminf_{n\rightarrow\infty} \inf_{\hat f} \sup_{f_{\rho}\in W(s,Q,\{\psi_j\})} E\left[n^{\frac{2s}{2s+1}}\|\hat f-f_{\rho}\|_{L^2_{\rho_X}}^2\right] \geq C
\end{equation}
where the infimum ranges over all possible functions $\hat f$ that are sufficiently measurable. For a derivation of this lower bound, see  \cite[Chapter~15]{wainwright2019high}.

\subsection{Adaptivity to $t_j$ for Properly Chosen $J_n$}
\label{section:fixedJn}
In section~\ref{section:fixedtj} we presented the optimality guarantee of Sieve-SGD: when the component-specific learning rate is properly chosen, i.e. $t_j = j^{-2s}$, Sieve-SGD is statistically optimal so long as the number of basis functions does not increase too slow, that is, $J_n \geq n^{\frac{1}{2s+1}}\log^2(n)$. Specifically, when $J_n = \infty$, the Sieve-SGD updating rule reduces to the kernel SGD updating rule \eqref{eq:RKHS_sgd} with kernel $K(X_n,\cdot) = \sum_{j=1}^{\infty}j^{-2s}\psi_j(X_n)\psi_j(\cdot)$. So long as we have access to the closed-form of $K(X_n,\cdot)$, the corresponding kernel SGD is also optimal under the same conditions. In such a scenario, Sieve-SGD can be seen as a truncated-version of a ``correct" kernel SGD method with much better computational properties.

However, if we choose $t_{j} = j^{-2\omega}, \omega \neq s$, the corresponding kernel SGD, using kernel $K(X_n,\cdot) = \sum_{j=1}^{\infty} j^{-2\omega}\psi_j(X_n)\psi_j(\cdot)$, is no longer optimal without modifying the learning rate $\gamma_n$ accordingly \cite{dieuleveut2016nonparametric}. But for Sieve-SGD, so long as the truncation level $J_n$ is properly selected, the statistical performance of Sieve-SGD is near optimal for a quite wide range of choice of $\omega$.

\begin{theorem}
\label{maintheorem_fixedJn}
Assume A1-A4. Set the component-specific learning rate to be $t_j = j^{-2\omega}$ with $\frac{1}{2} < \omega < s$. Choose the learning rate to be $\gamma_n = \gamma_0 n^{-\frac{1}{2s+1}}$, with $\gamma_0 \leq M^2\zeta(2\omega)/2$. Choose the number of basis functions to be $J_n~=~n^{\frac{1}{2s+1}}\log^2 n \vee 1$. 

Then the MSE of Sieve-SGD~\eqref{Sieve-SGD:estimator} converges at the following near optimal rate
\begin{equation}
  E\|\bar f_{n} - f_{\rho}\|^2_{L^2_{\rho_X}} = O\left(n^{-\frac{2s}{2s+1}}\log^2 n\right)  
\end{equation} 
\end{theorem}

\textbf{Note 1:} The requirement of $t_j < j^{-1}$ is to guarantee a finite ``second moment" of the gradient, recall the Note 3 under Theorem~\ref{maintheorem}. Once such a minimal requirement is satisfied, the decay rate of $t_j$ does not influence neither the rate of statistical guarantee, nor the computational expense of the estimators. As we will discuss very soon in section~\ref{subsection:optimalspace}, the choice of $J_n = n^{\frac{1}{2s+1}}\log^2 n$ in Theorem~\ref{maintheorem_fixedJn} and Theorem~\ref{maintheorem} would result in algorithms that are both statistically and computationally near-optimal, which is very rare in the literature of online nonparametric learning and could be of interest in practice.

\textbf{Note 2:} The most direct form of projection estimator determines the basis functions' coefficients by solving a (unpenalized) least square problem \eqref{NPLS} in which there are no learning rates involved. It is the truncation level $J_n$ that determines the bias-variance trade-off and statistical performance. In Theorem~\ref{maintheorem_fixedJn} we present a stochastic approximation version of such a result. From a reproducing-kernel methodology perspective, Theorem~\ref{maintheorem} investigates the cases when the capacity of the kernel ($\omega$) matches the source smoothness ($s$); in Theorem~\ref{maintheorem_fixedJn} we discussed under what conditions the mismatch between these two quantities does not affect the statistical (and computational) properties of Sieve-SGD. We also note that the overall proof structures of Theorem~\ref{maintheorem} and Theorem~\ref{maintheorem_fixedJn} are similar; the difference is, in the proof of Theorem~\ref{maintheorem} we need Lemma~\ref{lemma:Jnvariance} and related technical results, but for Theorem~\ref{maintheorem_fixedJn} we use Lemma~\ref{lemma:myvariancecontrol} instead.

\subsection{Near optimal space expense}
\label{subsection:optimalspace}

In this section we will show that Sieve-SGD is asymptotically (nearly) space-optimal for estimating $f_{\rho}$ in a Sobolev ellipsoid under the conditions listed in Section~\ref{section:modelassumptions}. We will show that, even with computer round-off error, Sieve-SGD only needs $\Theta(n^{\frac{1}{2s+1}}\log^3 n)$ \emph{bits} to achieve the minimax rate for MSE (or off by a $\log^2(n)$ term when $\omega \neq s$ as stated in Theorem~\ref{maintheorem_fixedJn}), and further, that there is no estimator with $o(n^{\frac{1}{2s+1}})$ bits of space expense that can achieve the minimax-rate for estimating  $f_{\rho}\in W(s,Q)$. In our analysis we note that computers cannot store decimals in infinite precision, and formally deal with a modified version of our algorithm that stores coefficients in fixed precision (that grows in $n$): This necessitates the extra $\log(n)$ term (compared with the number of basis function needed in Theorem~\ref{maintheorem} and \ref{maintheorem_fixedJn}). The modified algorithm with fixed, but growing precision still results in the same MSE when round-off error is not considered.

We first give a more formal definition of the space expense of an estimator in our analysis. A regression estimator can be seen as a mapping $M_n$ from the data $Z_1^n=\{(X_i,Y_i)\ |\ i=1,2,...,n\}$ to a function $\hat f_n\in\mathcal{F}$. For any such $M_n$ that can be engaged by a computer, must be decomposable into an ``encoder-decoder" pair $(E_n,D_n)$. Here $E_n$ represents the ``encoder" that compresses the information into computer memory. Formally, we define $E_n$ to be a mapping from $Z_1^n$ to a binary sequence of length $b_n$. And the corresponding $D_n$ is the ``decoder" of the binary sequence that translates the information saved in memory back to a mathematical object $\hat f_n$. Generally, the binary sequence length $b_n$ will increase with $n$: As more information is contained in the data, we need more memory to store an increasingly accurate estimate of our regression function.

Given an estimator that can decomposed into a pair $(E_n, D_n)$, one can see that the decomposition is not unique. There are, in fact, infinitely many pairs that are trivially different from each other for any such estimator. Moreover, $E_n,D_n$'s can be random mappings if we allow random algorithms: For example, random forests include additional randomness due to bootstrapping/variable selection. In order to be more precise regarding memory complexity constraints, we introduce the following formalization.

\begin{definition}[$b_n$-sized estimator]Given a sequence of integers $(b_i)_{i\in\mathbb{N}}$, we say an estimator $M_n:(\mathcal{X}\times\mathbb{R})^n\rightarrow \mathcal{F}$ is a $b_n$-sized estimator if it satisfies the following conditions:
\begin{enumerate}
    \item For every $n$, there exists an encoder mapping $E_n:(\mathcal{X}\times\mathbb{R})^n\rightarrow \{0,1\}^{b_n}$, and a decoder mapping $D_n: \{0,1\}^{b_n}\rightarrow\mathcal{F}$ such that \begin{equation}
        M_n = D_n\circ E_n
    \end{equation}
    \item The decoder $D_n$ is a known, fixed mapping. $E_n$ can be either a random or fixed mapping.
\end{enumerate}
\end{definition}

We use the sample mean as a toy example to illustrate the above definition. In practice, the sample mean is usually a 64-sized estimator of the population mean. Here 64 stands for the number of bits needed to represent a double-precision floating point number. In this case the size $b_n = 64$ does not increase with sample size $n$. However not every real number can be arbitrarily precisely specified by a fixed-length floating-point number, so a careful asymptotic analysis of estimation of the mean suggests that perhaps we should store a sample mean with growing levels of precision, i.e. $b_n$ would need to grow with $n$. A binary sequence of length $s$ can specify $2^s$ real numbers, so to achieve the $O(n^{-1})$ statistically optimal bound for mean estimation, a $\log(n)$-sized version of sample mean is formally required. In practice, 64-bit precision is generally more than enough for mean estimation. Nevertheless, in this manuscript we aim to give a more formal and precise asymptotic analysis of our Sieve-SGD estimator.

Readers who are more familiar with computational complexity theory may find our definition similar to a (probabilistic) Turing machine. However, in our framework the machine does not use binary sequences on tapes as input and output; nor do we need to identify the basic operations on the "machine". We aimed to remove unnecessary complexity for readers with a more statistical background. Discussion of Turing machines using finite length working tape can be found in \cite[Chapter~4]{arora2009computational}.

To construct Sieve-SGD estimators that achieve (near) optimal MSE, we only need to store the coefficients of the $J_n = \Theta(n^{\frac{1}{2s+1}}\log^2 n)$ basis functions. However, as in our example with the sample mean, we need to be careful about the precision with which we store those coefficients. We need to determine: i) how small we require the round-off error to be in order to maintain the statistical optimality of Sieve-SGD, and ii) how much space expense is required to achieve such precision. In Appendix    ~\ref{subsection:spaceexpenseofSieve-SGD} we identify how round-off error is introduced into the system and how it decreases as more bits are used to store each coefficient. In Corollary~\ref{corollary:spaceexpense} we show that by using $\Theta(\log n)$ bits per coefficient, a $O(n^{\frac{1}{2s+1}}\log^3n)$-sized version of Sieve-SGD can achieve the same optimal convergence rate as in the infinite precision setting (or equivalently round-off-error-free setting). 

Combining the above result with the following theorem, we can conclude that no MSE rate-optimal estimator can require less memory by a polynomial factor than Sieve-SGD.
\begin{theorem}
\label{theorem:minimummemory}
Let $b_n$ be a sequence of integers, and $b_n = o\left(n^{\frac{1}{2s+1}}\right)$. Let $\mathcal{M}(b_n)$ be the collection of all $b_n$-sized estimators, then we have
\begin{equation}
   \lim_{n\rightarrow\infty} \inf_{M_n\in\mathcal{M}(b_n)}\sup_{f_{\rho}\in W(s,Q,\{\psi_j\})} E\left[n^{\frac{2s}{2s+1}}\|M_n(Z_1^n) - f_{\rho}\|_{L^2_{\rho_X}}^2\right] = \infty
\end{equation}
i.e. no such $b_n$-sized estimators can be rate-optimal.
\end{theorem}
This theorem tells us we cannot find any minimax rate-optimal $o(n^{\frac{1}{2s+1}})$-sized estimator. Thus the best rate-optimal estimator one can expect to find is a $\Theta(n^{\frac{1}{2s+1}})$-sized estimator: Sieve-SGD's space expense only misses this lower bound by a poly-logarithmic factor.

We give the proof of the above theorem in Appendix~\ref{app:memoryusage}. Although here we focus on regression in Sobolev spaces, the technique used can be applied to other hypothesis spaces. The proof is based on the concept that metric-entropy is the minimal number of bits needed to represent an
arbitrary function from a function space up to $\epsilon$-error, which can be traced back to \cite{kolmogorov1959varepsilon}. Also, following a very similar argument, one can prove that no constant-sized estimator can be rate-optimal (or even consistent) for parametric regression problems. We discuss this further in the Appendix~\ref{app:memoryusage}. We also include some discussion of the time expense in Section~\ref{section:discussion}.

\section{Simulation study}
\label{section:simulation}
\subsection{Sieve-SGD for online regression}
\label{section:regressionsimulation}
In this section, we illustrate both the statistical and computational properties of Sieve-SGD with simulated examples. The two examples we use have different $f_{\rho}$, $W(s,Q,\{\psi_j\})$ and $\rho_X$. The user-specified measure $\nu$ is taken as the uniform distribution over $[0,1]$ in both. We provide the details of our simulation settings in Table~\ref{settingtable}. These two examples are designed for verifying our theoretical guarantees: The $f_{\rho}$ we use have known explicit series expansion or is constructed explicitly using the basis function $\psi_j$ (to ensure the truth is hard enough to learn in the assumed Sobolev ellipsoid). In Appendix~\ref{app:multivariate} we provide more numerical examples to better mimic the practical application: we engage with multi-variate features and compare Sieve-SGD with many popular machine learning methods. 

\textbf{Example 1} In this example, we examine the empirical performance of Sieve-SGD and compare it with two other methods in batch or online nonparametric regression: kernel ridge regression (KRR) \cite{wainwright2019high} and kernel SGD \cite{dieuleveut2016nonparametric}. We will see that the relationship between generalization error $E\|\bar f_n- f_{\rho}\|_2^2$ and sample size corresponds well with our theoretical expectations presented in Theorem~\ref{maintheorem} (Fig~\ref{example1}).

The true regression function we chose for Example~1 is also used in the analysis of kernel SGD \cite{dieuleveut2016nonparametric}. In that paper, kernel SGD with Polyak averaging is compared with other (kernel-based) nonparametric online estimators \cite{tarres2014online,ying2008online}, and has been shown to have similar or better performance, so we include only kernel SGD with averaging as the reference online-estimator. We also note that although KRR performs slightly better than online methods, its time expense (which is of order $\Theta(n^3)$ per update) is dramatically more than online-estimators (kernel SGD $\Theta(n)$, Sieve-SGD $\Theta(J_n)$, per update).

We compare the empirical performance of Sieve-SGD under two different distributions of $X$. In Fig~\ref{example1} panel (A), $X$ has an uniform distribution over $[0,1]$ and in panel (B) it has a distribution with a strictly smaller support (uniform over $[0.25, 0.75]$). The trigonometric basis functions we use are orthonormal w.r.t. $\nu$, the Lebesgue measure over $[0,1]$ (panel (A)) but not w.r.t. the one in panel (B). Although only the feature distribution in panel (A) satisfies the distribution assumption in A2, in both cases Sieve-SGD achieves the optimal-rate. This is a heuristic evidence indicating the lower bound requirement in A2 may be due to some artifacts in the proof. 
\begin{table}[!htbp]
\caption{Settings of simulation studies. $B_4(x) = x^{4}-2 x^{3}+x^{2}-\frac{1}{30}$ is the 4-th Bernoulli polynomial. $\{x\}$ indicates the fractional part of $x$. } 
\label{settingtable}
\begin{center}
\begin{tabular}{@{}ccc@{}}
\hline
& Example 1 & Example 2 \\
\hline
True $f_{\rho}$ & $B_4(x)$& $4\sqrt{2}\sum_{j=1}^{50} (-1)^{j+1}j^{-4}\sin((2j - 1)\pi x/2)$\\
ellipsoid para. $s$ & 2 & 3\\
$J_n$ & $n^{0.21}$  & $n^{0.10}$ \& $n^{0.15}$ \& $n^{0.43}$ \\
$t_j$ & $j^{-1.02}$ \& $j^{-4}$ & $j^{-6}$\\
\multirow{2}{*}{$\psi_j(x)$} & $\sin(2\pi \lceil j/2\rceil x)$, $j$ is even& \multirow{2}{*}{$\sqrt{2}\sin(\frac{(2j-1)\pi x}{2})$}\\
&$\cos(2\pi \lceil j/2\rceil x)$, $j$ is odd&\\
Kernel $K(s,t)$ & $-\frac{1}{24}B_4(\{ s-t\})$& $\min(s,t)$\\
Noise & Unif[$-0.02,0.02$] or Unif[$-0.2,0.2$]&  Normal(0,1)\\
$\gamma_0$ & 3& 1\\
\hline
\end{tabular}
\end{center}
\end{table}
\begin{figure}[!htbp]
\includegraphics[width=\textwidth]{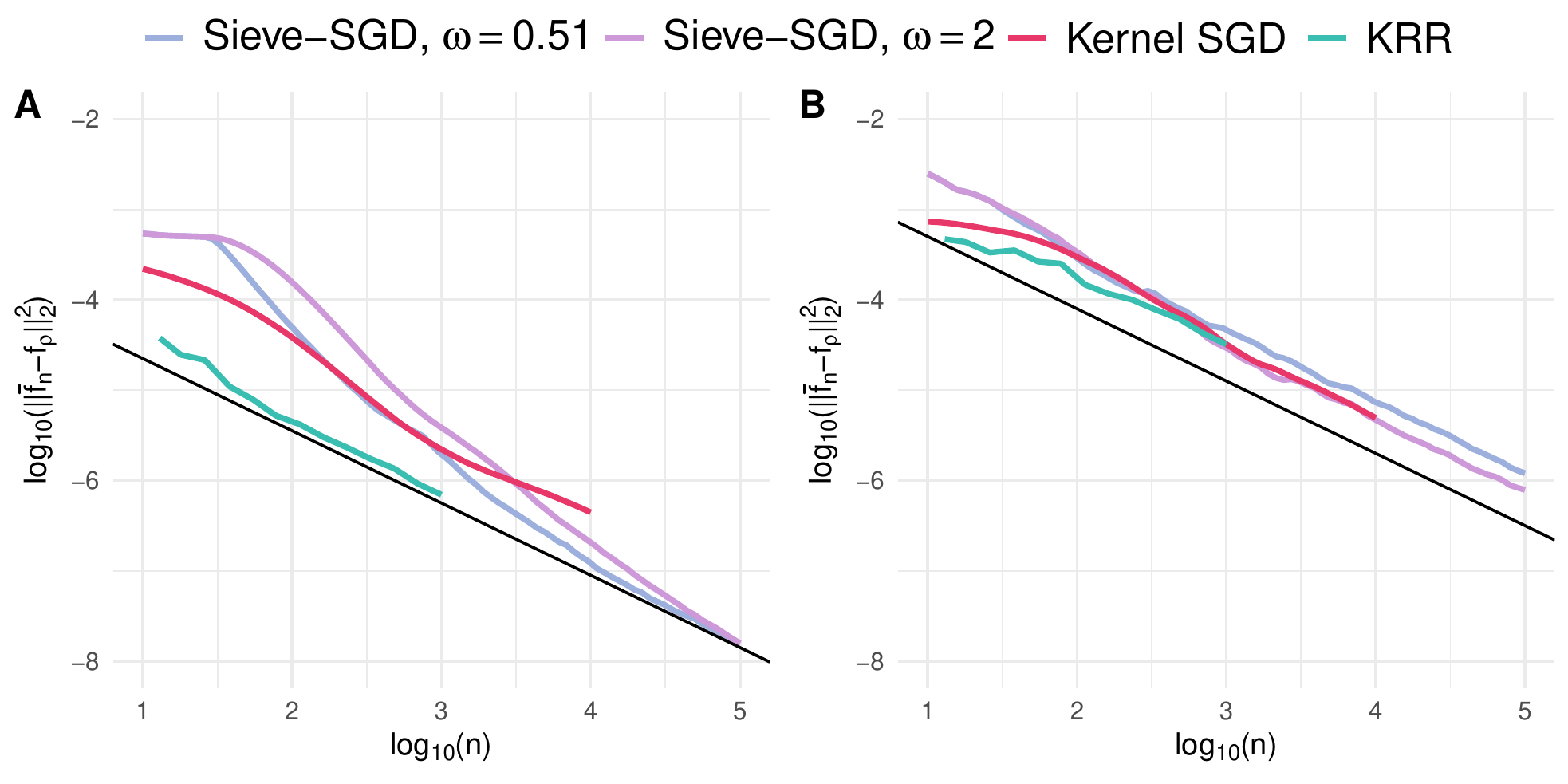}
\caption{Example 1, $\log_{10}\|\bar f_{n}-f_{\rho}\|_2^2$ against $\log_{10} n$. The Black line has slope $=-4/5$, which represents the optimal-rate. Each curve is calculated as the average of 100 repetitions. (\textbf{A}) $X$ is uniformly distributed over $[0,1]$. In this setting, SNR $\sim 3$. (\textbf{B}) $X$ has a distribution in which $\psi_j$ are not orthonormal. We present the results with very large noise, SNR$\sim 0.03$. Due to different computational costs, we chose different maximum $n$ for different methods. }
\label{example1}
\end{figure}

\textbf{Example 2} In this example, we consider the performance of Sieve-SGD under different $J_n = \lfloor n^{\alpha}\rfloor$ (number of basis functions). The $f_{\rho}$ we use is explicitly constructed with basis functions $\psi_j(x) = \sqrt{2}\sin\left((2j-1)\pi x/2\right)$ and we tune the proposed method based on the (correct) assumption that it belongs to Sobolev ellipsoid $W(3,Q,\{\psi_j\})$ (see Theorem 4.1 of \cite[Chapter~1]{hernandez1996first} for completeness and orthonormality of $\{\psi_j\}$).

According to Theorem~\ref{maintheorem}, in order to guarantee statistical optimality, we need $\alpha \geq \frac{1}{(2s+1)} \sim 0.14$. We consider several values of $\alpha$, one below the this threshold, and two above it:
\begin{equation}
    \mathbf{0.10}<\frac{1}{2s+1} \sim 0.14 < \mathbf{0.15} < \mathbf{0.43}
\end{equation}

As we can see from Fig~\ref{example2} (A), when $\alpha = 0.15\ \&\ 0.43$, Sieve-SGD is rate-optimal as expected. When $\alpha = 0.10$, we are using too few basis functions, which results in the sub-optimal statistical performance. Such a suboptimality is because of the bias term: there are too few basis functions used. In fact, the parameter setting $\alpha = 0.1$ is so small that there are only $3$ basis functions used when $n = 10^5$. To verify the above statement, we can briefly calculate when the second and the third basis functions are added in: $(10^3)^{0.1} \sim 2$, this corresponds to the first acceleration of the learning rate around $\log_{10}(n) = 3$; similarly, $(10^{4.8})^{0.1}\sim 3$, which explains the second one.

In Fig~\ref{example2} (B), we show the CPU time for reference. For Sieve-SGD, the \textit{accumulated} CPU time should be on the order of $\Theta(n^{1+\alpha})$: The larger $\alpha$, the more basis functions required, the slower the algorithm. We also include the CPU time of kernel SGD with averaging as a benchmark, which has a cumulative computational expense of order $\Theta(n^2)$. The code is written in R (4.0.4), and runs on (the CPU of) a machine with 1 Intel Core m3 processor, 1.2 GHz, with 8 GB of RAM.
\begin{figure}[!htbp]
\includegraphics[width=\textwidth]{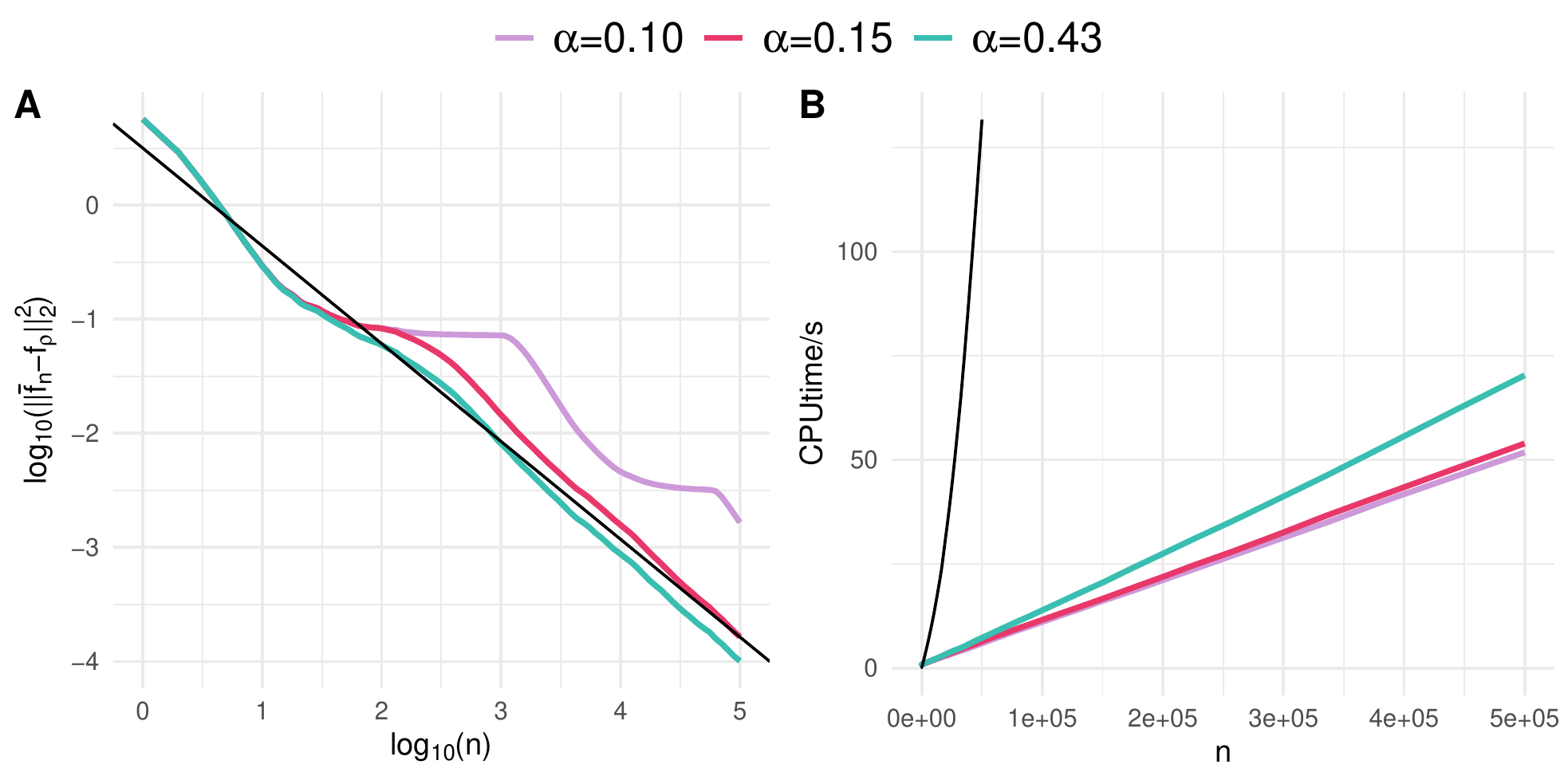}
\caption{Example 2, effect of truncation exponents $\alpha = 0.10,0.15,0.43$. (A) Statistical performance, $\log_{10}\|\bar f_{n}-f_{\rho}\|_2^2$ against $\log_{10} n$. The black line has slope $=-6/7$, which represents the optimal-rate. Each curve is calculated as the average of 100 repetitions. (B) The accumulated CPU time to process $n$ observations. The black line is the CPU time of kernel SGD, included for benchmark.}
\label{example2}
\end{figure}
\subsection{Sieve-SGD for Alternative Convex Losses}
In this section, we provide the results of an experiment applying Sieve-SGD to online nonparametric logistic regression. Although this manuscript gives no theoretical guarantees in this setting, it is still of interest to see the empirical performance of Sieve-SGD for general convex loss. Here, the distribution of class labels $Y$ was generated by $Y\sim 2\operatorname{Ber}(g(X))-1$, where $(g(x))^{-1} = 1+\exp(-5(1-2|x-0.5|))$; and the distribution of $X$ was uniform over $[0,1]$. Thus, the minimizer $f^*$ of loss $E[\ell(Y,f(X))] = E[\log(1+\exp(-Yf(X)))]$ is $f^* = 5(1-2|x-0.5|)$.

When we apply the Sieve-SGD estimator~\eqref{eq:convexSieve-SGD} to this problem, we assume
\begin{equation}
    f^*\in W\left(1,Q,\left\{\sqrt{2}\sin\left((2j-1)\pi x/2\right)\right\}\right)
\end{equation}
We try several $\alpha = 0.10,0.33,0.50$, all with $\gamma_0 = 6$. As we can see from Fig~\ref{example3}, the regret $E[\ell(\bar f_n) - \ell(f^*)]$ converges to zero at an apparent rate of $ n^{-2/3}$ when $\alpha = 0.33,0.50$ (which would agree with our result for squared error loss). When the number of basis functions increases too slowly (here is $\alpha = 0.10$), the regret decreases slowly after $\sim 10$ observations (for similar reason of overflowing bias term as we noted in section~\ref{section:regressionsimulation}).
\begin{figure}[!htbp]
\includegraphics[width=0.7\textwidth]{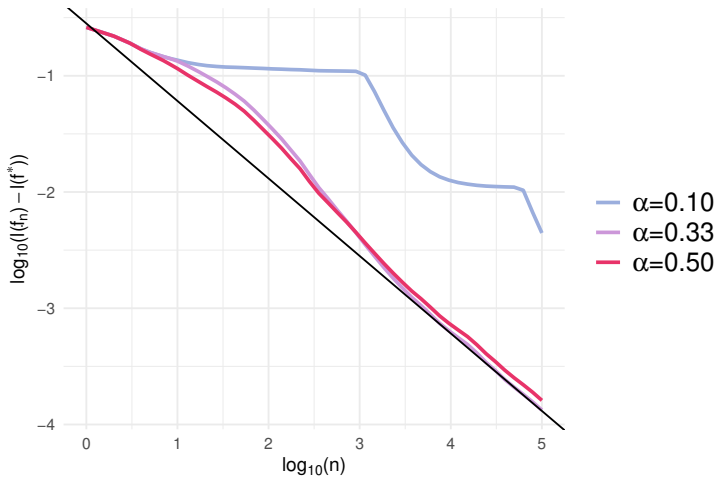}
\caption{Example 3, empirical performance of Sieve-SGD in nonparametric logistic regression problem. Plot $\log_{10}(l(\bar f_n) - l(f^*))$ against $\log_{10} n$. The Black line has slope $=-2/3$. Each curve is calculated as the average of 100 repetitions.}
\label{example3}
\end{figure}
\newpage

\section{Discussion}
\label{section:discussion}
In this paper, we considered online nonparametric regression in a Sobolev ellipsoid. We proposed the \emph{Sieve Stochastic Gradient Descent estimator (Sieve-SGD)}, an online estimator inspired by both a) the nonparametric projection estimator, which is a special realization of general sieve estimators; and b) estimators constructed using stochastic gradient descent algorithms. By using an increasing number of basis functions, Sieve-SGD has rate-optimal estimation error and is computationally very efficient.

For online learning problems with general convex losses, the optimal estimation rate depends on both the hypothesis space and loss function (e.g. whether it is Lipschitz or strongly convex). In this paper we did not establish theoretical guarantees for Sieve-SGD when applied to general convex loss, however, we gave some empirical evidence that it can perform well there. We believe our proof techniques might be extended beyond squared-error loss, perhaps using ideas in \cite{bach2013non, caponnetto2007optimal,marteau2019globally,marteau2019beyond}.

We've seen a rich collection of work in the past decade targeting the optimality of estimators under computational (especially time expense) constraints. A lot of those results are established in the context of sparse PCA and related sparse-low-rank matrix problems, e.g. \cite{cai2017computational,gao2015minimax,gao2017sparse,ma2015computational,wang2016statistical,zhang2014lower}. The main focus of these work is usually comparing the statistical performance of the best polynomial-time algorithm with that of the "optimal" algorithm without any computational restrictions. By relating their statistical problem with a known NP problem \cite{arora2009computational}, they can usually show the sub-optimality of polynomial-time algorithms under the famous conjecture $P\neq NP$. However, for the nonparametric regression problem in this paper, there is a polynomial-time estimator that can achieve the global minimax-rate. It is of theoretical interest to know if there are any statistically rate-optimal online estimators that require less than $\Theta(n^{1+\frac{1}{2s+1}})$ time-expense: We hypothesize that there are not.

\section*{Acknowledgments}
N.S and T.Z. were both supported by NIH grant R01HL137808.

\newpage
\begin{appendices}

\counterwithin{figure}{section}

\section{Algorithm of Sieve-SGD, Numerical Version}
\label{app:algorithm}
In the main text, section~\ref{section:algorithm}, we presented a functional form of the proposed Sieve-SGD algorithm. To facilitate the comprehension of it, we also attach an equivalent numerical version of it.

\noindent\rule{14cm}{0.4pt}

\textbf{Proposed Algorithm: Sieve Stochastic Gradient Descent (Sieve-SGD)}

\noindent\rule{14cm}{0.4pt}

Set $\alpha, \omega>0$, step size $\{\gamma_i\}$ and basis functions $\{\psi_j\}$. 

Initialize $\bar \beta_{j}, \hat \beta_{j} = 0$ for all $j\in\mathbb{N}^+$.

For $i = 1,2,...$ :
\begin{enumerate}
\item Calculate $J_i = \lfloor i^{\alpha} \rfloor$, collect data pair $(X_i,Y_i)$.
\item Update $\hat f_i$:
\begin{equation}
    \text{res}_i \leftarrow Y_i - \sum_{j=1}^{J_{i-1}} \hat\beta_j \psi_j(X_i)
\end{equation}

For $j = 1,...,J_i$:
\begin{equation}
    \hat\beta_j \leftarrow \hat\beta_j + \gamma_i\text{res}_i\left(j^{-2\omega} \psi_j(X_i)\right)
\end{equation}
\item Update $\bar f_i$\\

For $j = 1,...,J_i$:

\begin{equation}
    \bar\beta_j \leftarrow \frac{i}{i+1}\bar\beta_j + \frac{1}{i+1} \hat \beta_j
\end{equation}
\end{enumerate}

\noindent\rule{14cm}{0.4pt}

\section{Multivariate Regression Problems}
\label{app:multivariate}
In this section, we will give additional discussion of the technical details for multivariate regression using the Sieve-SGD estimator. 

\subsection{Hyperbolic cross and Sieve-SGD}
In main text Section~\ref{subsection:multivaraite}, we discussed using a tensor product basis to approach multivariate problems. Here we go into more details and more technical discussion. Given an univariate orthonormal basis $\{\psi_j, j\in\mathbb{N}^+\}$, the set of tensor product functions $\{\psi_{\mathbf{j}}(\mathbf{x}) = \prod_{k=1}^p \psi_{\mathbf{j}^{(k)}}(\mathbf{x}^{(k)}), \mathbf{j}\in\left(\mathbb{N}^+\right)^p\}$ is also an orthonormal basis ($\mathbf{v}^{(k)}$ is the $k$-th component of $\mathbf{v}\in\mathbb{R}^p$). However, there are more choices of the order in which we include basis functions when estimating an unknown regression function. We propose using the index product $\prod_{k=1}^p \mathbf{j}^{(k)}$ to determine such an ordering. That is, basis functions with smaller index product will be used earlier when constructing Sieve-SGD, such a choice is called hyperbolic cross in the literature \cite{dung2017hyperbolic, shen2010sparse}. Before we discuss the intuition of such an ordering (which has been established in the literature), we present some numerical examples of applying Sieve-SGD in multi-variate regression problems. 

We consider two dimension settings of the feature variable $\mathbf{X}$, $p = 2 \text{ and } 10$. The feature vector is defined as: $\mathbf{X}^{(1)} = U_1$, $\mathbf{X}^{(k)} = (U_{k} - U_{k-1} + 1)/2$ for $k = 2,...,p$. Here $U_k$ are independent $\text{Unif}[0,1]$ variables. The true regression function is defined as
\begin{equation}
    f_{\rho}(\mathbf{x}) = \sum_{k=1}^p\sum_{l = k}^p (0.5 - |\mathbf{x}^{(k)} - 0.5|)(0.5 - |\mathbf{x}^{(l)} - 0.5|)
\end{equation}
And the outcome $Y = f_{\rho}(\mathbf{x})+\epsilon$ is contaminated by a normal distributed noise, $\text{SNR}=3$. The main update rule of Sieve-SGD we applied here is
\begin{equation}
\hat{f}_{i}=\hat{f}_{i-1}+\gamma_0 i^{-1/(2s+1)}\left(Y_{i}-\hat{f}_{i-1}\left(\mathbf{X}_{i}\right)\right) \sum_{\mathbf{j}\in \mathcal{J}_{prod}(cpi^{1/(2s+1)})} \left(\prod_{k=1}^{p}\mathbf{j}^{(k)}\right)^{-2\omega} \psi_{\mathbf{j}}\left(\mathbf{X}_{i}\right) \psi_{\mathbf{j}},
\end{equation}
where $\omega = 0.51$ and $s = 2$. We use $\gamma_0\in\{0.1, 0.5\}$ and $c\in\{4,8\}$: the latter two parameters may be different in each replication. In Figure~\ref{exampleB1}, we present average performance of each method based on 100 replications. The index set $\mathcal{J}_{prod}(J)\subset \left(\mathbb{N}^+\right)^p$ contains the $p$-dimension index vectors of smallest product. For example, when $p = 2$, $\mathcal{J}_{prod}(5) = \{(1,1), (1,2), (2,1), (1,3), (3,1)\}$. Arbitrary choice is used when there is a tie. The univariate basis functions we use are 
\begin{equation}
\label{eq:cosbasis}
    \psi_1(x) = 1,\quad \psi_j = \sqrt{2}\cos((j-1)\pi x), \text{ for }j\geq 2.
\end{equation}

In Figure~\ref{exampleB1}, we compare the statistical performance of SieveSGD with other popular methods in statistics and computer science communities. The Sobolev tensor product kernel we use there is $K(\mathbf{s},\mathbf{t}) = \prod_{k=1}^p \left(1+\min\{\mathbf{s}^{(k)},\mathbf{t}^{(k)}\}\right)$. The RKHS corresponding to this kernel is the tensor product of univariate Sobolev space on $[0,1]$. Like many other learning methods trained with stochastic gradient descent, it is possible to have several pass over the data set to achieve better generalization ability. While processing the data multiple times, we continue to increase the number of basis function of Sieve-SGD. That is, after 5 passes we use $cp(5\times 10^5)^{1/(2s+1)}$ basis functions. This strategy is not feasible for kernel SGD methods so for include relevant results for comparison.

\begin{figure}[!htbp]
\includegraphics[width=\textwidth]{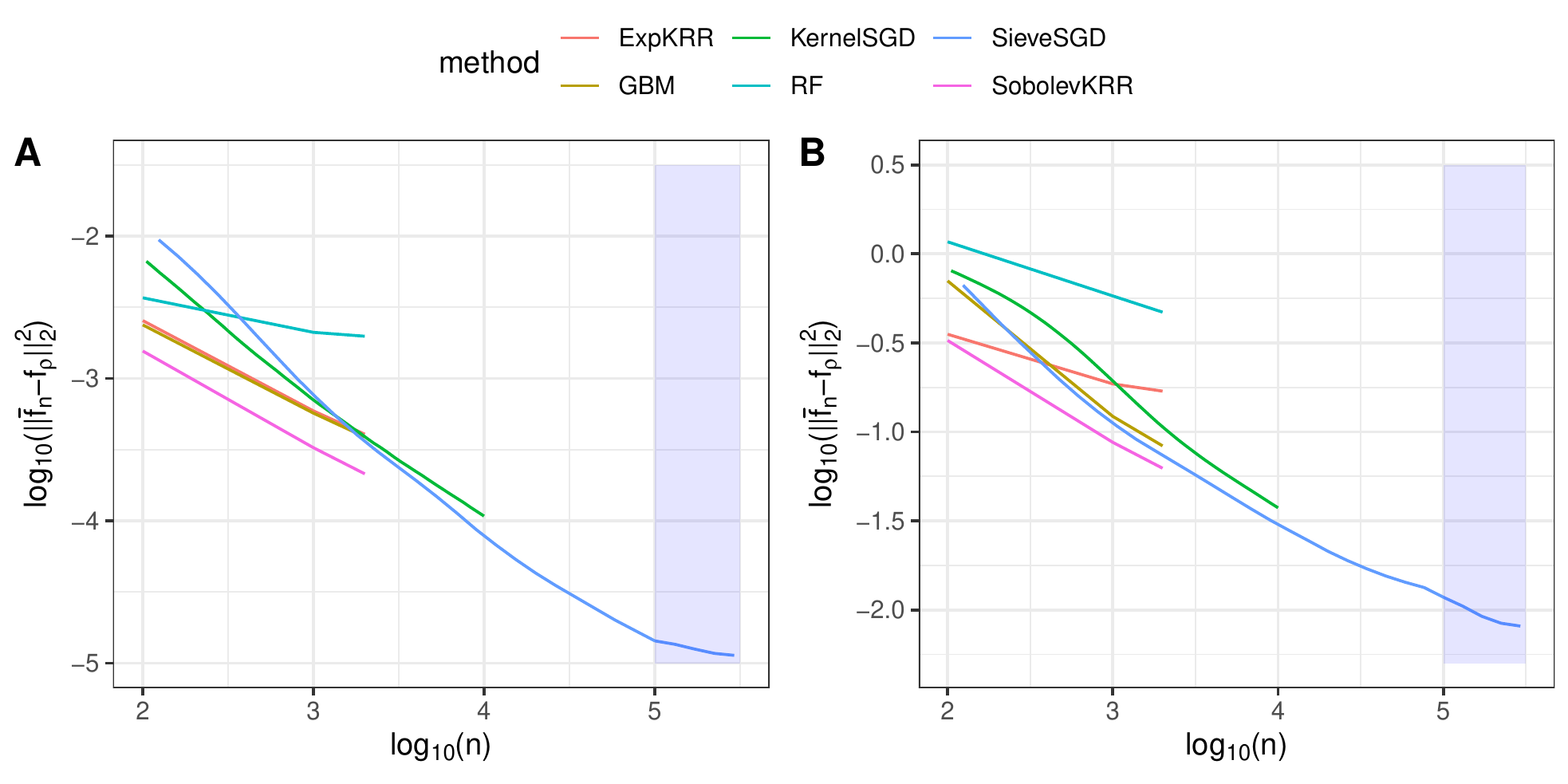}
\caption{Multivariate numerical examples of applying SieveSGD. Other benchmark methods: ExpKRR, kernel ridge regression with Gaussian kernel; KernelSGD, \cite{dieuleveut2016nonparametric} with tensor product Sobolev kernel; GBM, gradient boosting machine; RF, random forest; SobolevKRR, kernel ridge regression with tensor product Sobolev kernel. We present the result of each method under oracle hyperparameters (best testing error). The shaded area corresponds to the second to fifth pass of SieveSGD over the same training data ($10^5$ unique observations). (A) $p=2$ (B) $p = 10$.}
\label{exampleB1}
\end{figure}

Now we present some intuition behind our choice of ordering the multi-variate tensor product basis functions $\psi_{\mathbf{j}}$. While there are many (equivalent) ways to arrive at such a choice, we choose to use some basic theory in RKHS for easier exposition. While RKHSs are usually defined with the reproducing property (as we will see very soon in Theorem~\ref{RKHSdefines}), there is another characterization of the same function space which we will introduce now to help our exposition.

Let $\nu$ be a positive Borel measure on the feature domain $\mathcal{X}$ that has support$(\nu)=\mathcal{X}$, i.e. $\nu(U)>0$ for any nonempty open subset $U$ of $\mathcal{X}\subset \mathbb{R}$. Let $K$ be a Mercer kernel. We use $(\lambda_j,\psi_j)$ to denote the eigensystem of the covariance operator $T_{\nu}$ of $K$ ($\psi_j$'s are orthonormal w.r.t. $\nu$). It is known \cite{cucker2002mathematical} that the following Hilbert space is the same as the one we described in Theorem~\ref{RKHSdefines}.
\begin{theorem}
\label{theorem:equivalentdef}
Define a Hilbert space
\begin{equation}
\mathcal{H} = \left\{f \in L^2_{\nu}\mid f=\sum_{k=1}^{\infty} \theta_j \psi_{j}\ \text {with} \sum_{j=1}^{\infty} \left(\frac{\theta_j}{\sqrt{\lambda_j}}\right)^2< \infty\right\}    
\end{equation}
equipped with inner product:
\begin{equation}
\langle f, g\rangle_{\mathcal{H}}=\sum_{j=1}^{\infty} \frac{\theta_j \eta_j}{\lambda_{j}}
\end{equation}
for \(f=\sum_{j=1}^{\infty} \theta_j \psi_{j}\) and \(g=\sum_{j=1}^{\infty} \eta_{j} \psi_{j}\), where $\theta_j = \langle f, \psi_j\rangle_{L^2_{\nu}}$, $\eta_j = \langle g, \psi_j\rangle_{L^2_{\nu}}$. Then $(\mathcal{H},\langle \cdot, \cdot\rangle_{\mathcal{H}})$ is the reproducing Hilbert space of kernel $K$. And we have the following Mercer expansion of the kernel $K$:
\begin{equation}
    K(x,z) = \sum_{j=1}^{\infty} \lambda_j\psi_j(x)\psi_j(z)
\end{equation}
where the convergence is uniform and absolute.
\end{theorem}

If $\lambda_j = j^{-2s}$, then a ball in the kernel's RKHS is the same as the ellipsoid 
\begin{equation}
    W(s,Q,\{\psi_j\}) = \left\{f\in L^2_{\nu}\mid  f = \sum_{j=1}^{\infty}\theta_j\psi_j, \sum_{j=1}^{\infty}(\theta_j j^s)^2 \leq Q^2\right\}.
\end{equation}
If we consider the two-dimensional tensor product kernel $\tilde{K}:\mathbb{R}^2\times\mathbb{R}^2\rightarrow \mathbb{R}$ constructed from $K$, that is
\begin{equation}
    \tilde{K}(\mathbf{x},\mathbf{z}) = K(\mathbf{x}^{(1)},\mathbf{z}^{(1)})K(\mathbf{x}^{(2)},\mathbf{z}^{(2)})
\end{equation}
where $\mathbf{x}^{(j)}$ is the $j$-th component of $\mathbf{x}\in\mathbb{R}^2$. It is known (\cite{wainwright2019high} Section~12.4.2) that
\begin{equation}
    \tilde{K}(\mathbf{x},\mathbf{z}) = \sum_{j=1}^{\infty}\sum_{k=1}^{\infty}(jk)^{-2s}\psi_j(\mathbf{x}^{(1)})\psi_j(\mathbf{z}^{(1)})\psi_k(\mathbf{x}^{(2)})\psi_k(\mathbf{z}^{(2)})
\end{equation}
and a ball in the RKHS of $\tilde{K}$ takes the form
\begin{equation}
\label{eq:tensorellipsoid}
    \tilde{W} = \left\{f\in L^2_{\nu}\otimes L^2_{\nu}\mid f(x) = \sum_{j,k=1}^{\infty}\theta_{jk}\psi_j\left(\mathbf{x}^{(1)}\right)\psi_k\left(\mathbf{x}^{(2)}\right), \sum_{j,k=1}^{\infty}\left(\theta_{jk} (jk)^s\right)^2 \leq Q^2\right\},
\end{equation}
or equivalently, the eigenvalues are $\prod_{k=1}^2 \left(\mathbf{j}^{(k)}\right)^{-2s}$, $\mathbf{j}\in(\mathbb{N}^+)^2$. Accroding to \eqref{eq:tensorellipsoid}, we would intuitively expect $\theta_{jk}$ to be smaller when the \textit{product} of the index vector is larger. 

When the univariate RKHS is a Sobolev space, estimating with a tensor product kernel is essentially assuming the true regression function is (or can be well-approximated by a function) in the tensor product space of Sobolev space. The latter is also characterized as a Sobolev space with (dominating) mixed derivatives \cite{schmeisser2007recent}. In statistics, such a model has been studied under the name of nonparametric Tensor product ANOVA~\cite{lin2000tensor}. Although methods engaging with hyperbolic cross have been actively studied in numerical analysis in the past decade, there are few works adopting such an idea into statistics. Sobolev spaces with mixed derivatives are not homogeneous spaces in the sense that they contain functions of different smoothness in different directions. Specifically, functions in such spaces can be less smooth along the directions of coordinate axis than other directions. This can be useful in the case when the features as a strong ``main effect" on the outcome and a weaker ``interaction effect" (in the sense of \cite{lin2000tensor}).  

\subsection{Additive model and Sieve-SGD}
In the main text Section~\ref{subsection:multivaraite}, we described the nonparametric additive model and how to use Sieve-SGD to estimate it. We simplified things by omitting the intercept term to streamline exposition. The additive model with intercept is given by
\begin{equation}
\label{eq:additivewithintercept}
    f_{\rho}(\mathbf{x}) = \beta^0 + \sum_{k=1}^p f_{\rho,k}\left(\mathbf{x}^{(k)}\right)
\end{equation}
for some $\beta^0\in\mathbb{R}$ and  $f_{\rho,k}\in W_k(s_k,Q_k,\{\psi_{jk}\})$ for some centered $\psi_{jk}$ $\left(\int \psi_{jk}(x)d\nu(x) = 0\right)$ (for example the functions in \eqref{eq:cosbasis}). For the additive model with intercept \eqref{eq:additivewithintercept}, the updating rule \eqref{submainform} of Sieve-SGD could be replaced by a two-step procedure:
\begin{equation}
\begin{aligned}
    \hat f_{i}
    & = \hat f_{i-1} + \gamma_i\left(Y_i - \hat \beta_{i-1}^0  - \sum_{k=1}^p\hat f_{i-1,k}\left(\mathbf{X}_i^{(k)}\right)\right)\sum_{k=1}^p\sum_{j=1}^{J_{ik}}j^{-2s_k}\psi_{jk}\left(\mathbf{X}_i^{(k)}\right)\psi_{jk}\\
    \hat \beta^0_{i} &= \hat \beta_{i-1}^0 + \gamma_i\left(Y_i - \hat \beta_{i-1}^0  - \sum_{k=1}^p\hat f_{i-1,k}\left(\mathbf{X}_i^{(k)}\right)\right)
\end{aligned}
\end{equation}
here $J_{ik}$ is the truncation level of the $k$-th covariate when the sample size is equal to $i$; and $\hat f_{i-1,k}$ is the estimate of $f_{\rho,k}$. After applying Polyak averaging (averaging $\hat\beta^0_{i} + \hat f_{i}$ with previous estimates), we will get the Sieve-SGD estimate of $f_{\rho}$.

\section{Proof of Lemma~6.2}
\label{app:RKHS}
In this section, we will prove Lemma \ref{maintextlemma}, together with results regarding the spectrum of some related operators. To this end, we need to prepare the reader by reminding them about some established results and ideas in the literature. In this section we will
\begin{itemize}
    \item Define a Reproducing Kernel Hilbert Space (RKHS) formally;
    \item Define covariance operators characterized by a kernel and discuss related geometric properties, and; 
    \item Define the entropy of a compact operator and relate it to the eigenvalues of the operator.
\end{itemize}
After all these, we will be ready to give a proof of Lemma~\ref{maintextlemma}.\\
In this section, we need to distinguish functions (in the RKHS) and their equivalence class (in $L^2$ spaces) for a more rigorous discussion. For a given measure $\mu$ on $\mathcal{X}\subset \mathbb{R}^p$, we use $\mathcal{L}^2_{\mu}$ to denote the Hilbert space of all $\mu$-square integrable functions. The $L^2_{\mu}$ spaces should be understand as the quotient spaces of $\mathcal{L}^2_{\mu}$ under the equivalence relation:
\begin{equation}
    f = g \Leftrightarrow \int_{\mathcal{X}}\left(f(\tau) - g(\tau)\right)^2 d\mu(\tau) = 0.
\end{equation}
For a function $g\in \mathcal{L}^2_{\mu}$, we use $[g] = [g]_{\mu}\in L^2_{\mu}$ to denote its equivalence class. This mathematical framework \cite{steinwart2012mercer} we present here allows the discussion when the measure $\mu$ does not have a full-support over $\mathcal{X}$ (which is weaker than our Assumption A2), or when the RKHS is not dense in $\mathcal{L}^2_{\mu}$. 

\subsection{Mercer kernel and RKHS} We first introduce the definition of a Mercer kernel and its corresponding RKHS.
\begin{definition}[Mercer kernel]
A symmetric bivariate function $K:\mathcal{X}\times\mathcal{X}\rightarrow \mathbb{R}$ is positive semi-definite (PSD) if for any $n\geq 1$ and $(x_i)_{i=1}^n\subset\mathcal{X}$, the $n \times n$ matrix $\mathbb{K}$ whose elements are $\mathbb{K}_{ij}= K(x_i,x_j)$ is always a PSD matrix.

A continuous, bounded, PSD kernel function $K$ is called a \emph{Mercer kernel}.
\end{definition}
In Assumption A3 of the main manuscript, we assumed $\{\psi_j\}$ is a set of bounded, continuous functions in $\mathcal{L}^2_{\nu}$. 
For each $J\in \mathbb{N}^+ \cup\{\infty\}$ and $\omega > 0.5$, we can show that the bivariate functions 
\begin{equation}
\label{eq:truncatedkernel}
  K_{J}^{\omega}(s,t) := \sum_{j=1}^{J}j^{-2\omega} \psi_j(s)\psi_j(t),\quad 
\end{equation}
are Mercer kernels. We also use $K(s,t) = K_{\infty}^{s}(s,t)$ to denote the canonical (untruncated) kernel in our analysis.

It is well-known that for any Mercer kernel, there is a unique associated Hilbert space $(\mathcal{H}_K,\langle\cdot,\cdot\rangle_{K})$ which has the so-called reproducing property. The following theorem formally defines such a Hilbert space and states its uniqueness.

\begin{theorem}[\cite{cucker2002mathematical}]
\label{RKHSdefines}
For a Mercer Kernel $k:\mathbb{X}\times\mathbb{X}\rightarrow\mathbb{R}$, there exists an unique Hilbert Space $(\mathcal{H}_k,\langle \cdot,\cdot \rangle_{k})$ of functions on $\mathbb{X}$ satisfying the following conditions. Let $k_x:z\mapsto k(x,z)$:
\begin{enumerate}
    \item For all $x\in \mathbb{X}$, $k_x\in \mathcal{H}_k$.
    \item The linear span of $\{k_x\ |\ x\in \mathbb{X}\}$ is dense (w.r.t $\|\cdot\|_{k}$) in $\mathcal{H}_k$
    \item For all $f\in \mathcal{H}_k,x\in\mathbb{X}$, $f(x) = \langle f,k_x\rangle_{k}$
\end{enumerate}
We call this Hilbert space the \emph{Reproducing kernel Hilbert space (RKHS)} associated with kernel $k$ .
\end{theorem}
Note that in the above definition, we did not mention any measures on $\mathbb{X}$. The RKHS $\mathcal{H}_{k}$ of a kernel $k$ is defined independently as a Hilbert space (a complete linear space equipped with an inner product).

The inner product in Theorem~\ref{RKHSdefines} is implicitly defined and appear to be quite abstract. However, there is an equivalent definition of the RKHS corresponding to $K$ in which our Sobolev ellipsoid assumption appears:
\begin{theorem}[p.37,Theorem 4 in \cite{cucker2002mathematical}]
\label{theorem:equivalentdef}
The Hilbert space defined in Theorem~\ref{RKHSdefines} is identical (same function class with the same inner product) to the following Hilbert space $\mathbb{H}_K$.
\begin{equation}
    \mathbb{H}_K=\left\{f \in \mathcal{L}_{\nu}^{2}(X) \mid f=\sum_{j=1}^{\infty} a_{j} \psi_{j} \quad \text { with } \quad \sum_{j=1}^{\infty} \left(j^sa_j \right)^2 < \infty\right\}
\end{equation}
Equipped with the inner product:
\begin{equation}
    \langle f, g\rangle_{K}=\sum_{j=1}^{\infty} j^{2s}a_j b_j
\end{equation}
for $f=\sum a_{j} \psi_{j}, \text{ and }g=\sum b_{j} \psi_{j}$.
\end{theorem}

It is direct to check our assumption A3 is the same as assuming the conditional mean $f_{\rho}$ belongs to a ball of radius $Q$ in the above constructed RKHS.

In this section, we consider the RKHSs related to kernels $K_J$ with $\mathbb{X} = \mathcal{X}$, where $\mathcal{X}$ is the support of $\nu$. Under our assumption A2 and A3, the functions in $\mathcal{H}_{K}$ are all square-integrable.
We also define the identity mapping (w.r.t. measure $\mu$) $\text{id}_{\mu}$ as
\begin{equation}
\begin{aligned}
        \text{id}_{\mu}:\mathcal{H}_K & \rightarrow L^2_{\mu}\\
        g&\mapsto [g]_{\mu}.
\end{aligned}
\end{equation}

\subsection{Covariance operators}
\label{app:covariance}
In Section~\ref{app:covariance} and \ref{section:entropy and spectrum}, we engage with the RKHS of the canonical kernel $K$. Once the properties of its related operators are clear, we can directly generate our analysis techniques to other truncated kernels $K_{J}^{\omega}$. Recall our definitions of $T_X$ and $T_{\nu}$ in the main text:
\begin{equation} 
\begin{aligned}
T_{X} : L_{\rho_{X}}^{2} & \rightarrow  L_{\rho_{X}}^{2} \\ 
g & \mapsto \int_{\mathcal{X}} g(\tau) K(\tau,\cdot), d\rho_X(\tau) 
\end{aligned}
\end{equation}
and 
\begin{equation} 
\begin{aligned}
T_{\nu} : L^{2}_{\nu} & \rightarrow  L^{2}_{\nu} \\ 
g & \mapsto \int g(\tau) K(\tau,\cdot) d\nu(\tau).
\end{aligned}
\end{equation}

Now we state several basic properties of $T_X$. Similar properties also hold for $T_{\nu}$ and can be verified much easier without abstract analysis. For proofs of Lemma~\ref{lemma:easyprop} and more properties, see \cite[Section~2 \& 3]{steinwart2012mercer}.
\begin{lemma}
\label{lemma:easyprop}
Under Assumptions A2, A3 in the main text:
\begin{itemize}
    \item The operator $T_X$ is bounded, self-adjoint, positive.
    \item There exists a at most countable set of functions $\{\phi_j\}\subset \mathcal{H}_K$, $j\in\mathcal{J}$ and a at most countable sequence of positive numbers $\lambda_j$ (decreasingly ordered) such that
    \begin{equation}
        T_{X}(g) = \sum_{j\in \mathcal{J}}\lambda_{j}\langle g,\left[\phi_j\right]_{\rho_X}\rangle_{L^{2}_{\rho_X}}\left[\phi_j\right]_{\rho_X},\quad \text{for }g\in L^2_{\rho_X}.
    \end{equation}
    \item The $\{[\phi_j]_{\rho_X}\}$ above is an orthonormal system in $L^2_{\rho_X}$ and $\{\sqrt{\lambda_j}\phi_j\}$ is an orthonormal system in $\mathcal{H}_K$. Therefore, $(\lambda_j,[\phi]_j)$ is an eigensystem of $T_X$:
    \begin{equation}
        T_X([\phi_j]_{\rho_X}) = \lambda_j [\phi_j]_{\rho_X}
    \end{equation}
    \item $T_X$ is a trace class operator, i.e. $\sum_j \lambda_j <\infty$.
\end{itemize}
\end{lemma}

Under the assumptions A2, A3, we can actually say more about the properties of $\phi_j$. But we list them as a separate lemma since they are not necessarily true when we discuss truncated kernels later.

\begin{lemma}[Theorem~3.1, \cite{steinwart2012mercer}]
\label{lemma:completecondition}
Under the same assumptions as in Lemma~\ref{lemma:easyprop}. Let $\phi_j$ and $[\phi_j]_{\rho_X}$ denote the orthonormal systems in Lemma~\ref{lemma:easyprop}. Then
\begin{itemize}
    \item The family $[\phi_j]_{\rho_X}$ is an orthonormal basis of $L^2_{\rho_X}$.
    \item The family $\sqrt{\lambda_j}\phi_j$ is an orthonormal basis of $\mathcal{H}_K$.
\end{itemize}
\end{lemma}


Now we define several operators related to $T_{X}$ that will facilitate our analysis of the spectrum of it.
\begin{definition}
\label{def:SX}
Under the same assumptions as in Lemma~\ref{lemma:easyprop}, with the same $\{\lambda_j\}$ and $\{\phi_j\}$:
\begin{itemize}
    \item We define the $r$-th power of $T_X$ as:
    \begin{equation}
    \begin{aligned}
         T^{r}_X:L^2_{\rho_X}&\rightarrow L^2_{\rho_X}\\
         g &\mapsto \sum_j \lambda_j^r \langle g,\left[\phi_j\right]_{\rho_X}\rangle_{L^{2}_{\rho_X}}\left[\phi_j\right]_{\rho_X},\quad \text{for }g\in L^2_{\rho_X},
    \end{aligned}
    \end{equation}
    in this work, we are most interested in the square-root of $T_X$, i.e. $T_{X}^{1/2}$.
    \item Define the operator $S_X^{1/2}$ as:
    \begin{equation}
    \begin{aligned}
    \label{eq:S_X^1/2}
         S^{1/2}_X:L^2_{\rho_X}&\rightarrow \mathcal{H}_K\\
         g &\mapsto \sum_j \lambda_j^{1/2} \langle g,\left[\phi_j\right]_{\rho_X}\rangle_{L^{2}_{\rho_X}}\phi_j,\quad \text{for }g\in L^2_{\rho_X}
    \end{aligned}
    \end{equation}
\end{itemize}
\end{definition}

The operator $S_X^{1/2}$ has a very importance geometric properties: it preserves distance between two subspaces of $L^2_{\rho_X}$ and $\mathcal{H}_K$, as stated in the following lemma.
\begin{lemma}[\cite{steinwart2012mercer},Theorem~2.11]
\label{lemma:isosubspaces}
Under the same assumptions as in Lemma~\ref{lemma:easyprop}, let $S_X^{1/2}$ be the operator defined in \eqref{eq:S_X^1/2}. Then 
\begin{itemize}
    \item $S_{X}^{1/2}$ is bijective between $\overline{\text{span}\{ [\phi_j]_{\rho_X}, j\in \mathcal{J}\}} =  L^2_{\rho_X}$ and $\overline{\text{span}\{ \phi_j, j\in \mathcal{J}\}}  =  \mathcal{H}_K$ and,
    \item $\|S_X^{1/2}(g)\|_K = \|g\|_{L^2_{\rho_X}}$, for $g\in \overline{\text{span}\{ [\phi_j]_{\rho_X}, j\in \mathcal{J}\}}$.
\end{itemize}
That is, $S_{X}^{1/2}$ is an isometric isomorphism between $L^2_{\rho_X}$ and $\mathcal{H}_K$.
\end{lemma}

It is direct to check the following lemma by the equivalent definition of the RKHS (Theorem~\ref{theorem:equivalentdef}).
\begin{lemma}
\label{lemma:isonu}
Define $S_{\nu}^{1/2}$ similarly for the operator $T_{\nu}$:
\begin{equation}
    \begin{aligned}
         S^{1/2}_{\nu}:L^2_{\nu}&\rightarrow \mathcal{H}_K\\
         g &\mapsto \sum_j j^{-s} \langle g,\left[\psi_j\right]_{\nu}\rangle_{L^{2}_{\nu}}\psi_j,\quad \text{for }g\in L^2_{\nu}
    \end{aligned}
    \end{equation}
Then
\begin{itemize}
    \item $S_{\nu}^{1/2}$ is bijective between $L^2_{\nu}$ and $\mathcal{H}_K$ and,
    \item $\|S_{\nu}^{1/2}(g)\|_K = \|g\|_{L^2_{\nu}}$, for $g\in L^2_{\nu}$.
\end{itemize}
\end{lemma}

\subsection{Entropy of an operator and its spectrum}
\label{section:entropy and spectrum}
The above Lemma~\ref{lemma:isosubspaces} and Lemma~\ref{lemma:isonu} is one set of the elements we are going to use in the proof of Lemma~\ref{maintextlemma}. Another important part of our proof is the correspondence between the spectrum of an operator and its metric entropy. We believe that these results might help show connections between proof methods regarding nonparametric problems that using the spectrum of operators \cite{dieuleveut2016nonparametric,yuan2010reproducing} and those using metric entropy \cite{wainwright2019high,geer2000empirical}.\\

We first define the metric entropy of an operator. There is a correspondence between our definition and the "standard" definition of metric entropy for compact sets. In the following we will use $B_{E}$ to denote the unit ball of a function space $E$.
\begin{definition}[Entropy of an operator]
For $k\geq 1$ we define the $k$-th entropy number of a metric space $S$ to be
\begin{equation}
e_{k}(S)=\inf \left\{\epsilon>0 | \exists \text { closed balls } D_{1}, \ldots, D_{2^{k}-1} \text { with radius } \epsilon \text { covering } S\right\}
\end{equation}
If  \(T: E \rightarrow F\) is a linear map, then we define the $k$-th entropy number of $T$ as
\begin{equation}
e_{k}(T)=e_{k}\left(T\left(B_{E}\right)\right)
\end{equation}
\end{definition}
The first result bounds the entropy of an operator by its eigenvalues. 
\begin{theorem}
\label{eigen2entropy}
Let \(\lambda_{1} \geq \lambda_{2} \geq \ldots \geq \lambda_{j} \geq \ldots \geq 0\) be a sequence of real numbers, $w_j$ be a sequence of elements in $l^2$ (space of square-summable sequences). Consider the operator defined by
\begin{equation}
\begin{aligned}
T: l^{2} & \rightarrow l^{2} \\
(\xi_1,\xi_2,...,\xi_j,...) & \mapsto (\lambda_1\xi_1,\lambda_2\xi_2,...,\lambda_j\xi_j,...)
\end{aligned}
\end{equation}
If \(\lambda_{j} \leq C j^{-s}\) for some $C,s$ and all \(j \geq 1\), then for all $j\geq 2$
\begin{equation}
e_{j}(T) \leq 12Cs^s j^{-s}
\end{equation}
\end{theorem}
For proof, see Proposition 9 in \cite[Appendix~A]{cucker2002mathematical}. The proof there uses Proposition~1.3.2 in \cite{carl_stephani_1990}, which engages with $l^2$ spaces as well. Theorem~\ref{eigen2entropy} is good enough for our purpose because we can diagonalize the covariance operators of interest (Lemma~\ref{lemma:easyprop}). For a general result of the same nature as Proposition~1.3.2, one can use Proposition~3.4.2 combined with Proposition~4.4.1 in \cite{carl_stephani_1990}.

We also state a result in the other direction: Carl's inequality, see \cite[Theorem~4]{CARL1981290}.

\begin{theorem}
\label{entropy2eigen}
Let $E$ be a  Banach space and let $T:E\rightarrow E$ be a linear
compact operator. Then 
\begin{equation}
\left|\lambda_{j}(T)\right| \leq \sqrt{2} e_{j}(T)
\end{equation}
where $\lambda_j(T)$ is the non-increasing sequence of eigenvalues of $T$. 
\end{theorem}

Now we are ready to prove Lemma~\ref{maintextlemma} in the main text.
\begin{proof}[Proof of Lemma~\ref{maintextlemma}] 
We use $B_X$ to denote the unit ball in $L^2_{\rho_X}$ and $B_{\nu}$ for the unit ball in $L^2_{\nu}$. The proof of this lemma can be divided into three steps.\\
\textbf{Step 1: bound entropy by spectrum.} By the definition of the power of an operator, the eigenvalue of $T^{1/2}_{\nu}$ is $j^{-s}$ (because the eigenvalue of $T_{\nu}$ is $j^{-2s}$). And the corresponding eigenfunctions are $\psi_j$. For any function $f\in L^2_{\nu}$, there is a unique sequence $(f_j)_{j=1}^{\infty}\in l^2$ such that $f = \sum_{j=1}^{\infty} f_j\psi_j$.
Apply Theorem \ref{eigen2entropy} to $T^{1/2}_{\nu}$ we know:
\begin{equation}
    e_j(T^{1/2}_{\nu})\leq 12s^sj^{-s}
\end{equation}
\textbf{Step 2: relate the entropy of operators.} Now we are going to show that $e_j(T^{1/2}_X)\leq ue_j(T^{1/2}_{\nu})$.To investigate the entropy of an operator, we only need to look at the entropy of $T^{1/2}_X(B_X)$ and $T^{1/2}_{\nu}(B_{\nu})$. By Lemma~\ref{lemma:isosubspaces} and Lemma~\ref{lemma:isonu}, we know $S_{X}^{1/2}(B_X)= S_{\nu}^{1/2}(B_{\nu}) = B_{\mathcal{H}_K}$. When measuring the entropy of $T^{1/2}_X(B_X)$ we need to use $\text{id}_{X}$ to map it to $L^2_{\rho_X}$, but for $T^{1/2}_{\nu}(B_{\nu})$ we need to use $\text{id}_{\nu}$ to map it to $L^2_{\nu}$.

By the assumption A2 of $\rho_X$, for any $f\in L^2_{\rho_X} = L^2_{\nu}$, we have
\begin{equation}
    \int [f]_{\rho_X}^2(x) d\rho_X(x) \leq u^2\int [f]_{\nu}^2(x) d\nu(x)
\end{equation}
This means we can use the center of an $\epsilon$-cover of $T^{1/2}_{\nu}(B_{\nu})$ to construct a $u\epsilon$-cover of $T^{1/2}_{X}(B_X)$ using the same center points. By the definition of entropy,
\begin{equation}
    e_j(T^{1/2}_X)\leq ue_j(T^{1/2}_{\nu})
\end{equation}
\textbf{Step 3: bound spectrum by entropy.} We use Theorem \ref{entropy2eigen} to translate our bound from entropy back to the spectrum of our operator. For the eigenvalue $\lambda_j$ of $T_X$, we have
\begin{equation}
    \lambda_j = \left(\lambda_j^{1/2}\right)^2
    \leq \left(\sqrt{2} e_j(T^{1/2}_X)\right)^2
    \leq \left(\sqrt{2} ue_j(T^{1/2}_{\nu})\right)^2\leq \left(12\sqrt{2} us^sj^{-s}\right)^2
\end{equation}
This concludes our claim that the eigenvalues of $T_X$ satisfy $\lambda_j \leq 288 u^2 s^{2s}j^{-2s} = O(j^{-2s})$.

Similarly, we can show $j^{-2s}\leq 288 \ell^2 s^{2s}\lambda_j$ by going in the other direction. So we conclude that $\lambda_j=\Theta(j^{-2s})$.
\end{proof}


\subsection{Technical Results}
In this subsection we will present several results needed in the proof of Theorem~\ref{maintheorem}. Since they are related to the spectrum of covariance operators, we put them here for rather than the technical section of Appendix~\ref{section:mainproof}. 

For a fixed $1\leq J\leq \infty$ and $\omega > \frac{1}{2}$, we consider the kernel $K_J^{\omega}(s,t) = \sum_{j=1}^J j^{-2\omega}\psi_j(s)\psi_j(t)$. By Theorem~\ref{RKHSdefines}, there is an unique Hilbert space with the reproducing properties. According to Theorem~\ref{theorem:equivalentdef}, we have the hierarchy relation that $\mathcal{H}_{K_J^{\omega}}\subset \mathcal{H}_{K_I^{\omega}}$ for $I\geq J$, equipped with the inner product $\langle f, g\rangle_K = \sum_{j=1}^{I} j^{2\omega} a_jb_j$ for $f = \sum a_j\psi_j\text{ and }g\sum b_j\psi_j$. Similarly, we can define the covariance operators w.r.t. $\rho_X$ and $\nu$:
\begin{equation} 
\begin{aligned}
T_{X,J}^{\omega} : L_{\rho_{X}}^{2} & \rightarrow  L_{\rho_{X}}^{2} \\ 
g & \mapsto \int_{\mathcal{X}} g(\tau) K_J^{\omega}(\tau,\cdot), d\rho_X(\tau) 
\end{aligned}
\end{equation}
and 
\begin{equation} 
\begin{aligned}
T_{\nu,J}^{\omega} : L^{2}_{\nu} & \rightarrow  L^{2}_{\nu} \\ 
g & \mapsto \int g(\tau) K_J^{\omega}(\tau,\cdot) d\nu(\tau).
\end{aligned}
\end{equation}

Similar to Lemma~\ref{lemma:easyprop}, we can diagonalize $T_{X,J}^{\omega}$ with an eigensystem $(\lambda_{J,j}^{\omega},\phi_{J,j}^{\omega})$. Different from $\{\phi_j\}$, $\{\phi_{J,j}^{\omega}\}$ is a basis of $\mathcal{H}_{K_J^{\omega}}$ but $\{[\phi_{J,j}^{\omega}]_{\rho_X}\}$ is not a basis of $L^2_{\rho_X}$. We formally state it in the following lemma.

\begin{lemma}[Theorem~3.1, \cite{steinwart2012mercer}]
\label{lemma:systemandbasis}
Under assumptions A2, A3. Let $(\lambda_{J,j}^{\omega},\phi_{J,j}^{\omega})$ denote the eigensystem of $T_{X,J}^{\omega}$. Then
\begin{itemize}
    \item The family $[\phi_{J,j}^{\omega}]_{\rho_X}$ is an orthonormal system of $L^2_{\rho_X}$.
    \item The family $\sqrt{\lambda_{J,j}^{\omega}}\phi_{J,j}^{\omega}$ is an orthonormal basis of $\mathcal{H}_{K_J^{\omega}}$.
\end{itemize}
\end{lemma}

Related to Lemma~\ref{lemma:systemandbasis}, the mapping $\left( S_{X,J}^{\omega}\right)^{1/2}$ is not an isomorphism between $L^2_{\rho_X}$ and $ \mathcal{H}_{K_J^{\omega}}$ -- it has a non-trivial kernel space.

\begin{lemma}[\cite{steinwart2012mercer},Theorem~2.11]
\label{lemma:generalisosubspaces}
Under the same assumptions as in Lemma~\ref{lemma:systemandbasis}, define $\left( S_{X,J}^{\omega}\right)^{1/2}$ as
\begin{equation}
    \begin{aligned}
         \left( S_{X,J}^{\omega}\right)^{1/2}:L^2_{\rho_X}&\rightarrow \mathcal{H}_{K_J^{\omega}}\\
         g &\mapsto \sum_j \left(\lambda_{J,j}^{\omega}\right)^{1/2} \langle g,\left[\phi_{J,j}^{\omega}\right]_{\rho_X}\rangle_{L^{2}_{\rho_X}}\phi_{J,j}^{\omega},\quad \text{for }g\in L^2_{\rho_X}
    \end{aligned}
    \end{equation}
Then 
\begin{itemize}
    \item $\left( S_{X,J}^{\omega}\right)^{1/2}$ is bijective between $\overline{\text{span}\{ [\phi_{J,j}^{\omega}]_{\rho_X}, j\in \mathcal{J}\}} \subset  L^2_{\rho_X}$ and $\overline{\text{span}\{ \phi_{J,j}^{\omega}, j\in \mathcal{J}\}}  =  \mathcal{H}_{K_J^{\omega}}$ and,
    \item $\left\|\left( S_{X,J}^{\omega}\right)^{1/2}(g)\right\|_K = \|g\|_{L^2_{\rho_X}}$, for $g\in \overline{\text{span}\{ [\phi_{J,j}^{\omega}]_{\rho_X}, j\in \mathcal{J}\}}$.
\end{itemize}
\end{lemma}

Now we state and proof the main result of this section:
\begin{lemma}
\label{lemma:samejdifferentkernels}
Let $1\leq J \leq \infty$, $\omega > \frac{1}{2}$, assume A2 \& A3. We use $(\lambda_{J,j}^{\omega},\phi_{J,j}^{\omega})$ to denote the eigensystem of $T_{X,J}^{\omega}$ (similarly defined as in Lemma~\ref{lemma:easyprop}). Then
\begin{equation}
\label{eq:generalupperandlower}
    (288\ell^2 \omega^{2\omega})^{-1}j^{-2\omega} \leq \lambda_{J,j}^{\omega}\leq 288u^2 \omega^{2\omega}j^{-2\omega},\quad\text{for }1\leq j\leq J
\end{equation}
\end{lemma}
\begin{proof}
The proof assembles that of Lemma~\ref{maintextlemma}. To investigate the eigenvalues of $T_{X,J}^{\omega}$, we just need to compare the entropy of $\left( S_{X,J}^{\omega}\right)^{1/2}(B_X)$ and $\left( S_{\nu,J}^{\omega}\right)^{1/2}(B_{\nu})$, where $B_{\mu}$ is the unit ball in $L^2_{\mu}$. We know
\begin{equation}
    \left( S_{X,J}^{\omega}\right)^{1/2}(B_X) \stackrel{(1)}{=} \text{unit ball in }\mathcal{H}_{K_J^{\omega}}= \left( S_{\nu,J}^{\omega}\right)^{1/2}(B_{\nu}) 
\end{equation}
In (1) we used the distribution assumption A2, which ensures $\phi_{J,j}^{\omega}$ is a basis of $\mathcal{H}_{K_J^{\omega}}$ (Lemma~\ref{lemma:generalisosubspaces}). 
After embedding the unit ball in $\mathcal{H}_{K_J^{\omega}}$ back to the $L^2$ spaces, we know a $\epsilon$-covering of $T_{\nu,J}^{\omega}(B_{\nu})$ can generate an $u\epsilon$-covering of $T_{X,J}^{\omega}(B_{X})$, which gives the above upper bound by similar argument as in the proof of Lemma~\ref{maintextlemma}. It is also similar to show the lower bound in \eqref{eq:generalupperandlower}, noting that the feature distribution is assumed in A2 to have a strictly positive density.
\end{proof}

\textbf{Note:} the constants show up in \eqref{eq:generalupperandlower} does not depend on the truncation level $J$. Therefore, for a given $\omega$, we can treat the result in Lemma~\ref{lemma:samejdifferentkernels} as an uniform bound which is applicable to all truncation levels.

\section{Proof of Theorem 6.1}
\label{section:mainproof}
In this section we are going to prove the main performance guarantees, results Theorem~\ref{maintheorem}. In our proof, we first split the error into two parts: one part is noiseless and depends only on the the initial bias, the other is due to the noise in our data. We will bound each term separately and choose the learning rate $\gamma_n$ to balance the trade-off. The last part of this section will give some technical lemmas that will be referred to in the proofs of Theorem~\ref{maintheorem}. Some of the proof techniques are taken from \cite{bach2013non} and \cite{dieuleveut2016nonparametric}.
\subsection{Notation}
In this section, the RKHS we are considering is the one associated with mercer kernel $K(s,t)=\sum_{j=1}^{\infty} j^{-2s}\psi_j(s)\psi_j(t)$. We use $\|\cdot\|_K$ and $\langle\cdot,\cdot\rangle_K$ to denote the RKHS-norm and RKHS-inner product. For any elements $g, h \in \mathcal{H}_K,$ we define the operator $g \otimes h$ as a mapping from $\mathcal{H}_K$ to $\mathcal{H}_K$ such that $(g \otimes h)f = \langle f, h\rangle_{K}\ g$. 
As we will show in Lemma \ref{boundEK}, the quantity $\|K_{X_n,J_n}\|_K^2$ is bounded (and this bound only depends on $s$). We use $R^2$ to denote the smallest bound for $\|K_{X_n,J_n}\|_K^2$. And any $\gamma_n$ in this section is assumed to satisfy: $\gamma_0 R^2 < 1$ and $\gamma_n\asymp n^{-\frac{1}{2s+1}}$.\\
We consider a filtration of $\sigma$-algebras $\{\mathcal{F}_n\}$, where $\mathcal{F}_{n}$ is the $\sigma$-algebra generated by $(X_i,Y_i)_{i=1}^{n}$.\\
The sign $\preccurlyeq$ denotes the order between self-adjoint operators over the RKHS. That is, for self-adjoint operators $A,B:E\rightarrow \mathcal{H}_K$, $A\preccurlyeq B$ means $\langle f, (B-A) f\rangle_{K} \geq 0$ for any $f\in \mathcal{H}_K$. Intuitively we can think of $B-A$ as a positive semi-definite matrix in a finite-dimensional space. The expectation of random function/operator should be understood as the Bochner integral, a generalization of the Lebesgue integral where the random element takes value in a Banach space, see \cite{mikusinski2014bochner,cohn2013measure} or Chapter 4 of \cite{berlinet2011reproducing}. The $\zeta(s)$ function that shows up in this section is the Riemann-zeta function $\zeta(s):=\sum_{k=1}^{\infty} k^{-s}$. We use it for simplifying the notation --- it's neat that it shows up, but we do not need any of the exciting (and difficult to show) properties that number theorists/combinatorists are concerned with!
\subsection{Separation of the error}
\label{subsection:separation of error}
In our theoretical analysis of the generalization error, rather than study how $\bar f_n$ converge to $f_{\rho}$, we instead directly study how the difference shrinks to zero. We define 
\begin{equation}
\begin{aligned}
\Delta_{n}&=\hat f_{n}-f_{\rho}\\ 
\bar{\Delta}_{n}&= \frac{1}{n}\sum_{j = 1}^n \Delta_j = \bar{f}_{n}-f_{\rho}
\end{aligned}
\end{equation}
We can relate the $\|\cdot\|_2$-norm (the natural norm shows up in the regression problem) and the $\|\cdot\|_K$-norm (which facilitates our theoretical analysis) using $T_X^{1/2}$. We will be repeatedly using the following equivalence in our proof:
\begin{equation}
   \left\|g\right\|_2^{2} = \left\|T^{1/2}_X g\right\|_{K}^2 =\left\langle g, T_X g\right\rangle_K\quad\text{ for }g\in\mathcal{H}_K
\end{equation}
Similarly,
\begin{equation}
   \left\|g\right\|_2^{2} = \left\|T^{1/2}_{X,J_n} g\right\|_{K}^2 =\left\langle g, T_{X,J_n} g\right\rangle_K\quad\text{ for }g\in\mathcal{H}_K\cap \text{span}\{\psi_1,...,\psi_{J_n}\}
\end{equation}
And we have a recursive relationship for $\eta_n$ based on the recursive relationship for $\hat f_{n}$:
\begin{equation}
\begin{aligned}
\hat f_n & = \hat f_{n-1} - \gamma_n (Y_n - \hat f_{n-1}(X_{n})) K_{X_n,J_n}\\
\Rightarrow\hat f_{n}&=\left(I-\gamma_n T_{X_n,J_n}\right) \hat f_{n-1}+\gamma_n Y_n K_{X_n,J_n}\\
\Rightarrow \hat f_{n}-f_{\rho}&=\left(I-\gamma_n T_{X_n,J_n}\right)\left(\hat f_{n-1}-f_{\rho}\right)+\gamma_n \Xi_{n}
\end{aligned}
\end{equation}
where 
\begin{equation}
    \begin{aligned}
        &T_{X_n,J_n}(f) = f(X_n)K_{X_n,J_n}\\
        &\Xi_{n}=\left(Y_n-f_{\rho}\left(X_n\right)\right) K_{X_n,J_n}
    \end{aligned}
\end{equation}
Thus, we have a recursive formula for $\Delta_n$:
\begin{equation}
\begin{aligned}
\Delta_0 &= -f_{\rho}\\
\Delta_{n} & =\left(I-\gamma_{n} T_{X_n,J_n}\right) \Delta_{n-1}+\gamma_{n} \Xi_{n}
\end{aligned}
\end{equation}
We further decompose $\Delta_n$ into two parts $\Delta_n = \eta_n + \vartheta_n$:
\begin{enumerate}
    \item $\left(\eta_{n}\right) \text { is defined as } $:
\begin{equation}
    \begin{aligned}
            \eta_{0}&= -f_{\rho}\\
            \eta_{n}&=\left(I-\gamma_n T_{X_n,J_n}\right) \eta_{n-1}
    \end{aligned}
\end{equation}
It is the part of $\Delta_n$ due to an initial value not equal to $f_{\rho}$. We note that it does not contain the noise term $\Xi_n$, so the only randomness comes from features $X_n$.
\item The pure noise component \(\left(\vartheta_{n}\right)\) is defined as:
\begin{equation}
    \begin{aligned}
         \vartheta_{0}&=0 \\
         \vartheta_{n}&=\left(I-\gamma_n T_{X_n,J_n}\right) \vartheta_{n-1}+\gamma_n \Xi_{n}
    \end{aligned}
\end{equation}
\end{enumerate}
We can directly show that $\bar\Delta_n = \bar{\eta}_{n}+\bar{\vartheta}_{n}$. By Minkowski's inequality:
\begin{equation}
\label{split}
 \left(E\left[\left\|\bar{\Delta}_{n}\right\|_2^{2}\right]\right)^{1 / 2} 
 \leq \left(E\left[\left\|\bar{\eta}_{n}\right\|_2^{2}\right]\right)^{1/2}
 +\left(E\left[\left\|\bar{\vartheta}_{n}\right\|_2^{2}\right]\right)^{1/2}
\end{equation}
Our job now is just to bound the two terms separately and then choose the correct $\gamma_n$ to minimize the combined bound.
\subsection{Bound on initial condition sub-process}
\label{initialsection}
In this section, we will engage with bounding $\eta_n$, which is the part of error due to the imperfect initialization. 
\begin{proof}[proof of bound on initial value]
By definition,
\begin{equation}
    \eta_n = \eta_{n-1} - \gamma_n \eta_{n-1}(X_n)K_{X_n,J_n}
\end{equation}
We square both sides w.r.t. the RKHS inner product 
\begin{equation}
        \|\eta_n\|_K^2 =  \|\eta_{n-1}\|_K^2 - 2\gamma_n \eta_{n-1}(X_n) \langle \eta_{n-1}, K_{X_n,J_n}\rangle_K + \gamma_n^2 \eta_{n-1}^2(X_n)\left\| K_{X_n,J_n}\right\|_K^2
\end{equation} 
We take the expectation on both sides: conditioned on $\mathcal{F}_{n-1}$ first, then unconditionally.
\begin{equation}
\label{eq:begin1}
\begin{aligned}
        E\|\eta_n\|_K^2 =&  E\|\eta_{n-1}\|_K^2 - 2\gamma_n E\left[\eta_{n-1}(X_n) \langle \eta_{n-1},K_{X_n,J_n}\rangle_K\right] \\
        &+ \gamma_n^2 E\left[\eta_{n-1}^2(X_n)\left\| K_{X_n,J_n}\right\|_K^2\right]\\
        \stackrel{(1)}{\leq} &E\|\eta_{n-1}\|_K^2 - 2\gamma_n E\left[\eta_{n-1}(X_n) \langle \eta_{n-1},K_{X_n,J_n}\rangle_K\right] 
        + \gamma_n E\|\eta_{n-1}\|_2^2\\
    \end{aligned}
\end{equation} 
in $(1)$ we use the fact that for any $n$, $\gamma_n\left\| K_{X_n,J_n}\right\|_K^2 \leq  \gamma_0 R^2 <1$. This is actually how we choose our $(\gamma_n)$. If the $\khatxn$ in the middle term above is  actually $K_{X_n}$, then the whole middle term will become just $E\|\eta_{n-1}\|_2^2$, which makes the following algebra easier. However, since we do not use exactly $K_{X_n}$ but truncated at level $J_n$, we need some extra effort to deal with it.
\begin{equation}
\label{eq:truncate}
\begin{aligned}
        E\|\eta_n\|_K^2 
        &\leq   E\|\eta_{n-1}\|_K^2 - 2\gamma_n E\left[\eta_{n-1}(X_n) \langle \eta_{n-1}, K_{X_n}+K_{X_n,J_n}-K_{X_n}\rangle_K\right]
        + \gamma_n E\|\eta_{n-1}\|_2^2\\
        & = E\|\eta_{n-1}\|_K^2 - \gamma_n E\|\eta_{n-1}\|_2^2 +2\gamma_n E\left[\eta_{n-1}(X_n) \langle \eta_{n-1},  K_{X_n} - \khatxn\rangle_K\right] \\
         & \stackrel{(1)}{\leq} E\|\eta_{n-1}\|_K^2 - \gamma_n E\|\eta_{n-1}\|_2^2 +\frac{1}{2}\gamma_n E\|\eta_{n-1}\|_2^2 + 2\gamma_n E\left[ \langle \eta_{n-1},  K_{X_n} - \khatxn\rangle_K^2\right] \\
    \end{aligned}
\end{equation}
In (1) we use Young's inequality. Now we bound the last term:
\begin{equation}
\begin{aligned}
        E\left[ \langle \eta_{n-1}, K_{X_n} - K_{X_n,J_n}\rangle_K^2\right] 
        & = E\left[ \langle \eta_{n-1},  \sum_{j=J_n+1}^{\infty} j^{-2s}\psi_j(X_n)\psi_j\rangle_K^2\right] \\
        & = E\left[
        \left(\sum_{j=J_n+1}^{\infty} \eta_{n-1,j}\psi_j(X_n)\right)^2\right] \text{\ where\ } \eta_{n-1,j} = \langle \eta_{n-1}, \psi_j\rangle_{L^2_{\nu}}\\
        & \leq u E\left[\int_{\mathcal{X}} \left(\sum_{j=J_n+1}^{\infty} \eta_{n-1,j}\psi_j(x)\right)^2d\nu(x)\right]\\
        & \leq uJ_n^{-2s}E\left[\sum_{j = J_n+1}^{\infty}\left(j^s\eta_{n-1,j}\right)^2\right] \leq uJ_n^{-2s}E\|\eta_{n-1}\|_K^2
    \end{aligned}
\end{equation} 
Continue \eqref{eq:truncate}:
\begin{equation}
\label{eq:conciserecur}
        E\|\eta_n\|_K^2 \leq E\|\eta_{n-1}\|_K^2 - \frac{1}{2}\gamma_nE\|\eta_{n-1}\|_2^2 + 2u\gamma_n J_n^{-2s}E\|\eta_{n-1}\|_K^2      
\end{equation}
Now for each $i$ we have such a recursive relationship for $\|\eta_i\|_K^2$, $\|\eta_{i-1}\|_K^2$ and $\|\eta_{i-1}\|_2^2$. We can sum this from $i=1$ to $n$.
\begin{equation}
    \begin{aligned}
        E\|\eta_n\|_K^2 &\leq E\|\eta_0\|_K^2 - \frac{1}{2}\sum_{i=1}^n \gamma_i E\|\eta_{i-1}\|_2^2 + 2u\sum_{i=1}^n \gamma_iJ_i^{-2s}E\|\eta_{i-1}\|_K^2\\
        \Rightarrow \frac{1}{2}\sum_{i=1}^n \gamma_i E\|\eta_{i-1}\|_2^2 &\leq \|\eta_0\|_K^2 + 2u\sum_{i=1}^n \gamma_iJ_i^{-2s}E\|\eta_{i-1}\|_K^2\\
        \Rightarrow  \frac{\gamma_n}{n}\sum_{i=1}^n E\|\eta_{i-1}\|_2^2 & \leq 2\|\eta_0\|_K^2/n + 4u\sum_{i=1}^n \gamma_iJ_i^{-2s}E\|\eta_{i-1}\|_K^2/n\\
        \stackrel{(1)}{\Rightarrow} E\left\|\frac{1}{n}\sum_{i=1}^n\eta_{i-1}\right\|^2_2 &\leq \frac{2\|f_{\rho}\|_K^2}{n\gamma_n} + \frac{4u\sum_{i=1}^n \gamma_iJ_i^{-2s}E\|\eta_{i-1}\|_K^2}{n\gamma_n}\\
        \Rightarrow \left(E[\|\bar{\eta}_n\|_2^2]\right)^{1/2} &\leq \left(\frac{2\|f_{\rho}\|_K^2+4u\sum_{i=1}^n \gamma_iJ_i^{-2s}E\|\eta_{i-1}\|_K^2}{n\gamma_n}\right)^{1/2}
    \end{aligned}
\end{equation}
Under assumptions A1-A3, we use Lemma \ref{lemma:newboundr} to show that for $J_i\geq i^{\alpha}\log^2 i \vee 1$ for some $\alpha \geq \frac{1}{2s+1}$, then the series $\sum_{i=1}^n \gamma_iJ_i^{-2s}E\|\eta_{i-1}\|_K^2$ is convergent and uniformly bounded, that is, it can be bounded by a constant that does not depend on $n$. Recall that we chose $\gamma_n = \gamma_0 n^{-\frac{1}{2s+1}}$, so we conclude
\begin{equation}
\label{eq:finalboundini}
    \left(E[\|\bar{\eta}_n\|_2^2]\right)^{1/2}  = O\left(n^{-\frac{s}{2s+1}}\right)
\end{equation}
\end{proof}
\subsection{Bound on noise sub-process}
\label{noisesection}
We remind the reader of the definition of our noise sub-process:
\begin{equation}
\begin{aligned}
    \vartheta_0&=0\\
    \vartheta_n&=\left(I-\gamma_{n} T_{X_n,J_n}\right) \vartheta_{n-1}+\gamma_{n} \Xi_{n}
\end{aligned}
\end{equation}
where $\Xi_{n}=\left(Y_{n}-f_{\rho}\left(X_{n}\right)\right) K_{X_n,J_n}$ (we also remind our reader $K_{X_n,J_n}$ is the "truncated kernel" at level $J_n$). Also, $T_{X,J_n}$ and $T_{X_n,J_n}$ are defined as:
\begin{equation}
    \begin{aligned}
        T_{X,J_n}(f)&= \int_{\mathcal{X}} \langle f,K_{x,J_n} \rangle_K K_{x,J_n} d\rho_X(x)\quad\text{population operator}\\
        T_{X_n,J_n}(f) & = \langle f,K_{X_n,J_n} \rangle_K K_{X_n,J_n}\quad\text{random operator}
    \end{aligned}
\end{equation}

\begin{proof}[proof of bound on noise]
We need to define several sequences that are related to $\vartheta_n$ for our technical analysis. The first sequence is:
\begin{equation}
\label{eq:eta0n}
\begin{aligned}
    \eta_0^0&=0 \\
    \eta_n^0&=(I-\gamma_n T_{X,J_n}) \eta_{n-1}^0+\gamma_n \Xi_{n}^0
\end{aligned}
\end{equation}
where $\Xi_n^0=\Xi_n = \left(Y_{n}-f_{\rho}\left(X_{n}\right)\right) K_{X_n,J_n}$.We also define additional sequences, $\eta_n^r$, for each integer $r>0$, by
\begin{equation}
\label{defineetanr}
\begin{aligned}
    \eta_0^r&=0 \\
    \eta_n^r&=(I-\gamma_n T_{X,J_n}) \eta_{n-1}^r+\gamma_n \Xi_{n}^{r}
\end{aligned}
\end{equation}
where \(\Xi_n^r=\left(T_{X,J_n}-T_{X_n,J_n}\right) \eta_{n-1}^{r-1}\). These sequences are easier to analyze than $\vartheta_n$ because the operator in the recursive relationship $(I-\gamma_n T_{X,J_n})$ is a deterministic (population) operator. In contrast the operator in the original $\vartheta_n$ is random. We show in Lemma~\ref{onXinr} that as $r$ increases, the amplitudes of "noise" $\Xi_n^r$ get smaller. Additionally, using the fact that all the sequences $\eta_n^r$ start with $\eta_0^r=0$, we can show that $\eta_n^r$ becomes more concentrated about $0$ for larger $r$. 

We split the noise process $\vartheta_n$ into two parts:
\begin{equation}
    \vartheta_n = \left(\vartheta_n - \sum_{k=0}^r \eta_n^k \right) + \sum_{k=0}^r \eta_n^k
\end{equation}
So its average satisfies
\begin{equation}
    \bar\vartheta_n = \left(\bar\vartheta_n - \sum_{k=0}^r \bar\eta_n^k \right) + \sum_{k=0}^r \bar\eta_n^k
\end{equation}
Here $\bar\eta_n^k$ is the averaged sequence of $\eta_n^k$ (i.e. $\bar\eta_n^k = \frac{1}{n}\sum_{i=1}^n \eta_i^k$). Applying Minkowski’s inequality gives us
\begin{equation}
\label{noisesplit}
 \left(E\|\bar{\vartheta}_n\|_2^{2}\right)^{1/2} 
 \leq \sum_{k=0}^r\left(E\|\bar\eta_n^k\|_2^{2}\right)^{1/2} 
 +\left(E\|\bar\vartheta_n - \sum_{k=0}^r \bar\eta_n^k\|_2^{2}\right)^{1/2} 
\end{equation}
Now we define
\begin{equation}
\label{definealphan}
    \alpha_{n}^{r}=\vartheta_n-\sum_{k=0}^{r} \eta_{n}^{k}
\end{equation}
We will show (Lemma \ref{lemma:exactlyzero}) that for $r\geq n$, we have $\alpha_n^r = \eta_n^r = 0$. We can see the second term in \eqref{noisesplit} is exactly zero when we choose $r\geq n$, that is:
\begin{equation}
    \bar{\vartheta}_{n}-\sum_{k=0}^{r} \bar{\eta}_{n}^{k} = \frac{1}{n}\sum_{i=1}^n \left(\vartheta_i - \sum_{k=0}^r \eta_i^k\right) =  \frac{1}{n}\sum_{i=1}^n \alpha_i^r = 0
\end{equation}

Now our task is just bounding the first term $\sum_{k=0}^r (E\|\bar \eta_n^k \|_2^2)^{1/2}$, we analyze the summation term by term. In Lemma~\ref{lemma:Jnvariance}, we show that:
\begin{equation}
\label{eq:etakbar}
E\left[\left\|\bar{\eta}_{n}^k\right\|_2^{2}\right] = O\left(\gamma_{0}^kR^{2k}  C_{\epsilon}^{2}n^{-\frac{2s}{2s+1}}\right).
\end{equation}
To get this result, we first show in Lemma~\ref{onXinr} that the $\Xi_j^k$ variables are centered and satisfy some moment bounds. With these properties in hand, we prove the above result in Lemma~\ref{lemma:Jnvariance}, with the help of some technical Lemma~\ref{gap1inlemma7} and \ref{gap2inlemma7}

Following \eqref{eq:etakbar}, we have
\begin{equation}
\label{eq:etakbar2}
    \begin{aligned}
        \sum_{k=0}^r \left(E\|\bar{\eta}_n^k\|_2^2\right)^{1/2} 
        &\leq \sum_{k=0}^rC (\gamma_0R^2)^{k/2} C_{\epsilon} n^{-\frac{s}{2s+1}}\\
        & \leq  C C_{\epsilon} n^{-\frac{s}{2s+1}} \sum_{k=0}^{\infty} (\gamma_0 R^2)^{k/2}\\
        & = C C_{\epsilon} n^{-\frac{s}{2s+1}} \frac{1}{1 - \sqrt{\gamma_0 R^2}} = O\left( n^{-\frac{s}{2s+1}}\right)
    \end{aligned}   
\end{equation}
we note that $\gamma_0 R^2 < 1$, which allows us to sum up the geometric series.

Combining the results in \eqref{eq:etakbar2} with $\alpha_n^r = 0$ for $r\geq n$, we can conclude from \eqref{noisesplit} that, for $r\geq n$:
\begin{equation}
\label{final:noise}
 \left(E\left[\left\|\bar \vartheta_n\right\|_2^2\right]\right)^{1/2} 
 \leq \sum_{k=0}^r\left(E\|\bar\eta_n^k\|_2^{2}\right)^{1/2} 
 + 0= O\left(n^{-\frac{s}{2s+1}}\right)
\end{equation}
\end{proof}

\subsection{Combining the bounds}
Plugging the final bounds \eqref{eq:finalboundini} and \eqref{final:noise} back into \eqref{split}, we have the desired result:
\begin{equation}
    E\left[\left\|\bar f_n - f_{\rho}\right\|_2^2\right] = O\left(n^{-\frac{2s}{2s+1}}\right)
\end{equation}

\subsection{Technical Results}
\label{app:technical}
\begin{lemma} 
\label{boundEK} 
There exists $R<\infty$, such that: 
\begin{equation} 
\left\|K_{X_n,J_n}\right\|_K^2 \leq R^2 \end{equation} 
\end{lemma}
for any $J_n\geq 1,X_n\in\mathcal{X}$.
\begin{proof}
By the definition of $\|\cdot\|_K$ we have:
\begin{equation}
    \begin{aligned}
        \left\|K_{X_n,J_n}\right\|_K^2 &=
        \left\|\sum_{j=1}^{J_n} j^{-2s}\psi_{j}(X_n)\psi_j\right\|_K^2\\
        & =\sum_{j=1}^{J_n}\frac{j^{-4s}\psi_j^2(X_n)}{j^{-2s}}\\
        &\leq M^2 \zeta(2s) =: R^2
        \\
    \end{aligned}
\end{equation}
where $\zeta(\cdot)$ is the Riemann-zeta function.
\end{proof}
We use $R^2$ rather than $R$, in the bounds in this lemma because it simplifies calculation where this lemma is applied

\begin{lemma}
\label{lemma:newboundr}
Under A1-A3, and if we choose $J_i\geq i^{\alpha}\log^2 i \vee 1$ for some $\alpha\geq\frac{1}{2s+1}$, $\gamma_i = \gamma_0 i^{-\frac{1}{2s+1}}$, then there exists a number $C$ that does not depend on $n$ such that
\begin{equation}
    \sum_{i=1}^n \gamma_iJ_i^{-2s}E\|\eta_{i-1}\|_K^2 \leq C
\end{equation}
for all $n$.
\begin{proof}
We first show the expectation of $\|\eta_n\|_K^2$ can be uniformly bounded for $J_n$ that increases fast enough. Recall we have the following recursive relationship \eqref{eq:conciserecur}:
\begin{equation}
\begin{aligned}
        E\|\eta_n\|_K^2 
        &\leq E\|\eta_{n-1}\|_K^2 - \frac{1}{2}\gamma_nE\|\eta_{n-1}\|_2^2 + 2u\gamma_n J_n^{-2s}E\|\eta_{n-1}\|_K^2\\
        & \leq (1+2u\gamma_n J_n^{-2s})E\|\eta_{n-1}\|_K^2\\
        \Rightarrow E\|\eta_n\|_K^2& \leq \prod_{i=1}^n(1+2u\gamma_i J_i^{-2s})\|\eta_0\|_K^2
\end{aligned}
\end{equation}
When we take $J_i \geq i^{\alpha}\log^2 i\vee 1$, for some $\alpha \geq \frac{1}{2s+1}$ we have $\gamma_i J_i^{-2s}\leq \gamma_0\left((i\log^2i)^{-1}~\wedge~1\right)$. Therefore $\prod_{i=1}^n(1+2u\gamma_i J_i^{-2s})$ converges as $n\rightarrow\infty$:
\begin{equation}
    \begin{aligned}
        \prod_{i=1}^n(1+2u\gamma_i J_i^{-2s})
        &\leq \prod_{i=1}^n(1+C(i\log^2i)^{-1}\wedge 1) \\
        & = \exp\left(\log(\prod_{i=1}^n(1+C(i\log^2 i)^{-1}\wedge 1)\right)\\
        & = \exp\left(\sum_{i=1}^n\log((1+C(i\log^2 i)^{-1}\wedge 1)\right)\\
        &\leq \exp\left(\sum_{i=1}^n C(i\log^2i)^{-1}\wedge 1\right)
    \end{aligned}
\end{equation}
We note that the series in the exponent of the last line converges. So at this point we know, $E\|\eta_n\|_K^2$ can be uniformly bounded by $K\|\eta_0\|_K = K\|f_{\rho}\|_K$, with a number $K$ that does not depend on $n$. Now it direct to control the quantity of interest:
\begin{equation}
    \sum_{i=1}^n \gamma_iJ_i^{-2s}E\|\eta_{i-1}\|_K^2 \leq K\|f_{\rho}\|\sum_{i=1}^n\gamma_i J_i^{-2s} < \infty,
\end{equation}
when $J_i \geq i^{\alpha}\log^2 i\vee 1$ with $\alpha \geq \frac{1}{2s+1}$.
\end{proof}
\end{lemma}

\begin{lemma}
\label{onXinr}
Assume A1-A3, for any integer $r,n \geq 0$ we have:
\begin{itemize}
        \item $\Xi_n^r$ is $\mathcal{F}_n$ measurable and $\Xi_n^{r}\in \text{span}\{\psi_1,...,\psi_{J_n}\}$,
    \item $E[\Xi_n^r\ |\ \mathcal{F}_{n-1}] = 0$, and
    \item $E\left[\Xi_{n}^{r} \otimes \Xi_{n}^{r}\right] \preccurlyeq \gamma_{0}^{r}R^{2 r}C_{\epsilon}^2T_{X,J_n}$. 
\end{itemize}
Here $\mathcal{F}_n$ is the $\sigma$-algebra generated by $(X_i, Y_i)_{i=1}^n$.
\end{lemma}
\textbf{Note}: because $\gamma_0 R^2 <1$ by our choice of $\gamma_0$, the upper bound on $E\left[\Xi_{n}^{r} \otimes \Xi_{n}^{r}\right]$ is smaller for larger $r$ (the third line in Lemma~\ref{onXinr}).
\begin{proof}
We remind our readers the definitions of $\Xi_{n}^r$ and $\eta_{n}^r$ are given around \eqref{defineetanr}. The first claim is direct. We first note that $\eta^0_n$, $\Xi_n^0$ are both $\mathcal{F}_n$-measurable and belong to $\text{span}\{\psi_1,...,\psi_{J_n}\}$. Then we can show the corresponding properties of $\eta_n^r$, $\Xi_n^r$ by induction.

For the second claim, we calculate directly:
\begin{equation}
\begin{aligned}
        E[\Xi_n^r|\mathcal{F}_{n-1}]
        & = E[(T_{X,J_n}-T_{X_n,J_n})\eta_{n-1}^{r-1}|\mathcal{F}_{n-1}]\\
        & = E[(T_{X,J_n}-T_{X_n,J_n})|\mathcal{F}_{n-1}]\eta_{n-1}^{r-1} = 0
\end{aligned}
\end{equation}

Now we show the last claim. We define
\begin{equation}
    D_k^n =(I-\gamma_{n}T_{X,J_n})(I-\gamma_{n-1}T_{X,J_{n-1}})\cdot\cdot\cdot(I-\gamma_{k}T_{X,J_{k}}) = \prod_{i=k}^{n}\left(I-\gamma_{i} T_{X,J_i}\right).
\end{equation}
Note that each of the element in $D_k^n$ is self-adjoint, positive but they in general do not commute. Because $T_{X,J_i}$ is a positive operator, we have $I - \gamma_i T_{X,J_i} \preccurlyeq I$ for our choice of $\gamma_i$. We are also going to use the following relationship in the rest of this proof. Recall that we denote the adjoint of operator $A$ as $A^*$:
\begin{equation}
\begin{aligned}
\sum_{k=1}^{n} D_{k+1}^n \gamma_{k}^{2} T_{X,J_k} \left(D_{k+1}^n\right)^*
&\preccurlyeq \gamma_{0} \sum_{k=1}^{n} D_{k+1}^n \gamma_{k} T_{X,J_k}\left(D_{k+1}^n\right)^*\\
&\stackrel{(1)}{=} \gamma_{0} \sum_{k=1}^{n} D_{k+1}^n\left(D_{k+1}^n\right)^*-D_{k}^n\left(D_{k+1}^n\right)^*\\
& = \gamma_0 \sum_{k=1}^n \left(I-\gamma_{n} T_{X,J_n}\right)\cdot\cdot\cdot \left(I-\gamma_{k+1} T_{X,J_{k+1}}\right)\left(I-\gamma_{k+1} T_{X,J_{k+1}}\right) \cdot\cdot\cdot \left(I-\gamma_{n} T_{X,J_n}\right)\\
&\hspace{30mm}- D_{k}^n\left(D_{k+1}^n\right)^*\\
& \stackrel{(2)}{\preccurlyeq} \gamma_0 \sum_{k=1}^n \left(I-\gamma_{n} T_{X,J_n}\right)\cdot\cdot\cdot \left(I-\gamma_{k+1} T_{X,J_{k+1}}\right) \cdot\cdot\cdot \left(I-\gamma_{n} T_{X,J_n}\right)- D_{k}^n\left(D_{k+1}^n\right)^*\\
& = \gamma_0 \sum_{k=1}^n D_{k+1}^n\left(D_{k+2}^n\right)^*- D_{k}^n\left(D_{k+1}^n\right)^* = \gamma_0 \left(I - D_1^n\left(D_{2}^n\right)^*\right) \preccurlyeq \gamma_0 I
\end{aligned}
\end{equation}
In (1) we used $D_{k+1}^n\gamma_{k} T_{X,J_k} = D_{k+1}^n-D_{k}^n$. In (2) we used 
\begin{equation}
   ABBA^* = AB^{1/2}BB^{1/2}A^*\preccurlyeq ABA 
\end{equation}
for positive, self-adjoint $B\preccurlyeq I$.\\
Now we give an inductive argument (applying induction on $r$):\\
\textbf{Initialisation}: 
When $r=0$, recall that $\Xi_n^0 = (Y_n - f_{\rho}(X_n)) K_{X_n,J_n}$. Thus, we have that
\begin{equation}
\begin{aligned}
    \langle f, E\left[\Xi_{n}^{0} \otimes \Xi_{n}^{0}\right] (f)\rangle_K 
    & = \langle f, E\left[ (Y_n - f_{\rho}(X_n))^2 \langle K_{X_n,J_n}, f\rangle_K K_{X_n,J_n}\right] \rangle_K\\
    & \stackrel{(1)}{\leq} \langle f, C_{\epsilon}^2E\left[ \langle K_{X_n,J_n}, f\rangle_K K_{X_n,J_n} \right]\rangle_K \\
    & = C_{\epsilon}^2  \langle f, T_{X,J_n}f\rangle_K
\end{aligned}
\end{equation}
In (1) we used our noise assumption (A4). So we have 
\begin{equation}
    E\left[\Xi_{n}^{0} \otimes \Xi_{n}^{0}\right] \preccurlyeq C_{\epsilon}T_{X,J_n}
\end{equation}

To perform induction over $r$, we also need a bound on $E\left[\eta_{n}^{0} \otimes \eta_{n}^{0}\right]$ as well. Recall that $\eta_{n}^{0}=\sum_{k=1}^{n} \gamma_{k} D_{k+1}^n  \Xi_{k}^{0}$ as defined in \eqref{eq:eta0n}. Thus we have: 
\begin{equation}
\begin{aligned}
E\left[\eta_{n}^{0} \otimes \eta_{n}^{0}\right]
&=\sum_{k=1}^{n}\sum_{j=1}^n \gamma_{j}\gamma_k D_{j+1}^n  E\left[\Xi_{j}^{0} \otimes \Xi_{k}^{0}\right] \left(D_{k+1}^n\right)^*\\
&\stackrel{(1)}{=}\sum_{k=1}^{n} \gamma_{k}^{2}D_{k+1}^n  E\left[\Xi_{k}^{0} \otimes \Xi_{k}^{0}\right] \left(D_{k+1}^n\right)^* \\
& \preccurlyeq C_{\epsilon} \sum_{k=1}^{n} D_{k+1}^n \gamma_{k}^{2} T_{X,J_k} \left(D_{k+1}^n\right)^* \\
& \preccurlyeq C_{\epsilon} \gamma_{0} I
\end{aligned}
\end{equation}
In (1) the interaction terms vanish because the noise variables $\Xi_j^0,\Xi_k^0$ are mean-zero and independent when $j\neq k$.\\
\textbf{Induction} : If we assume for $r\geq 0$, 
\begin{equation}
  E\left[\Xi_{n}^{r} \otimes \Xi_{n}^{r}\right] \preccurlyeq \gamma_{0}^{r} R^{2 r} C_{\epsilon}  T_{X,J_n} 
\end{equation}
and 
\begin{equation}
  E\left[\eta_{n}^{r} \otimes \eta_{n}^{r}\right] \preccurlyeq\gamma_{0}^{r+1} R^{2 r} C_{\epsilon} I 
\end{equation}
then for $r+1$:
\begin{equation}
\begin{aligned}
E\left[\Xi_{n}^{r+1} \otimes \Xi_{n}^{r+1}\right] 
& \stackrel{(1)}{=} E\left[\left(T_{X,J_n}-T_{X_n,J_n}\right) \eta_{n-1}^{r} \otimes \eta_{n-1}^{r}\left(T_{X,J_n}-T_{X_n,J_n}\right)\right] \\
&=E\left[\left(T_{X,J_n}-T_{X_n,J_n}\right) E\left[\eta_{n-1}^{r} \otimes \eta_{n-1}^{r}\right]\left(T_{X,J_n}-T_{X_n,J_n}\right)\right] \\
& \preccurlyeq \gamma_{0}^{r+1} R^{2 r} C_{\epsilon} E\left[\left(T_{X,J_n}-T_{X_n,J_n}\right)^{2}\right] \\
& = \gamma_{0}^{r+1} R^{2 r} C_{\epsilon}\left(E\left[\left(T_{X_n,J_n}\right)^2\right] - T_{X,J_n}^2\right) \\
& \stackrel{(2)}{\preccurlyeq} \gamma_{0}^{r+1} R^{2 r+2} C_{\epsilon} T_{X,J_n}
\end{aligned}
\end{equation}
Here (1) is the definition of $\Xi_n^{r+1}$. For (2), it is sufficient to show $E\left[\left(T_{X_n,J_n}\right)^2\right]\preccurlyeq R^2 T_{X,J_n}$. This is true because:
\begin{equation}
    \langle f, E\left[\left(T_{X_n,J_n}\right)^2\right] f\rangle_K = E[\|T_{X_n,J_n}(f)\|_K^2] = E[\langle f, K_{X_n,J_n}\rangle_K^2\|\khatxn\|_K^2] \leq R^2 \langle f, T_{X,J_n} f\rangle_K
\end{equation}
Recall that $\eta_{n}^{r+1}=\sum_{k=1}^{n} D_{k+1}^n\gamma_{k} \Xi_{k}^{r+1}$, then
\begin{equation}
\begin{aligned}
E\left[\eta_{n}^{r+1} \otimes \eta_{n}^{r+1}\right] 
& =  E\left[\sum_{k=1}^{n}\gamma_k^2 D_{k+1}^n \Xi_{k}^{r+1} \otimes \Xi_{k}^{r+1} \left(D_{k+1}^n\right)^*\right] \\
& = \sum_{k=1}^{n}\gamma_k^2 D_{k+1}^n E[\Xi_{k}^{r+1} \otimes \Xi_{k}^{r+1}] \left(D_{k+1}^n\right)^*\\
& \preccurlyeq C_{\epsilon} \gamma_{0}^{r+1} R^{2 r} \sum_{k=1}^{n} D_{k+1}^n \gamma_{k} T_{X,J_k} \left(D_{k+1}^n\right)^*\\
& \preccurlyeq C_{\epsilon} \gamma_{0}^{r+2} R^{2 r+2} I
\end{aligned}
\end{equation}
\end{proof}

\begin{lemma}
\label{lemma:Jnvariance}
Under assumptions A1-A3, we have
\begin{equation}
    E\left[\left\|\bar{\eta}_{n}^{r}\right\|_{2}^{2}\right] = O\left(\gamma_0^rR^{2r}C_{\epsilon}^2n^{-\frac{2s}{2s+1}}\right)
\end{equation}
\end{lemma}
\begin{proof}
\begin{equation}
\label{eq:variance main eq}
    \begin{aligned}
        n^2E\left[\left\|\bar{\eta}_{n}^{r}\right\|_{2}^{2}\right]
        &=   E\left\| \sum_{j=1}^{n} \sum_{k=1}^{j}\left[\prod_{i=k+1}^{j}\left(I-\gamma_{i} T_{X,J_i}\right)\right] \gamma_{k} \Xi_k^r\right\|_2^{2} \\
    & =  E\left\| \sum_{k=1}^{n} \sum_{j=k}^{n}\left[\prod_{i=k+1}^{j}\left(I-\gamma_{i} T_{X,J_i}\right)\right] \gamma_{k} \Xi_k^r\right\|_2^{2} \\
     & =  \sum_{k=1}^n\gamma_k^2E\left\|  \sum_{j=k}^{n}\left[\prod_{i=k+1}^{j}\left(I-\gamma_{i} T_{X,J_i}\right)\right] \Xi_k^r\right\|_2^{2}\\
     & = \sum_{k=1}^n\gamma_k^2
     E\left\langle\sum_{j=k}^{n}\left[\prod_{i=k+1}^{j}\left(I-\gamma_{i} T_{X,J_i}\right)\right]  \Xi_{k}^r, T_{X,J_n}\sum_{j=k}^{n}\left[\prod_{i=k+1}^{j}\left(I-\gamma_{i} T_{X,J_i}\right)\right]  \Xi_{k}^r\right\rangle_{K}\\
        &  = \sum_{k=1}^{n} \gamma_{k}^{2} E \operatorname{tr}\left( T_{X,J_n} M_{k}^{n} \Xi_{k}^r \otimes \Xi_{k}^r\left(M_{k}^{n}\right)^*\right)\quad\text{where }M_k^n=\sum_{j=k}^{n}\left[\prod_{i=k+1}^{j}\left(I-\gamma_{i} T_{X,J_i}\right)\right]\\
        & = \sum_{k=1}^{n} \gamma_{k}^{2} \sum_{l=k}^n\sum_{m=k}^n\operatorname{tr}\left( T_{X,J_n} \left[\prod_{i=k+1}^l (I - \gamma_i T_{X,J_i})\right] E[\Xi_{k}^r \otimes \Xi_{k}^r]\left[\prod_{i=k+1}^m (I - \gamma_i T_{X,J_i})\right]\right)\\
        & \stackrel{(1)}{\leq} \sum_{k=1}^{n} \gamma_{k}^{2} \sum_{l=k}^n\sum_{m=k}^n\sum_{t=1}^{\infty}\left( \lambda_{n,t}\left[\prod_{i=k+1}^l (1 - \gamma_i \lambda_{i,t})\right] \gamma_0^rR^{2r}C_{\epsilon}^2\lambda_{k,t} \left[\prod_{i=k+1}^m (1 - \gamma_i \lambda_{i,t})\right]\right)\\
        & \stackrel{(2)}{\leq } \gamma_0^rR^{2r}C_{\epsilon}^2\sum_{k=1}^n\gamma_k^2\sum_{t=1}^{J_k}\lambda_{k,t}\lambda_{n,t}\left(\sum_{j=k}^{n}\left[\prod_{i=k+1}^{j}\left(1-\gamma_{i} C\lambda_{k,t}\right)\right]\right)^{2}.\\
    \end{aligned}
\end{equation}
Here we denote the $t$-th eigenvalue of $T_{X,J_k}$ as $\lambda_{k,t}\geq 0$. In (1), we used the trace inequality $\operatorname{tr}(A_1\cdot\cdot\cdot A_m) \leq \sum_{t} \lambda_t(A_1)\cdot\cdot\cdot\lambda_t(A_m)$, where $\lambda_t(A)$ takes the $t$-th largest eigenvalue of $A$ (p.342 of \cite{marshall1979inequalities}). In this step we also used $E\left[\Xi_{n}^{r} \otimes \Xi_{n}^{r}\right] \preccurlyeq \gamma_{0}^{r}R^{2 r}C_{\epsilon}^2T_{X,J_n}$, stated in Lemma~\ref{onXinr}. In step (2) we used the rank of $T_{X,J_k}$ is at most $J_k$. We also apply the uniform bound on the eigenvalues stated in Lemma~\ref{lemma:samejdifferentkernels}. 
We claim we can further extend the inequality as follows, the gap will be provided as an technical lemma in Lemma~\ref{gap1inlemma7}
\begin{equation}
\label{eq:afterpain}
    \begin{aligned}
        (\gamma_0^rR^{2r}C_{\epsilon}^2)^{-1}n^2E\left[\left\|\bar{\eta}_{n}^{r}\right\|_{2}^{2}\right]
        &\leq  \sum_{k=1}^{n} \gamma_{k}^{2} \sum_{t=1}^{J_k} \lambda_{k,t}\lambda_{n,t}\left((n-k)^2 \wedge C\left(\lambda_{k,t}^{-2/(1-\zeta)} +  \lambda_{k,t}^{-2}k^{2\zeta}\right)\right)\\
        &\leq \underbrace{\sum_{k=1}^{n} \gamma_{k}^{2}  \sum_{t=1}^{J_k} \lambda_{k,t}\lambda_{n,t}\left((n-k)^{2} \wedge C\lambda_{k,t}^{-2/(1-\zeta)}\right)}_{S_{1}}\\
        &+\underbrace{\sum_{k=1}^{n} \gamma_{k}^{2}  \sum_{t=1}^{J_k}\lambda_{k,t}\lambda_{n,t} \left((n-k)^{2} \wedge \lambda_{k,t}^{-2}k^{2 \zeta}\right)}_{S_{2}}
        \stackrel{(1)}{=}O\left( n^{2-\frac{2s}{2s+1}}\right)
    \end{aligned}
\end{equation} 
In step $(1)$, we can show that both of $S_1$ and $S_2$ are of order $O\left( n^{2-\frac{2s}{2s+1}}\right)$. These results are provided in lemma \ref{gap2inlemma7}, which concludes our proof.
\end{proof}

\begin{lemma}
\label{gap1inlemma7}
Using the same notation as the last line of \eqref{eq:variance main eq}. We have
\begin{equation}
\sum_{j=k}^{n}\left[\prod_{i=k+1}^{j}\left(1-\gamma_{i} \lambda_{k, t}\right)\right]
\leq (n-k) \wedge C\left(\lambda_{k, t}^{-1 /(1-\zeta)}+\lambda_{k, t}^{-1} k^{\zeta}\right)
\end{equation}
where $\zeta = \frac{1}{2s+1}$.
\begin{proof}
We first bound the inside term:
\begin{equation}
\begin{aligned}
\prod_{i=k+1}^{j}\left(1-\gamma_{i} \lambda_{k,t}\right)
& = \prod_{i=k+1}^{j} \exp\left(\log(1-\gamma_i\lambda_{k,t})\right)\\
&\leq \exp \left( - \sum_{i=k+1}^{j}\left(\gamma_{i} \lambda_{k,t}\right)\right)\\
&\leq \exp \left(-\lambda_{k,t} \int_{u=k+1}^{j+1}\left(\frac{1}{u^{\zeta}} d u\right)\right) \quad \left(\gamma_i = \gamma_0i^{-\zeta}\right)\\
&\leq \exp \left(-\lambda_{k,t} \frac{(j+1)^{1-\zeta}-(k+1)^{1-\zeta}}{1-\zeta}\right)
\end{aligned}
\end{equation}
Then we have
\begin{equation}
\begin{aligned}
\sum_{j=k}^{n} \prod_{i=k+1}^{j}\left(1-\gamma_{i} \lambda_{k,t}\right) 
& \leq \sum_{j=k}^{n} \exp \left(-\lambda_{k,t} \frac{(j+1)^{1-\zeta}-(k+1)^{1-\zeta}}{1-\zeta}\right) \\
& \leq \int_{k}^{n} \exp \left(-\lambda_{k,t} \frac{(u+1)^{1-\zeta}-(k+1)^{1-\zeta}}{1-\zeta}\right) d u
\end{aligned}
\end{equation}
We provide two upper bounds for this quantity: The first one is simply $n-k$, because $\zeta < 1/2$ and $n,k,t \geq 1$.\\

Now we derive the second bound:
\begin{equation}
\label{tobechanged}
\int_k^{n} \exp\left( -\lambda_{k,t} \frac{(u+1)^{1-\zeta}-(k+1)^{1-\zeta}}{1-\zeta} \right)d u=\int_{k+1}^{n+1} \exp \left(-\lambda_{k,t} \frac{u^{1-\zeta}-(k+1)^{1-\zeta}}{1-\zeta} \right)d u
\end{equation}
Now we perform a change of variables, denote $\rho=1-\zeta$:
\begin{equation}
\begin{aligned}
v^{\rho} &=\rho^{-1}\lambda_{k,t}\left( (u)^{\rho}-(k+1)^{\rho}\right) \\
v &= \rho^{-1/\rho}\lambda_{k,t}^{1/\rho}\left((u)^{\rho}-(k+1)^{\rho}\right)^{1 / \rho} \\
d v &=\rho^{-1/\rho}\lambda_{k,t}^{1/\rho}\left(1-\left(\frac{k+1}{u}\right)^{\rho}\right)^{1 / \rho-1} d u\\
    du &= \rho^{1/\rho}\lambda_{k,t}^{-1/\rho} \left(1-\left(\frac{k+1}{u}\right)^{\rho}\right)^{1- 1/\rho}dv\\
    & = \rho^{1/\rho}\lambda_{k,t}^{-1/\rho}\left(1 - \frac{(k+1)^{\rho}}{v^{\rho}\rho\lambda_{k,t}^{-1} + (k+1)^{\rho}}\right)^{1-1/\rho}dv\\
    & = \rho^{1/\rho}\lambda_{k,t}^{-1/\rho}\left(1+\frac{(k+1)^{\rho}}{v^{\rho}\rho\lambda_{k,t}^{-1}}\right)^{1/\rho - 1}dv
\end{aligned}
\end{equation}
Plug this into \eqref{tobechanged}:
\begin{equation}
\begin{aligned}
\sum_{j=k}^{n} \prod_{i=k+1}^{j}\left(1-\gamma_{i} \lambda_{k,t}\right) 
& \leq \int_{0}^{\infty} \rho^{1/\rho}\lambda_{k,t}^{-1/\rho}\left(1+\frac{(k+1)^{\rho}}{v^{\rho}\rho\lambda_{k,t}^{-1}}\right)^{1/\rho - 1}\exp \left(-v^{\rho}\right) d v \\
& \leq 2^{1/\rho - 1}\rho^{1/\rho}\lambda_{k,t}^{-1/\rho} \int_{0}^{\infty}\left(1\vee\frac{(k+1)^{\rho}}{v^{\rho}\rho\lambda_{k,t}^{-1}}\right)^{1/\rho - 1} \exp \left(-v^{\rho}\right) d v \\
& \leq 2^{1/\rho - 1}\rho^{1/\rho}\lambda_{k,t}^{-1/\rho} \left(I_{1} \vee (k+1)^{1-\rho} \lambda_{k,t}^{1/\rho-1}I_2\right) \\
& = 2^{1/\rho - 1}\rho^{1/\rho}I_1\lambda_{k,t}^{-1/\rho} \vee 2^{1 - 2\rho +1/\rho}\rho^{1/\rho}I_2k^{1-\rho}\lambda_{k,t}^{-1}
\end{aligned}
\end{equation}
which concludes the lemma.
\end{proof}
\end{lemma}

\begin{lemma}
\label{gap2inlemma7}
For both $S_1$ and $S_2$ in \eqref{eq:afterpain}, we have
\begin{equation}
    S_i = O\left( n^{2-\frac{2s}{2s+1}}\right)\quad i=1,2
\end{equation}
\end{lemma}
\begin{proof}
First we derive the bound for $S_1$, denote $\rho = 1-\zeta$:
\begin{equation}
\begin{aligned}
S_{1} & = \sum_{k=1}^{n} \gamma_{k}^{2}  \sum_{t=1}^{J_k} \lambda_{k,t}\lambda_{n,t}\left((n-k)^{2} \wedge C\lambda_{k,t}^{-2/\rho}\right)\\
& \stackrel{(1)}{\leq} C\sum_{k=1}^{n} \gamma_{k}^{2}  \sum_{t=1}^{\infty} t^{-4s}\left((n-k)^{2} \wedge Ct^{4s/\rho}\right)\\
& = C\sum_{k=1}^{n} \gamma_{k}^{2}\left( \sum_{t=1}^{(n-k)^{\rho/2s}} t^{4s(1/\rho - 1)} + (n-k)^2\sum_{t=(n-k)^{\rho/2s}}^{\infty} t^{-4s}\right)\\
& \leq C\sum_{k=1}^n \gamma_k^2\left( (n-k)^{2-2\rho+\rho/2s} + (n-k)^2 (n-k)^{(-4s+1)\rho/2s}\right)\\
& \leq C\sum_{k=1}^n \gamma_k^2 (n-k)^{3\zeta} = C\sum_{k=1}^n k^{-2\zeta} (n-k)^{3\zeta}\\
                &= C \sum_{k=1}^{n}\left(\frac{n}{k}-1\right)^{3\zeta} k^{\zeta} =n^{\zeta} \sum_{k=1}^{n}\left(\frac{1}{k / n}-1\right)^{3\zeta}\left(\frac{k}{n}\right)^{\zeta}\\
&=C  n^{1+\zeta}\left(\frac{1}{n} \sum_{k=1}^{n}\left(\frac{1}{k / n}-1\right)^{3\zeta}\left(\frac{k}{n}\right)^{\zeta}\right)\\
&=C n^{1+\zeta}\left(\frac{1}{n} \sum_{k=1}^{n}\left(\frac{1}{k / n}-1\right)^{2 \zeta}\left(1-\frac{k}{n}\right)^{\zeta}\right)
\end{aligned}
\end{equation}
In (1) we used the bound for $\lambda_{k,t}$ and $\lambda_{n,t}$ proved in Lemma~\ref{lemma:samejdifferentkernels}. Next, we use
\begin{equation}
\int_{0}^{1}\left(\frac{1}{x}-1\right)^{2 \zeta}(1-x)^{\zeta} d x \leq \int_{0}^{1}\left(\frac{1}{x}-1\right)^{2 \zeta} d x <\infty
\end{equation}
So for $S_1$ we conclude 
\begin{equation}
S_{1} \leq Cn^{1+\zeta} = O(n^{2-\frac{2s}{2s+1}})
\end{equation}
Now we bound $S_2$:
\begin{equation}
\begin{aligned}
S_{2} &\leq C\sum_{k=1}^{n} \gamma_{k}^{2}  \sum_{t=1}^{\infty} \left(t^{-4s}(n-k)^{2} \wedge k^{2 \zeta}\right)\\
& \leq  C\sum_{k=1}^{n} \gamma_{k}^{2} \left(\sum_{t=1}^{(n-k)^{\frac{1}{2s}}/k^{\frac{\zeta}{2s}}} k^{2 \zeta}+\sum_{t=(n-k)^{\frac{1}{2s}}/k^{\frac{\zeta}{2s}}}^{\infty} t^{-4s}(n-k)^{2}\right)\\
&\leq  C\sum_{k=1}^{n} \gamma_{k}^{2} \left(k^{2 \zeta} \sum_{t=1}^{(n-k)^{\frac{1}{2s}}/k^{\frac{\zeta}{2s}}} 1+(n-k)^{2} \sum_{t=(n-k)^{\frac{1}{2s}}/k^{\frac{\zeta}{2s}}}^{\infty} t^{-4s}\right)\\
&\leq  C\sum_{k=1}^{n} \gamma_{k}^{2} \left(k^{2 \zeta} \frac{(n-k)^{\frac{1}{2s}}}{k^{\frac{\zeta}{2s}}}+(n-k)^{2}\left(\frac{(n-k)^{\frac{1}{2s}}}{k^{\frac{\zeta}{2s}}}\right)^{1-4s}\right)\\
&=C\sum_{k=1}^{n} \gamma_{k}^{2} \left(k^{2 \zeta-\frac{\zeta}{2s}}(n-k)^{\frac{1}{2s}}+(n-k)^{\frac{1}{2s}} k^{\frac{\zeta}{2s}(4s-1)}\right)\\
&=C \sum_{k=1}^{n} \frac{1}{k^{2 \zeta}}(n-k)^{\frac{1}{2s}} k^{\frac{\zeta}{2s}(4s-1)}
=C\sum_{k=1}^{n} k^{-\frac{\zeta}{2s}}(n-k)^{\frac{1}{2s}}\\
&=C n^{\left(1-\frac{\zeta}{2s}+\frac{1}{2s}\right)} \left(\frac{1}{n} \sum_{k=1}^{n}\left(\frac{k}{n}\right)^{-\frac{\zeta}{2s}}\left(1-\frac{k}{n}\right)^{\frac{1}{2s}}\right)\\
&\stackrel{(1)}{\leq} C n^{\left(1+\frac{1-\zeta}{2s}\right)} = O\left( n^{2-\frac{2s}{2s+1}}\right),
\end{aligned}
\end{equation}
in (1) we use
\begin{equation}
\begin{aligned}
        \int_0^1 x^{-\zeta/2s}(1-x)^{1/(2s)} dx &\leq \int_0^1 x^{-\zeta/(2s)} dx \\
        & = \int_{1}^{\infty} u^{\zeta/(2s) - 2}du <\infty
\end{aligned}
\end{equation}
\end{proof}

\begin{lemma}
\label{lemma:exactlyzero}
Let $\eta_n^r$ be the sequences defined in \eqref{defineetanr}, then for any $r\geq n$ we have \begin{equation}
    \eta_n^r = 0.
\end{equation}
As a further consequence, for $\alpha_n^r$ defined in \eqref{definealphan}, for any $r\geq n$, we have
\begin{equation}
    \alpha_n^r = 0.
\end{equation}
\end{lemma}
\begin{proof}
We prove both results by induction (over $n$). We recall the definition of $\Xi_k^r$, for $n,r\geq 1$:
\begin{equation}
    \Xi_n^r = (T_{X,J_n} - T_{X_n,J_n})\eta_{n-1}^{r-1}
\end{equation}
Let's first show $\eta_n^r = 0$ for any $r\geq n$. 

When $n=0$, by definition for any $r\geq 0$, $\eta_0^r = 0$.

Now assume for $k$ and any $r\geq k$ we have $\eta_k^r = 0$, then for any $r\geq k+1$
\begin{equation}
    \begin{aligned}
    \eta_{k+1}^r & = (I-\gamma_{k+1}T_{X,J_{k+1}})\eta_k^r + \gamma_{k+1}\Xi_{k+1}^r\\
    & = 0 + \gamma_{k+1} \left(T_{X,J_{k+1}} - T_{X_{k+1},J_{k+1}}\right) \eta_{k}^{r-1}\\
    &  = 0 + 0
    \end{aligned}    
\end{equation}
This shows that $\eta_{k+1}^r = 0$ for any $r\geq k+1$.

Now we prove the second part. Here we need to use the following recursive relationship of $\alpha_n^r$ (proof is postponed later):
\begin{equation}
\label{eq:alpharecur}
    \alpha_{n}^{r}=\left(I-\gamma_{n} T_{X_n,J_n}\right) \alpha_{n-1}^{r}+\gamma_{n} \Xi_{n}^{r+1}
\end{equation}
When $n = 0$, by definition $\vartheta_0 = \sum_{k=0}^r\eta_0^k$ for any $r\geq 0$. Therefore $\alpha_0^r = 0$ for any $r\geq 0$.

Then assume for $k$ we have $\forall r\geq k$, $\alpha_k^r = 0$, then for $r \geq k+1$
\begin{equation}
\begin{aligned}
    \alpha_{k+1}^{r}
    &=\left(I-\gamma_{k+1} T_{X_{k+1},J_{k+1}}\right) \alpha_k^{r}+\gamma_{k+1} \Xi_{k+1}^{r+1}\\
    & = 0 + \gamma_{k+1} \left(T_{X,J_{k+1}} - T_{X_{k+1},J_{k+1}}\right) \eta_k^r \\
    & = 0 + 0
\end{aligned}
\end{equation}
Now we just need to verify the claimed recursive formula \eqref{eq:alpharecur}.
\begin{equation}
    \begin{aligned}
        (I-\gamma_{n} &T_{X_n,J_n}) \alpha_{n-1}^{r}+\gamma_{n} \Xi_{n}^{r+1} = \left(I-\gamma_{n} T_{X_n,J_n}\right) \alpha_{n-1}^{r}+\gamma_{n} (T_{X,J_n} - T_{X_n,J_n})\eta_{n-1}^r\\
        & = \left(I-\gamma_{n} T_{X_n,J_n}\right)\vartheta_{n-1} - \sum_{k=0}^r(I - \gamma_nT_{X_n,J_n})\eta_{n-1}^k + \gamma_{n} (T_{X,J_n} - T_{X_n,J_n})\eta_{n-1}^r\\
        & = \vartheta_{n} - \gamma_n\Xi_n - \sum_{k=0}^r(I - \gamma_nT_{X_n,J_n} + \gamma_nT_{X,J_n} - \gamma_nT_{X,J_n})\eta_{n-1}^k \\
        &\quad + \gamma_{n} (T_{X,J_n} - T_{X_n,J_n})\eta_{n-1}^r\\
        & = \vartheta_{n} - \gamma_n\Xi_n - \sum_{k=0}^r(I - \gamma_nT_{X,J_n})\eta_{n-1}^k
        - \sum_{k=0}^r \gamma_n\Xi_n^{k+1} +\gamma_n\Xi_{n}^{r+1}\\
        & = \vartheta_n - \sum_{k=0}^r\eta_n^k = \alpha_n^r
    \end{aligned}
\end{equation}
\end{proof}

\section{Proof of Theorem~6.3}
In this section we will show Sieve-SGD achieve a near-optimal convergence rate under the parameter regime specified in Theorem~\ref{maintheorem_fixedJn} in the main text. The proof is similar to that of Theorem~\ref{maintheorem}. But in the section, we need to consider the RKHSs associated with kernels
\begin{equation}
    K_{J_n}(s,t) =  \sum_{j = 1}^{J_n} j^{-2\omega} \psi_j(s)\psi_j(t),\quad \text{with }J_n = \lfloor n^{\frac{1}{2s+1}}\log^2 n\rfloor, \omega \in\left(\frac{1}{2}, s\right).
\end{equation}
To clarify, our reader should treat $\omega$ as a fixed value and $J_n$ is a determinisic sequence that increases with $n$. The aforementioned series of RKHSs are subspaces of the RKHS (denoted as $\mathcal{H}_K$) spanned by the kernel
\begin{equation}
    K(s,t) = \sum_{j = 1}^{\infty} j^{-2\omega} \psi_j(s)\psi_j(t),
\end{equation}
equipped with the same inner product
\begin{equation}
    \langle f, g \rangle_{K} = \sum_{j = 1}^{\infty} j^{2\omega} \langle f, \psi_j\rangle_{L^2_{\nu}} \langle g,\psi_j \rangle_{L^2_{\nu}}.
\end{equation}
Note that, the above inner product no longer have a direct correspondence with our ellipsoid assumptions. 
\subsection{Separation of the error}
Similar to section~\ref{subsection:separation of error}, we consider the following stochastic sequences. The first one is the ``total deviation" sequence $\Delta_n$:
\begin{equation}
\begin{aligned}
    \Delta_0 & =-f_{\rho} \\
    \Delta_{n} & =\left(I-\gamma_{n} T_{X_{n}, J_{n}}\right) \Delta_{n-1}+\gamma_{n} \Xi_{n},
\end{aligned}
\end{equation}
where
\begin{equation}
    \begin{aligned}
&T_{X_{n}, J_{n}}(f)=f\left(X_{n}\right) K_{X_{n}, J_{n}}  = f\left(X_{n}\right) \left(\sum_{j=1}^{J_n} j^{-2\omega} \psi_j(X_n)\psi_j\right)\\
&\Xi_{n}=\left(Y_{n}-f_{\rho}\left(X_{n}\right)\right) K_{X_{n}, J_{n}} = \left(Y_{n}-f_{\rho}\left(X_{n}\right)\right)\left(\sum_{j=1}^{J_n} j^{-2\omega} \psi_j(X_n)\psi_j\right).
\end{aligned}
\end{equation}
The average of $\Delta_n$ is the differene between Sieve-SGD and $f_{\rho}$:
\begin{equation}
    \bar{\Delta}_{n}=\frac{1}{n} \sum_{j=1}^{n} \Delta_{j}=\bar{f}_{n}-f_{\rho}
\end{equation}
Similarly, we decompose the $\Delta_n$ into two parts $\Delta_{n}=\eta_{n}+\vartheta_{n}$, where
\begin{equation}
    \begin{aligned}
&\eta_{0}=-f_{\rho} \\
&\eta_{n}=\left(I-\gamma_{n} T_{X_{n}, J_{n}}\right) \eta_{n-1},
\end{aligned}
\end{equation}
and
\begin{equation}
     \begin{aligned}
&\vartheta_{0}=0 \\
&\vartheta_{n}=\left(I-\gamma_{n} T_{X_{n}, J_{n}}\right) \vartheta_{n-1}+\gamma_{n} \Xi_{n}
\end{aligned}
\end{equation}
We give bounds on $E\left[\left\|\bar{\eta}_{n}\right\|_{2}^{2}\right]$ and $E\left[\left\|\bar{\vartheta}_{n}\right\|_{2}^{2}\right]$ separately and combine them to get one for $E\left[\left\|\bar{\Delta}_{n}\right\|_{2}^{2}\right]$. 

\subsection{Dound on initial condition sub-process.} 
\begin{proof}[Poof of bound on initial value]
The proof formally ensembles section~\ref{initialsection} in a line-by-line fashion. But as we mentioned, the meanings of $K_{X_n,J_n}$ and RKHS inner product here are different. Specifically, to ensure $\|\eta_0\|_K^2 = \sum_{j=1}^{\infty} \left(j^{\omega}\langle f_{\rho}, \psi_j\rangle_{L^2_{\nu}}\right)^2$ (or in general $\|\eta_n\|_K^2$) finite, we need $\omega \leq s$. Define $R^2 = M^2\zeta(2\omega)$, it is also direct to verify that $\gamma_n\|K_{X_n,J_n}\|_K^2 \leq \gamma_0 R^2 < 1$ by our choice of $\gamma_n$. The conclusion of this part is:
\begin{equation}
    \left(E\left[\left\|\bar{\eta}_{n}\right\|_{2}^{2}\right]\right)^{1 / 2}= \left(E\left[\left\|\frac{1}{n}\sum_{i=1}^n \eta_i\right\|_{2}^{2}\right]\right)^{1 / 2}= O\left(n^{-\frac{s}{2 s+1}}\right)
\end{equation}
\end{proof}

\subsection{Bound on noise sub-process}
The basic structure of proof is similar to the corresponding part of Theorem~\ref{maintheorem}. But the details are different: In Lemma~\ref{lemma:Jnvariance}, we used the fact that $t_j = j^{-2s}$ decreases quickly enough to control the magnitude of the noise; however, here we will leverage the finiteness of operators to give a different (and technically slightly simpler) bound, which is unique to sieve-type SGD.

\begin{proof}[proof of bound on noise]
We still need the following working sequences to facilitate the analysis:
\begin{equation}
    \begin{aligned}
&\eta_{0}^{0}=0 \\
&\eta_{n}^{0}=\left(I-\gamma_{n} T_{X, J_{n}}\right) \eta_{n-1}^{0}+\gamma_{n} \Xi_{n}^{0}
\end{aligned}
\end{equation}
where $\Xi_{n}^{0}=\Xi_{n}=\left(Y_{n}-f_{\rho}\left(X_{n}\right)\right) K_{X_{n}, J_{n}}$ and for $f\in \mathcal{H}_{K}$,
\begin{equation}
    T_{X, J_{n}}(f)=\int_{\mathcal{X}}\left\langle f, K_{x, J_{n}}\right\rangle_{K} K_{x, J_{n}} d \rho_{X}(x) = \int_{\mathcal{X}}f(x) \left(\sum_{j=1}^{J_n}j^{-2\omega}\psi_j(x)\psi_j\right) d \rho_{X}(x)
\end{equation}
For each $r > 0$, we define
\begin{equation}
    \begin{aligned}
&\eta_{0}^{r}=0 \\
&\eta_{n}^{r}=\left(I-\gamma_{n} T_{X, J_{n}}\right) \eta_{n-1}^{r}+\gamma_{n} \Xi_{n}^{r}
\end{aligned}
\end{equation}
where $\Xi_{n}^{r}=\left(T_{X, J_{n}}-T_{X_{n}, J_{n}}\right) \eta_{n-1}^{r-1}$. Then we have

\begin{equation}
    \begin{aligned}
        \left(E\left\|\bar{\vartheta}_{n}\right\|_{2}^{2}\right)^{1 / 2} 
        &\leq \sum_{k=0}^{r}\left(E\left\|\bar{\eta}_{n}^{k}\right\|_{2}^{2}\right)^{1 / 2}+\left(E\left\|\bar{\vartheta}_{n}-\sum_{k=0}^{r} \bar{\eta}_{n}^{k}\right\|_{2}^{2}\right)^{1 / 2} \\
        &\stackrel{(1)}{= } \sum_{k=0}^{r}\left(E\left\|\bar{\eta}_{n}^{k}\right\|_{2}^{2}\right)^{1 / 2}+ 0,\quad\text{when }r\geq n\\
        & \stackrel{(2)}{\leq} \sum_{k=0}^{r} C\left(\gamma_{0} R^{2}\right)^{k / 2} C_{\epsilon} n^{-s/(2s+1)}\log n,\quad \text{with }R^2 = M^2\zeta(2\omega)\\
        & = O\left(n^{-s/(2s+1)}\log n\right).
    \end{aligned}
\end{equation}
In (1) we used Lemma~\ref{lemma:exactlyzero} (after taking another average). Step (2) leveraged the finiteness of the rank of $T_{X,J_n}$, which is given in Lemma~\ref{lemma:myvariancecontrol}. Our choice of $\omega > \frac{1}{2}$ ensures that $R$ is a finite number not depending on $n$.
\end{proof}

\begin{lemma}
\label{lemma:myvariancecontrol}
Under assumptions A1-A3, we have
\begin{equation}
    E\left[\left\|\bar{\eta}_{n}^{r}\right\|_{2}^{2}\right] = O\left(\gamma_0^rR^{2r}C_{\epsilon}^2 n^{-2s/(2s+1)} \log^{2} n\right)
\end{equation}
\end{lemma}
\begin{proof}
Denote $\zeta = \frac{1}{2s+1}, \rho = 1-\zeta$. According to the proof of Lemma~\ref{lemma:Jnvariance} (equation \eqref{eq:afterpain}), we have 
\begin{equation}
    \begin{aligned}
        (\gamma_0^rR^{2r}C_{\epsilon}^2)^{-1}n^2E\left[\left\|\bar{\eta}_{n}^{r}\right\|_{2}^{2}\right]
        &\leq \sum_{k=1}^{n} \gamma_{k}^{2}  \sum_{t=1}^{J_k} \lambda_{k,t}\lambda_{n,t}\left( C\lambda_{k,t}^{-2/\rho}\right) + \lambda_{k,t}\lambda_{n,t} \left(\lambda_{k,t}^{-2}k^{2 \zeta}\right)\\
        & \stackrel{(1)}{\leq }C\sum_{k=1}^n \gamma_k^2 \sum_{t=1}^{J_k} \left(t^{-4\omega + 4\omega/\rho} + k^{2\zeta}\right)\\
        & \leq C \sum_{k=1}^n k^{-2\zeta} (J_k)^{4\omega/\rho - 4\omega + 1} + J_k \\
        & \stackrel{(2)}{\leq }C\sum_{k = 1}^n k^{\zeta} + k^{\zeta} \log^{2} k = O\left( n^{1+\zeta} \log^{2} n\right)
    \end{aligned}
\end{equation} 
In step (1) we used the result of Lemma~\ref{lemma:samejdifferentkernels}. For step (2) we note $\omega < s$.
\end{proof}

\section{Space Expense Analysis}
In this section, we are going to formally model how round-off errors appear in the process of collecting data and constructing the Sieve-SGD estimator. We are also going to characterize how to optimally asymptotically increase space expense to ensure that round-off error does not affect model performance (beyond a multiplicative log term). Under minor simplification, we will show in Section~\ref{subsection:spaceexpenseofSieve-SGD} that $O(\log(n))$ times more space resources (counted in bits) is enough to make the influence of round-off error on statistical performance negligible. On the other hand, in Section~\ref{app:memoryusage} we will give the minimal space expense (also counted in bits) required for constructing a statistically rate-optimal estimator (using any procedures). Notably, the optimal space expense of Sieve-SGD only differs from this lower bound by a polylog term, therefore we claim the space expense of Sieve-SGD is almost optimal.

\textit{Notation:} the left subscript ${}_r\cdot$ will be used to denote quantities that are directly related to round-off error.
\subsection{Sieve-SGD under round-off error}
\label{subsection:spaceexpenseofSieve-SGD}
In this subsection we are going to give an analysis of how a $O(\log^3(n)n^{\frac{1}{2s+1}})$-sized version of Sieve-SGD can achieve the optimal rate for estimating $f_{\rho}$, under assumptions A1 - A4 and some extra assumptions (A5,A6) regarding round-off error. We focus on the case when $\omega = s$ (Theorem~\ref{maintheorem}). Very similar argument can be applied to the case when $\omega \neq s$ to proof Sieve-SGD can achieve near-optimality with the save space expense (Theorem~\ref{maintheorem_fixedJn}). 

We note that the size of the estimators above can be decomposed as \begin{equation}
    \log^3(n)n^{\frac{1}{2s+1}} =\log^2(n)n^{\frac{1}{2s+1}}\cdot \log(n), 
\end{equation}
where the $\log^2(n)n^{\frac{1}{2s+1}}$ term corresponds to the minimal number of basis functions needed to construct Sieve-SGD as stated in Theorem~\ref{maintheorem}, and the extra logarithm term is due to the precision loss when storing a real number as a float point number.

Modern statistical estimation procedures are performed exclusively with the help of digital computers. Although computers cannot store general real numbers with arbitrary precision, statisticians usually do not count in such ubiquitous round-off errors when analyzing statistical procedures due to their tiny magnitude (for an example when it may cause some troubles, see \cite{tarkhan2020bigsurvsgd}). However, we need to model and analyze in a finer scale because our space expense is calculated in the unit of bit (rather than \textit{number} of basis function or \textit{number} of float point numbers). Let's be more specific about the round-off error in our estimation setting:

Recall the Sieve-SGD updating rule:
 \begin{equation}
    \hat f_{n+1} = \hat f_n + \gamma_n(Y_n - \hat f_n(X_n))K_{X_n,J_n}
    \end{equation}
and we denote $\hat f_n = \sum_{j=1}^{J_n}\hat \beta_{nj} \psi_j$. The above function update can be reduced to a simutanous update of $J_n$ regression coefficients $\hat\beta_{nj}$ (as stated in Appendix~\ref{app:algorithm}):
\begin{equation}
      \label{eq:betawithoutroundofferror} \hat\beta_{(n+1)j} = \hat \beta_{nj} + \gamma_n(Y_n - \hat f_n(X_n))j^{-2s}\psi_j(X_n)
    \end{equation}
However, because general real numbers cannot be stored in a computer with infinite precision, the right-hand-side quantity of the above update rule cannot be evaluated perfectly. What is calculated and stored in the computer is a round-off version instead:
\begin{equation}
 \text{round}_n\left(\hat \beta_{nj} + \gamma_n(Y_n - \hat f_n(X_n))j^{-2s}\psi_j(X_n)\right)
\end{equation}
Here $\text{round}_n(z)$ rounds/truncates the decimal expansion of $z$ after some digit (which we allow to be a function of $n$). Thus, there is round-off error between the rounded version and the exact version, which we denote as
\begin{equation}
    {}_r\epsilon_{nj} := \hat \beta_{nj} + \gamma_n(Y_n - \hat f_n(X_n))j^{-2s}\psi_j(X_n) - \text{round}_n\left(\hat \beta_{nj} + \gamma_n(Y_n - \hat f_n(X_n))j^{-2s}\psi_j(X_n)\right)
\end{equation}
The round-off error is due, both, to the inexact storage of data $X_n,Y_n$, and potentially inexact evaluation of the intermediate quantities such as $\hat f(X_n), \psi_j(X_n)$. Even in the case when all the above is done without round-off, once we store the coefficients in computer memory, an inevitable precision loss will be introduced because only a finite length of memory is assigned to each $\hat\beta_{nj}$.

In assumption A5 we formalize a sequence of Sieve-SGD estimates contaminated by round-off errors and specify how small we require the errors to be to maintain statistical rate-optimality of our estimator:

\begin{itemize}
    \item[A5](Iteration with round-off error) The recursive relation of Sieve-SGD \eqref{submainform} is given under round-off error. That is
    \begin{equation}
        \hat f_{n+1} = \underbrace{\hat f_n + \gamma_n(Y_n - \hat f_n(X_n))K_{X_n,J_n}}_{\text{exact value we should have assigned to }\hat f_{n+1}} + \underbrace{ \sum_{j=1}^{J_n} {}_r\epsilon_{nj} \psi_j}_{\text{round-off error (in function form)}}
    \end{equation}
    \end{itemize}
    Moreover, for each $j=1,...,J_n$, we assume the round-off error sequence (indexed by $n$) ${}_r\epsilon_{nj}$ is of order $o(n^{-2})$.\\

There is an equivalent way to express our assumption: Let $\hat f_n = \sum_{j=1}^{J_n}\hat \beta_{nj} \psi_j$, we assume the updating of coefficient $\hat\beta_{nj}$ is under round-off error ${}_r\epsilon_{nj}$, i.e.
    \begin{equation}
      \label{eq:betawithroundofferror} \hat\beta_{(n+1)j} = \hat \beta_{nj} + \gamma_n(Y_n - \hat f_n(X_n))j^{-2s}\psi_j(X_n) + {}_r\epsilon_{nj}
    \end{equation}
where the round-off errors ${}_r\epsilon_{nj}\in\mathbb{R},j=1,...,J_n$ are of order $o(n^{-2})$.\\

\textbf{Note 1:} As our readers will see very soon, we propose to assign more digits to store each $\hat\beta_{nj}$ as more data is collected. This will result in round-off errors that decrease as the sample size $n$ increases.

\textbf{Note 2:} We assumed the round-off error of updating each coefficient $\hat\beta_{nj}$ is of order $o(n^{-2})$. There are many other options that people have to model the size of the round-off error: Maybe the upper bound ($o(n^{-2})$) should not only depend on $n$, but also depend on $j$ (for each $j$, ${}_r\epsilon_{nj} = o(a_{nj})$ with some decreasing sequence $a_{nj}$ ); Alternatively, we could have not put assumptions on the \textit{difference} between the exact value and the rounded one, but assume their \textit{ratio} is not too far away from $1$. The treatment we present in this study could be extended: Further discussion of other candidate assumptions is left to future work.

\begin{theorem}
\label{theorem:maintheoremwitheroundoff}
Under the same assumptions as Theorem~\ref{maintheorem}, if we further assume the round-off error satisfy the assumption A5. Then the Sieve-estimator $\bar f_n = \frac{1}{n}\sum_{i=1}^n\hat f_i$, where $\hat f_i$'s are contaminated by the round-off error, is still rate-optimal for estimating $f_{\rho}$.
\end{theorem}
\begin{proof}
The proof of this theorem is basically the same as that of Theorem~\ref{maintheorem}. The only difference now is there is an extra round-off error term in the recursion. Recall in the proof of Theorem~\ref{maintheorem} in Appendix~\ref{section:mainproof} we define the difference between our estimates and $f_{\rho}$ as $\Delta_n$:
\begin{equation}
\begin{aligned}
\Delta_{n}&=\hat f_{n}-f_{\rho}\\ \bar{\Delta}_{n}&=\bar{f}_{n}-f_{\rho}
\end{aligned}
\end{equation}
And we have a recursive formula for $\Delta_n$ under A5:
\begin{equation}
\begin{aligned}
\Delta_0 &= -f_{\rho}\\
\Delta_{n} & =\left(I-\gamma_{n} T_{X_n,J_n}\right) \Delta_{n-1}+\gamma_{n} \Xi_{n}+ {}_r\Xi_n
\end{aligned}
\end{equation}
where 
\begin{equation}
    \begin{aligned}
        T_{X_n,J_n}(f) &= f(X_n)K_{X_n,J_n}\\
        \Xi_{n}&=\left(Y_n-f_{\rho}\left(X_n\right)\right) K_{X_n,J_n}\\
   {}_r\Xi_n &= \sum_{j=1}^{J_n} {}_r\epsilon_{nj} \psi_j
    \end{aligned}
\end{equation}
Here ${}_r\Xi_n$ represents the influence of the round-off error. All we are going to do is show this is a higher order error term. Similar to our previous proofs, we further decompose $\Delta_n$ into two parts: $\Delta_n = \eta_n + \vartheta_n$:
\begin{enumerate}
    \item $\left(\eta_n\right) \text { is defined as } $:
\begin{equation}
\begin{aligned}
   \eta_0& =-f_{\rho}\\
   \eta_n& =\left(I-\gamma_n T_{X_n,J_n}\right) \eta_{n-1}     
\end{aligned} 
\end{equation}
We note $\eta_n$ is exactly the same sequence as in Section~\ref{subsection:separation of error}, which means we already have the optimal bound on it. We do not need to worry about it in the rest of this proof. 
\item The pure noise part \(\left(\vartheta_n\right)\) now has the round-off error noise:
\begin{equation}
    \begin{aligned}      
\vartheta_0&=0\\
\vartheta_n& =\left(I-\gamma_n T_{X_n,J_n}\right) \vartheta_{n-1}+\gamma_n \Xi_{n}+ {}_r\Xi_n
    \end{aligned}
\end{equation}
\end{enumerate}
To control $E[\|\bar \vartheta_n\|_2^2]$, we introduce $\eta_n^k$ for $k\geq 0$ as in Section~\ref{noisesection}, that is
\begin{equation}
\begin{aligned}
    \eta_0^0&=0 \\
    \eta_n^0&=(I-\gamma_n T_{X,J_n}) \eta_{n-1}^0+\gamma_n \Xi_{n}^0
\end{aligned}
\end{equation}
where $\Xi_n^0:=\Xi_n = \left(Y_{n}-f_{\rho}\left(X_{n}\right)\right) K_{X_n,J_n}$. And for each integer $k>0$:
\begin{equation}
\begin{aligned}
    \eta_0^k&=0 \\
    \eta_n^k&=(I-\gamma_n T_{X,J_n}) \eta_{n-1}^k+\gamma_n \Xi_{n}^k
\end{aligned}
\end{equation}
where \(\Xi_n^k=\left(T_{X,J_n}-T_{X_n,J_n}\right) \eta_{n-1}^{k-1}\). And now we need define another sequence to count in the round-off error:
\begin{equation}
\begin{aligned}
    {}_r\eta_0^0&=0 \\
    {}_r\eta_n^0&=(I-\gamma_n T_{X,J_n}) \left({}_r\eta_{n-1}^0\right)+\gamma_n \Xi_{n}^0 + {}_r\Xi_n
\end{aligned}
\end{equation}
Similar to Lemma~\ref{lemma:exactlyzero}, we can verify that
\begin{equation}
    \vartheta_n = {}_r\eta^0_n + \sum_{k=1}^m \eta_{n}^k,\quad\text{for }m\geq n
\end{equation}
Then we use the triangular inequality:
\begin{equation}
\begin{aligned}
      \left(E\|\bar{\vartheta}_{n}\|_2^{2}\right)^{1/2} 
 & \leq \left(E\|{}_r\bar\eta_n^0\|_2^2 \right)^{1/2}+ \sum_{k=1}^m\left(E\|\bar\eta_n^k\|_2^{2}\right)^{1/2} 
 +\left(E\|\bar\eta_n - {}_r\bar\eta_n^0 - \sum_{k=1}^m \bar\eta_n^k\|_2^{2}\right)^{1/2} \\ 
 & \stackrel{(1)}{\leq}  o(n^{-\frac{s}{2s+1}}) + \sum_{k=0}^m\left(E\|\bar\eta_n^k\|_2^{2}\right)^{1/2} 
 +0\quad\text{for }m\geq n\\
 &\leq O(n^{-\frac{s}{2s+1}})
\end{aligned}
\end{equation}
Here is a more detailed calculation of step (1)
\begin{equation}
\begin{aligned}
E\left[\left\|{}_r\bar{\eta}_{n}^0\right\|_2^{2}\right]
&=E\left[\left\langle{}_r\bar{\eta}_n^0, T_{X,J_n} {}_r\bar{\eta}_n^0\right\rangle_K\right] 
= E \left[\left\|T_{X,J_n}^{1/2}{}_r\bar \eta_n^0\right\|_K^2\right]\\
&  = \frac{1}{n^{2}} E\left[\left\|T^{1 / 2}_{X,J_n} \sum_{i=1}^{n} \sum_{j=1}^{i}\left[\prod_{l=j+1}^{i}\left(I-\gamma_l T_{X,J_l}\right)\right]\left( \gamma_j \Xi_j^0+ {}_r\Xi_j\right)\right\|_{K}^{2}\right]\\
& \leq E\left[\left\|\bar{\eta}_{n}^0\right\|_2^{2}\right] + 
\frac{1}{n^2}E\left[\left\|\sum_{i=1}^n \sum_{j=1}^i {}_r\Xi_j\right\|_K^2\right]\\
& \leq E\left[\left\|\bar{\eta}_{n}^0\right\|_2^{2}\right] + 
n\sum_{i=1}^n\|{}_r\Xi_i\|_K^2\\
& \leq E\left[\left\|\bar{\eta}_{n}^0\right\|_2^{2}\right] + n\sum_{i=1}^n {}_r\epsilon_i^2\cdot i
\end{aligned}
\end{equation}
When ${}_r\epsilon_i^2 = o(i^{-4})$, second round-off error term will become higher order, and we have the desired optimal rate.
\end{proof}
Now we specify how we model the decrease of the round-off errors as we use a longer binary sequence to store $\hat\beta_{nj}$
\begin{itemize}
    \item[A6] An $(\alpha+1)\log(n)$-long binary sequence is needed for each of the coefficient $\hat\beta_{nj}$ \eqref{eq:betawithroundofferror} to ensure the round-off errors ${}_r\epsilon_{nj}$ to be of order $o(n^{-\alpha})$.
\end{itemize}
We state this as an assumption, rather than a result because our theoretical roundoff error model allows for potential error to be introduced at multiple places in our update. In the case that everything is calculated exactly, and the only error comes from a final truncation, then it is straightforward to show that A6 holds.

We now give some intuition for the assumption. If we have a $(\alpha+1)\log(n)$-long binary sequence in hand, we can use it to specify $2^{(\alpha+1)\log(n)}\sim n^{\alpha+1}$ numbers. Therefore, for any number $a$ that belongs to a bounded interval $[-M,M]$, we can 1) specify an equally-spaced grid using this binary sequence (there are $\sim n^{\alpha+1}$ grid points); and 2) there must exist a grid point that can approximate any number with an error less than $\sim Mn^{-(\alpha+1)}$. This is the basic intuition that how a $\alpha\log(n)$ length binary sequence should in general give us an $n^{-\alpha}$ accuracy.

Our assumption A6 does not perfectly match with how float point numbers are used in modern computer. The protocal of IEEE 754 standard of float point representation is significantly more complicated and technical \cite{overton2001numerical} than the simplification we present in A6. However, our assumption still captures the main relationship between binary sequence length and round-off error in the sense that every one more digit will give us a doubled accuracy to represent a real number.

Now we state our main result of this section, which can be best understood when compared with Theorem~\ref{theorem:minimummemory}.

\begin{corollary}
\label{corollary:spaceexpense}
Under assumptions A1-A6, there is a $O(\log^3(n)n^{\frac{1}{2s+1}})$-sized version of Sieve-SGD that can achieve the minimax optimal statistical convergence rate. 
\end{corollary}
\begin{proof}
We managed to show in Theorem~\ref{theorem:maintheoremwitheroundoff} that when the round-off error ${}_r\epsilon_{nj}$ is of size $o(n^{-2})$, Sieve-SGD can still achieve the optimal convergence rate. Under A6, it means we need a $3\log(n)$-length binary sequence to specify the coefficients $\hat\beta_{nj}$. Because there are $J_n = O(\log^2(n)n^{\frac{1}{2s+1}})$ coefficients used when sample size is $n$ (Theorem~\ref{maintheorem}), we conclude a $O(\log^3(n)n^{\frac{1}{2s+1}})$-sized version Sieve-SGD can achieve the minimax bound.
\end{proof}
\subsection{Proof of Theorem 6.5 and Discussion}
\label{app:memoryusage}
In this section we will show there is no $b_n$-size estimator with $b_n = o(n^{\frac{1}{2s+1}})$ that can achieve the minimax rate when estimating $f_{\rho}\in W(s,Q,\{\psi_j\})$. We first recall the metric entropy of a Sobolev ellipsoid satisfies (see \cite[Chapter~5]{wainwright2019high}):
\begin{equation}
\label{eq:wsmetricentropy}
    \log\mathcal{N}(\delta;W(s),\|\cdot\|_2)\asymp \left(\frac{1}{\delta}\right)^{1/s}\quad\text{for all suitably small }\delta>0
\end{equation}
We introduce the notation of $\delta$-net of a decoder $D_n:\{0,1\}^{b_n}\rightarrow \mathcal{F}$ (under $\|\cdot\|_2$-norm)
\begin{equation}
  \operatorname{net}(\delta, b_n ;D_n,\mathcal{F}) = \left\{\left.f\in \mathcal{F}\ \right|\ \exists s_n\in \{0,1\}^{b_n}, \text{ such that }\|f-D_n(s_n)\|_2\leq\delta\right\}
\end{equation}
We use the notation $\operatorname{net}(\delta, b_n)$ when it is clear what $D_n,\mathcal{F}$ we are refering to.
\begin{proof}[Proof of Theorem \ref{theorem:minimummemory}]
Let $M_n$ be any $b_n$-sized estimator with $b_n = o\left(n^{\frac{1}{2s+1}}\right)$. We denote its decoder function as $D_n$. We also choose a sequence $c_n$ such that $b_n = o(c_n),c_n = o(n^{\frac{1}{2s+1}})$.\\

Now, we plug $\delta = c_n^{-s}$ into \eqref{eq:wsmetricentropy}
\begin{equation}
    \log_2\mathcal{N}(c_n^{-s};W(s),\|\cdot\|_2)\asymp c_n
\end{equation}
Because there are at most $2^{b_n}$ elements in $D_n(\{0,1\}^{b_n})$ ($D_n$ is a known function), we also note 
\begin{equation}
    2^{b_n}\leq C 2^{c_n}
\end{equation} 
for some constant $C$. So we know $D_n(\{0,1\}^{b_n})$ cannot be a $c_n^{-s}$-cover of $W(s)$ for large enough $n$. In other words, for large $n$,
\begin{equation}
    W(s)\backslash\operatorname{net}(c_n^{-s}, b_n)
\end{equation}
is not an empty set. \\

Then we know
\begin{equation}
    \begin{aligned}
        \sup_{f_{\rho}\in W(s)} E[\|M_n((X_i,Y_i)^n_{i=1})- f_{\rho}\|_2^2]
        & = \sup_{f_{\rho}\in W(s)} E[\|D_n(s_n) - f_{\rho}\|_2^2]\quad\text{where }s_n = E_n((X_i,Y_i)_{i=1}^n)\\
        & \geq \sup_{f_{\rho}\in W(s)\backslash\operatorname{net}(c_n^{-s}, b_n)} E[\|D_n(s_n) - f_{\rho}\|_2^2]\\
        & \geq \sup_{f_{\rho}\in W(s)\backslash\operatorname{net}(c_n^{-s}, b_n)} \inf_{s_n\in \{0,1\}^{b_n}}\|D_n(s_n) - f_{\rho}\|_2^2\\
        & \geq c_n^{-s}
    \end{aligned}
\end{equation}
Because this is true for any $b_n$-sized estimator $M_n$, we have
\begin{equation}
\begin{aligned}
    \inf_{M_n} \sup_{f_{\rho}\in W(s)} E[c_n^{s}\|M_n((X_i,Y_i)^n_{i=1})- f_{\rho}\|_2^2] & \geq 1\\
   \Rightarrow \inf_{M_n} \sup_{f_{\rho}\in W(s)} E[n^{\frac{2s}{2s+1}}\|M_n((X_i,Y_i)^n_{i=1})- f_{\rho}\|_2^2] & \geq n^{\frac{2s}{2s+1}}c_n^{-s}\\
   \Rightarrow \lim_{n\rightarrow \infty}\inf_{M_n} \sup_{f_{\rho}\in W(s)} E[n^{\frac{2s}{2s+1}}\|M_n((X_i,Y_i)^n_{i=1})- f_{\rho}\|_2^2] & \rightarrow \infty\\
\end{aligned}
\end{equation}
where the last line follows from the definition of $c_n$.
\end{proof}
Now we give a little bit of discussion on applying the above argument in parametric learning problem. Suppose we have a very simple model
\begin{equation}
    Y = \theta X + \epsilon
\end{equation}
where $X\in [0,1]$, $\theta\in [0,1]$, $\epsilon$ is uniformly-bounded, centered noise. Using the above argument, we can show that for any $b_n$-sized estimator $M_n$ with $b_n = o(\log n)$, we have
\begin{equation}
    \lim_{n\rightarrow\infty}\inf_{M_n} \sup_{\theta\in [0,1]} E\left[n\|M_n((X_i,Y_i)_{i=1}^n) - \theta\|_2^2\right] = \infty
\end{equation}
This seems to suggest that for any estimator that uses a constant amount of memory, we cannot get a rate-optimal estimator of $\theta$. This feels counter-intuitive because $\hat\theta_n$ is just a number. However we need to emphasize that in our formalization, the memory usage is counted in the unit of bit, i.e. one bit is $O(1)$. In practice we usually give the estimator substantial available memory ($>64$ bits) and do not require extreme estimation accuracy. Thus our theory does not contradict the common belief that ``parametric problem can be solved within $O(1)$ memory", because normally the unit of counting memory is a single stored real-number, rather than a single bit.

\textbf{Note:} Instead of the covering number of $W(s)$, the above result needs the following metric entropy result of the interval $[0,1]$ (see Prop.4.2.12 in \cite{vershynin2018high})
\begin{equation}
    \log\mathcal{N}(\delta;[0,1],\|\cdot\|_2)\asymp \log\left(\frac{1}{\delta}\right)
\end{equation}

\end{appendices}
\bibliographystyle{apalike}
\bibliography{main}

\begin{thebibliography}{}

\bibitem[Arora and Barak, 2009]{arora2009computational}
Arora, S. and Barak, B. (2009).
\newblock {\em Computational complexity: a modern approach}.
\newblock Cambridge University Press.

\bibitem[Babichev and Bach, 2018]{babichev2018constant}
Babichev, D. and Bach, F. (2018).
\newblock Constant step size stochastic gradient descent for probabilistic
  modeling.
\newblock {\em stat}, 1050:21.

\bibitem[Bach and Moulines, 2013]{bach2013non}
Bach, F. and Moulines, E. (2013).
\newblock Non-strongly-convex smooth stochastic approximation with convergence
  rate o (1/n).
\newblock In {\em Advances in neural information processing systems}, pages
  773--781.

\bibitem[Belkin et~al., 2018]{belkin2018reconciling}
Belkin, M., Hsu, D., Ma, S., and Mandal, S. (2018).
\newblock Reconciling modern machine learning and the bias-variance trade-off.
\newblock {\em arXiv preprint arXiv:1812.11118}.

\bibitem[Berlinet and Thomas-Agnan, 2011]{berlinet2011reproducing}
Berlinet, A. and Thomas-Agnan, C. (2011).
\newblock {\em Reproducing kernel Hilbert spaces in probability and
  statistics}.
\newblock Springer Science \& Business Media.

\bibitem[Borkar, 2009]{borkar2009stochastic}
Borkar, V.~S. (2009).
\newblock {\em Stochastic approximation: a dynamical systems viewpoint},
  volume~48.
\newblock Springer.

\bibitem[Bottou, 2010]{bottou2010large}
Bottou, L. (2010).
\newblock Large-scale machine learning with stochastic gradient descent.
\newblock In {\em Proceedings of COMPSTAT'2010}, pages 177--186. Springer.

\bibitem[Cai et~al., 2017]{cai2017computational}
Cai, T.~T., Liang, T., Rakhlin, A., et~al. (2017).
\newblock Computational and statistical boundaries for submatrix localization
  in a large noisy matrix.
\newblock {\em The Annals of Statistics}, 45(4):1403--1430.

\bibitem[Calandriello et~al., 2017]{calandriello2017efficient}
Calandriello, D., Lazaric, A., and Valko, M. (2017).
\newblock Efficient second-order online kernel learning with adaptive
  embedding.
\newblock In {\em Advances in Neural Information Processing Systems}, pages
  6140--6150.

\bibitem[Caponnetto and De~Vito, 2007]{caponnetto2007optimal}
Caponnetto, A. and De~Vito, E. (2007).
\newblock Optimal rates for the regularized least-squares algorithm.
\newblock {\em Foundations of Computational Mathematics}, 7(3):331--368.

\bibitem[Carl, 1981]{CARL1981290}
Carl, B. (1981).
\newblock Entropy numbers, s-numbers, and eigenvalue problems.
\newblock {\em Journal of Functional Analysis}, 41(3):290--306.

\bibitem[Carl and Stephani, 1990]{carl_stephani_1990}
Carl, B. and Stephani, I. (1990).
\newblock {\em Entropy, Compactness and the Approximation of Operators}.
\newblock Cambridge Tracts in Mathematics. Cambridge University Press.

\bibitem[Christmann and Steinwart, 2008]{christmann2008support}
Christmann, A. and Steinwart, I. (2008).
\newblock Support vector machines.

\bibitem[Cohn, 2013]{cohn2013measure}
Cohn, D.~L. (2013).
\newblock {\em Measure theory}.
\newblock Springer.

\bibitem[Cucker and Smale, 2002]{cucker2002mathematical}
Cucker, F. and Smale, S. (2002).
\newblock On the mathematical foundations of learning.
\newblock {\em Bulletin of the American mathematical society}, 39(1):1--49.

\bibitem[Dieuleveut and Bach, 2016]{dieuleveut2016nonparametric}
Dieuleveut, A. and Bach, F. (2016).
\newblock Nonparametric stochastic approximation with large step-sizes.
\newblock {\em The Annals of Statistics}, 44(4):1363--1399.

\bibitem[Duchi, 2014]{duchi2014multiple}
Duchi, J.~C. (2014).
\newblock {\em Multiple Optimality Guarantees in Statistical Learning}.
\newblock PhD thesis, UC Berkeley.

\bibitem[Dũng et~al., 2017]{dung2017hyperbolic}
Dũng, D., Temlyakov, V.~N., and Ullrich, T. (2017).
\newblock Hyperbolic cross approximation.

\bibitem[Eubank and Speckman, 1990]{Eubank1990CurveFB}
Eubank, R. and Speckman, P. (1990).
\newblock Curve fitting by polynomial-trigonometric regression.
\newblock {\em Biometrika}, 77:1--9.

\bibitem[Fasshauer and McCourt, 2015]{fasshauer2015kernel}
Fasshauer, G.~E. and McCourt, M.~J. (2015).
\newblock {\em Kernel-based approximation methods using Matlab}, volume~19.
\newblock World Scientific Publishing Company.

\bibitem[Frostig et~al., 2015]{frostig2015competing}
Frostig, R., Ge, R., Kakade, S.~M., and Sidford, A. (2015).
\newblock Competing with the empirical risk minimizer in a single pass.
\newblock In {\em Conference on learning theory}, pages 728--763.

\bibitem[Gaillard and Gerchinovitz, 2015]{gaillard2015chaining}
Gaillard, P. and Gerchinovitz, S. (2015).
\newblock A chaining algorithm for online nonparametric regression.
\newblock In {\em Conference on Learning Theory}, pages 764--796.

\bibitem[Gao et~al., 2015]{gao2015minimax}
Gao, C., Ma, Z., Ren, Z., Zhou, H.~H., et~al. (2015).
\newblock Minimax estimation in sparse canonical correlation analysis.
\newblock {\em The Annals of Statistics}, 43(5):2168--2197.

\bibitem[Gao et~al., 2017]{gao2017sparse}
Gao, C., Ma, Z., Zhou, H.~H., et~al. (2017).
\newblock Sparse cca: Adaptive estimation and computational barriers.
\newblock {\em The Annals of Statistics}, 45(5):2074--2101.

\bibitem[Geer and van~de Geer, 2000]{geer2000empirical}
Geer, S.~A. and van~de Geer, S. (2000).
\newblock {\em Empirical Processes in M-estimation}, volume~6.
\newblock Cambridge university press.

\bibitem[Gy{\"o}rfi et~al., 2006]{gyorfi2006distribution}
Gy{\"o}rfi, L., Kohler, M., Krzyzak, A., and Walk, H. (2006).
\newblock {\em A distribution-free theory of nonparametric regression}.
\newblock Springer Science \& Business Media.

\bibitem[Hall and Opsomer, 2005]{hall2005theory}
Hall, P. and Opsomer, J.~D. (2005).
\newblock Theory for penalised spline regression.
\newblock {\em Biometrika}, 92(1):105--118.

\bibitem[H{\"a}rdle et~al., 2012]{hardle2012wavelets}
H{\"a}rdle, W., Kerkyacharian, G., Picard, D., and Tsybakov, A. (2012).
\newblock {\em Wavelets, approximation, and statistical applications}, volume
  129.
\newblock Springer Science \& Business Media.

\bibitem[Hastie et~al., 2009]{hastie2009elements}
Hastie, T., Tibshirani, R., and Friedman, J. (2009).
\newblock {\em The elements of statistical learning: data mining, inference,
  and prediction}.
\newblock Springer Science \& Business Media.

\bibitem[Hern{\'a}ndez and Weiss, 1996]{hernandez1996first}
Hern{\'a}ndez, E. and Weiss, G. (1996).
\newblock {\em A first course on wavelets}.
\newblock CRC press.

\bibitem[Kennedy et~al., 2013]{kennedy2013classification}
Kennedy, R.~A., Sadeghi, P., Khalid, Z., and McEwen, J.~D. (2013).
\newblock Classification and construction of closed-form kernels for signal
  representation on the 2-sphere.
\newblock In {\em Wavelets and Sparsity XV}, volume 8858, page 88580M.
  International Society for Optics and Photonics.

\bibitem[Kolmogorov and Tikhomirov, 1959]{kolmogorov1959varepsilon}
Kolmogorov, A.~N. and Tikhomirov, V.~M. (1959).
\newblock $\varepsilon$-entropy and $\varepsilon$-capacity of sets in function
  spaces.
\newblock {\em Uspekhi Matematicheskikh Nauk}, 14(2):3--86.

\bibitem[Koppel et~al., 2019]{koppel2019parsimonious}
Koppel, A., Warnell, G., Stump, E., and Ribeiro, A. (2019).
\newblock Parsimonious online learning with kernels via sparse projections in
  function space.
\newblock {\em The Journal of Machine Learning Research}, 20(1):83--126.

\bibitem[Kushner and Yin, 2003]{kushner2003stochastic}
Kushner, H. and Yin, G.~G. (2003).
\newblock {\em Stochastic approximation and recursive algorithms and
  applications}, volume~35.
\newblock Springer Science \& Business Media.

\bibitem[Liang and Rakhlin, 2018]{liang2018just}
Liang, T. and Rakhlin, A. (2018).
\newblock Just interpolate: Kernel" ridgeless" regression can generalize.
\newblock {\em arXiv preprint arXiv:1808.00387}.

\bibitem[Lin et~al., 2000]{lin2000tensor}
Lin, Y. et~al. (2000).
\newblock Tensor product space anova models.
\newblock {\em The Annals of Statistics}, 28(3):734--755.

\bibitem[Lu et~al., 2016]{lu2016large}
Lu, J., Hoi, S.~C., Wang, J., Zhao, P., and Liu, Z.-Y. (2016).
\newblock Large scale online kernel learning.
\newblock {\em The Journal of Machine Learning Research}, 17(1):1613--1655.

\bibitem[Ma et~al., 2015]{ma2015computational}
Ma, Z., Wu, Y., et~al. (2015).
\newblock Computational barriers in minimax submatrix detection.
\newblock {\em The Annals of Statistics}, 43(3):1089--1116.

\bibitem[Marshall et~al., 1979]{marshall1979inequalities}
Marshall, A.~W., Olkin, I., and Arnold, B.~C. (1979).
\newblock {\em Inequalities: theory of majorization and its applications},
  volume 143.
\newblock Springer.

\bibitem[Marteau-Ferey et~al., 2019a]{marteau2019globally}
Marteau-Ferey, U., Bach, F., and Rudi, A. (2019a).
\newblock Globally convergent newton methods for ill-conditioned generalized
  self-concordant losses.
\newblock In {\em Advances in Neural Information Processing Systems}, pages
  7634--7644.

\bibitem[Marteau-Ferey et~al., 2019b]{marteau2019beyond}
Marteau-Ferey, U., Ostrovskii, D., Bach, F., and Rudi, A. (2019b).
\newblock Beyond least-squares: Fast rates for regularized empirical risk
  minimization through self-concordance.
\newblock {\em arXiv preprint arXiv:1902.03046}.

\bibitem[Michel, 2012]{michel2012lectures}
Michel, V. (2012).
\newblock {\em Lectures on Constructive Approximation: Fourier, Spline, and
  Wavelet Methods on the Real Line, the Sphere, and the Ball}.
\newblock Springer Science \& Business Media.

\bibitem[Mikusinski and Weiss, 2014]{mikusinski2014bochner}
Mikusinski, P. and Weiss, E. (2014).
\newblock The bochner integral.
\newblock {\em arXiv preprint arXiv:1403.5209}.

\bibitem[Nemirovski, 2000]{nemirovski2000topics}
Nemirovski, A. (2000).
\newblock Topics in non-parametric.
\newblock {\em Ecole d’Et{\'e} de Probabilit{\'e}s de Saint-Flour}, 28:85.

\bibitem[Novak and Wozniakowski, 2008]{novak2008tractability}
Novak, E. and Wozniakowski, H. (2008).
\newblock {\em Tractability of multivariate problems. Vol. 1: Linear
  information}.

\bibitem[Overton, 2001]{overton2001numerical}
Overton, M.~L. (2001).
\newblock {\em Numerical computing with IEEE floating point arithmetic}.
\newblock SIAM.

\bibitem[Rakhlin and Sridharan, 2015]{rakhlin2015online}
Rakhlin, A. and Sridharan, K. (2015).
\newblock Online nonparametric regression with general loss functions.
\newblock {\em arXiv preprint arXiv:1501.06598}.

\bibitem[Raskutti et~al., 2009]{raskutti2009lower}
Raskutti, G., Yu, B., and Wainwright, M.~J. (2009).
\newblock Lower bounds on minimax rates for nonparametric regression with
  additive sparsity and smoothness.
\newblock {\em Advances in Neural Information Processing Systems},
  22:1563--1570.

\bibitem[Schmeisser, 2007]{schmeisser2007recent}
Schmeisser, H.-J. (2007).
\newblock Recent developments in the theory of function spaces with dominating
  mixed smoothness.
\newblock {\em Nonlinear Analysis, Function Spaces and Applications}, pages
  145--204.

\bibitem[Shen and Wang, 2010]{shen2010sparse}
Shen, J. and Wang, L.-L. (2010).
\newblock Sparse spectral approximations of high-dimensional problems based on
  hyperbolic cross.
\newblock {\em SIAM Journal on Numerical Analysis}, 48(3):1087--1109.

\bibitem[Shen, 1997]{shen1997methods}
Shen, X. (1997).
\newblock On methods of sieves and penalization.
\newblock {\em The Annals of Statistics}, pages 2555--2591.

\bibitem[Si et~al., 2018]{si2018nonlinear}
Si, S., Kumar, S., and Li, Y. (2018).
\newblock Nonlinear online learning with adaptive nystr$\backslash$"$\{$o$\}$ m
  approximation.
\newblock {\em arXiv preprint arXiv:1802.07887}.

\bibitem[Steinwart and Scovel, 2012]{steinwart2012mercer}
Steinwart, I. and Scovel, C. (2012).
\newblock Mercer’s theorem on general domains: On the interaction between
  measures, kernels, and rkhss.
\newblock {\em Constructive Approximation}, 35(3):363--417.

\bibitem[Stone, 1985]{stone1985additive}
Stone, C.~J. (1985).
\newblock Additive regression and other nonparametric models.
\newblock {\em The annals of Statistics}, pages 689--705.

\bibitem[Sun, 2005]{sun2005mercer}
Sun, H. (2005).
\newblock Mercer theorem for rkhs on noncompact sets.
\newblock {\em Journal of Complexity}, 21(3):337--349.

\bibitem[Tarkhan and Simon, 2020]{tarkhan2020bigsurvsgd}
Tarkhan, A. and Simon, N. (2020).
\newblock Bigsurvsgd: Big survival data analysis via stochastic gradient
  descent.
\newblock {\em arXiv preprint arXiv:2003.00116}.

\bibitem[Tarres and Yao, 2014]{tarres2014online}
Tarres, P. and Yao, Y. (2014).
\newblock Online learning as stochastic approximation of regularization paths:
  Optimality and almost-sure convergence.
\newblock {\em IEEE Transactions on Information Theory}, 60(9):5716--5735.

\bibitem[Tsybakov, 2008]{introtononpara}
Tsybakov, A. (2008).
\newblock {\em Introduction to Nonparametric Estimation}.
\newblock Springer Science \& Business Media.

\bibitem[Vempala, 2005]{vempala2005random}
Vempala, S.~S. (2005).
\newblock {\em The random projection method}, volume~65.
\newblock American Mathematical Soc.

\bibitem[Vershynin, 2018]{vershynin2018high}
Vershynin, R. (2018).
\newblock {\em High-dimensional probability: An introduction with applications
  in data science}, volume~47.
\newblock Cambridge university press.

\bibitem[Wahba, 1990]{wahba1990spline}
Wahba, G. (1990).
\newblock {\em Spline models for observational data}, volume~59.
\newblock Siam.

\bibitem[Wainwright, 2019]{wainwright2019high}
Wainwright, M.~J. (2019).
\newblock {\em High-dimensional statistics: A non-asymptotic viewpoint},
  volume~48.
\newblock Cambridge University Press.

\bibitem[Wang et~al., 2016]{wang2016statistical}
Wang, T., Berthet, Q., Samworth, R.~J., et~al. (2016).
\newblock Statistical and computational trade-offs in estimation of sparse
  principal components.
\newblock {\em The Annals of Statistics}, 44(5):1896--1930.

\bibitem[Wood, 2017]{wood2017generalized}
Wood, S.~N. (2017).
\newblock {\em Generalized additive models: an introduction with R}.
\newblock CRC press.

\bibitem[Ying and Pontil, 2008]{ying2008online}
Ying, Y. and Pontil, M. (2008).
\newblock Online gradient descent learning algorithms.
\newblock {\em Foundations of Computational Mathematics}, 8(5):561--596.

\bibitem[Yuan et~al., 2010]{yuan2010reproducing}
Yuan, M., Cai, T.~T., et~al. (2010).
\newblock A reproducing kernel hilbert space approach to functional linear
  regression.
\newblock {\em The Annals of Statistics}, 38(6):3412--3444.

\bibitem[Yuan and Zhou, 2016]{yuan2016minimax}
Yuan, M. and Zhou, D.-X. (2016).
\newblock Minimax optimal rates of estimation in high dimensional additive
  models.
\newblock {\em The Annals of Statistics}, 44(6):2564--2593.

\bibitem[Zhang and Simon, 2021]{zhang2021online}
Zhang, T. and Simon, N. (2021).
\newblock An online projection estimator for nonparametric regression in
  reproducing kernel hilbert spaces.
\newblock {\em arXiv preprint arXiv:2104.00780}.

\bibitem[Zhang et~al., 2014]{zhang2014lower}
Zhang, Y., Wainwright, M.~J., and Jordan, M.~I. (2014).
\newblock Lower bounds on the performance of polynomial-time algorithms for
  sparse linear regression.
\newblock In {\em Conference on Learning Theory}, pages 921--948.

\end{thebibliography}

\end{document}